\documentclass{amsart}
\usepackage[english]{babel}
\usepackage{amsmath}
\usepackage{amssymb}
\usepackage{amsthm}
\usepackage{mathrsfs}
\usepackage{hyperref}
\usepackage{physics}
\usepackage[pdftex]{graphicx}
\usepackage{cleveref} 

\newtheorem{thm}{Theorem}[section]
\newtheorem{lem}[thm]{Lemma}
\newtheorem{prop}[thm]{Proposition}
\newtheorem{defn}[thm]{Definition}

\newtheorem{coro}[thm]{Corollary}

\newtheorem{rmk}[thm]{Remark}

\allowdisplaybreaks

\DeclareMathOperator{\diam}{diam}
\DeclareMathOperator{\supp}{supp}

\DeclareMathOperator{\Span}{span}

\DeclareMathOperator{\spec}{spec}

\numberwithin{equation}{section}

\title[Minimal submanifolds with multiple isolated singularities]{Minimal submanifolds with multiple isolated singularities}
\author{Bryan Dimler}
\address{Department of Mathematics, UC Irvine}
\email{\tt bdimler@uci.edu}

\begin{document}

\begin{abstract}
    We extend Smale's singular bridge principle [\emph{Ann. of Math.} \textbf{130} (1989), 603-642] for $n$-dimensional strictly stable minimal cones in $\mathbb{R}^{n+1}$ $(n \geq 7$) to arbitrary codimension and each $n \geq 3$. We then apply the procedure to copies of the Lawson-Osserman cone to produce a four dimensional minimal graph in $\mathbb{R}^7$ with any finite number of isolated singularities. 
\end{abstract}

\maketitle

\section{Introduction}

Introduced by L\'evy  in 1948, the bridge principle is the idea that it should be possible to glue minimal submanifolds by a thin bridge to produce an ``approximately minimal" submanifold, called the \emph{approximate solution}, and apply a small perturbation so that it becomes minimal (e.g. see \cite{C,L, NS1, W1}). Their stated bridge principle was quite general, but they gave a heuristic argument instead of a rigorous proof (\cite{W1}). In \cite{C}, Courant outlined a potential method of proof, and Kruskal published a proof in this spirit in the 1950's (\cite{K}). However, Kruskal's proof was later found to be incomplete (\cite{JCC}). It was not until the 1980's that the first complete proofs of the bridge principle for stable (i.e. second variation of area nonnegative) hypersurfaces were published (\cite{LM,MY}).

In 1987, Smale proved a bridge principle for minimal submanifolds in Euclidean space of arbitrary dimension and codimension by solving a fixed point problem for the stability operator $L$ acting on normal vector fields on the approximate solution (\cite{NS1}). The theorem is very broad, requiring only that the minimal submanifolds to be glued have no Jacobi fields vanishing on the boundary, and was the first published example of the bridge principle for unstable minimal submanifolds. In 1989, Smale applied their bridge principle combined with the methods in \cite{CHS} to strictly stable (i.e. $L$ positive definite) minimal hypercones to produce the first examples of (stable) minimal hypersurfaces with multiple isolated singularities (\cite{NS2}). Shortly after, White proved bridge principles for general strictly stable, unstable, and singular minimal submanifolds with arbitrary dimension and codimension using geometric measure theory (\cite{W1, W2}). However, in the singular case White required the minimal submanifolds to be glued to be uniquely area-minimizing with respect to their boundaries (e.g. as an integral current) in some open subset of the ambient space instead of being strictly stable. Aside from Smale's and White's work, the bridge principle has been extended to other settings such as constant and positive Gauss curvature surfaces (\cite{Haus}), harmonic maps between manifolds (\cite{LWD1, Mou}), and harmonic diffeomorphisms between manifolds (\cite{LWD2}). For an account of the early history of the bridge principle and its applications, see \cite{W1}. 

Overall, White's bridge principles are more general than Smale's since there is more freedom in constructing the bridges, the singularities do not need to be isolated, and they extend to arbitrary smooth ambient manifolds and critical points of smooth parametric elliptic functionals. On the other hand, in the singular case White's method does not ensure singularity preservation --- it only guarantees the existence of nearby (possibly smooth) area-minimizing minimal submanifolds with boundary coinciding with the original boundaries away from the bridges. As a consequence, Smale's singular bridge principle is well suited for constructing new examples of singular minimal submanifolds. See Sec. 7 for additional details.

In this paper, we apply Smale's singular bridge principle to produce families of strictly stable minimal graphs in Euclidean space having any finite number of isolated singularities and prescribed rates of decay to their tangent cones at their singularities. We begin by generalizing the singular bridge principle to arbitrary codimension and each dimension $n \geq 3$ (\Cref{mainthm}). Since 2-dimensional minimal cones with an isolated singularity must intersect the unit sphere along a great circle, they are flat. Hence, the dimension $n = 3$ is optimal. The main difficulties to overcome in the generalization are as follows: 
\begin{enumerate}
    \item When the dimension of the cones is $n = 3,4,5$, the mean curvature of the bridges joining them must be updated so that their $L^p$ mean curvature ($p \geq 2$) is at most $O(\epsilon^{\frac{5}{p}})$ instead of the assumed estimate $O(\epsilon^{\frac{n-1}{p}})$ when $n \geq 6$, where $\epsilon > 0$ is the width of the bridges (Sec. 8.1). This allows us to prove the estimates needed to solve for the perturbed minimal submanifold as a fixed point problem for $L$ on normal vector fields (Sec. 4). 
    \item There is no canonical normal frame for the perturbed submanifold, so a suitable frame must be developed to show it is strictly stable by comparing the spectrum of $L$ on the perturbed and approximate solutions. 
\end{enumerate}

When $n \geq 6$, preliminary bridges (i.e. the \emph{pre-bridges}) are constructed as ruled surfaces having $(n-1)$-disks of width $\epsilon$ as cross-sections that meet the boundaries of the cones tangentially (though not smoothly) at a point. We then smoothly press the pre-bridges down onto the cones in a small region (i.e. the \emph{patching region}) near their points of contact to obtain new bridges, called the \emph{$\epsilon$-bridges}, and their associated approximate solutions (Sec. 8.1.1). Since the volume of the bridges is $O(\epsilon^{n-1})$, we have significant freedom in how we attach them to the cones since we only need that their mean curvature is pointwise bounded independent of $\epsilon$ to obtain the desired $L^p$ estimates. 

For dimensions $n = 3,4,5$, we begin by carefully extending and flattening the cones onto the tangent $n$-planes lying along the radial paths extending from the vertices of the cones to the points on the extensions at maximum distance from their respective vertices. The extensions are flattened onto the $n$-planes so that their mean curvature is pointwise of order $O(\epsilon^2)$ and their volume is of order $O(\epsilon^{n-1})$ (Sec. 8.1.2). This is enough when $n = 3$, since we can position the extended cones so that they lie tangent to the same $3$-plane and connect the extensions by planar paths with smooth boundary. 

Since an $n$-dimensional Lipschitz minimal graph is smooth when $n \leq 3$ (\cite{FC}), we need a more general construction when $n = 4,5$ as we intend to build graphical bridges for each $n \geq 4$. When $n = 4,5$, we join the extended cones using the same construction for the pre-bridges in the $n \geq 6$ case, where this time the pre-bridges are built so that they smoothly attach to the tangent $n$-planes along their ends; hence, the extended cones. By perturbing the pre-bridges so that their mean curvature is zero along their center curves, we obtain new bridges having mean curvature that is pointwise $O(\epsilon)$ by a Taylor expansion. In this case, the $\epsilon$-bridges consist of two regions: the extended part of the cones and the perturbed pre-bridges. Combining the mean curvature and volume bounds on both regions gives (1) (Sec. 8.1.2). 

Point (2) is resolved by constructing an orthonormal normal frame around each point on the perturbed submanifold whose elements are small perturbations of an orthonormal normal frame on the approximate solution (Sec. 5). The perturbed frame is written as $n_\alpha + \xi_\alpha$, where $n_1,\ldots,n_{m+1}$ is the frame on the approximate solution and the $\xi_\alpha$ are vector fields in the ambient space which tend to zero in $\epsilon$ in suitable norms. This allows us to directly compare the spectrum of $L$ on the approximate solution with its spectrum on the perturbed submanifold. 

After generalizing Smale's singular bridge principle, we apply it to copies of the Lawson-Osserman cone (see Sec. 2.2.2) to produce a 4-dimensional graphical strictly stable minimal submanifold in $\mathbb{R}^7$ having any finite number of isolated singularities (Theorem \ref{gsolution}). This is achieved by first showing that, if the cones to be glued are graphical and the approximate solution is constructed so as to remain graphical, then the perturbed submanifold is a graph for small $\epsilon$ (Proposition \ref{graphical}). We then demonstrate that graphical approximate solutions can always be built from graphical cones when $n \geq 4$ using the constructions described above as our base (Sec. 6.1.1 and 6.1.2). Since the Lawson-Osserman cone is a strictly stable graph (\cite{DL, LO}), the result is immediate. To the author's knowledge, this is the first example of a minimal graph with multiple isolated singularities. In addition, this construction is sharp (\cite{D}; see also Sec. 2.4). 

\vspace{-2.8pt}

\section*{Acknowledgments}
The author would like to express their sincere gratitude to Connor Mooney and Rick Schoen for bringing this problem to their attention and for many insightful discussions on the subject of this paper. In addition, the author thanks Joshua Jordan for several helpful conversations related to this paper. The author was partially supported by NSF grant DMS-2143668 along with C. Mooney's Sloan and UCI Chancellor's Fellowships.

\vspace{-2.8pt}

\section{Preliminaries}
\subsection{Notation and Function Spaces}
Let $(M,g)$ be an $n$-dimensional smooth embedded Riemannian submanifold of $\mathbb{R}^{n+m+1}$ (with or without boundary) where $g$ is the induced metric. We will always assume $M$ is oriented. When $M$ has boundary $\partial M$, we identify $M$ with its interior. We will denote the normal bundle and tangent bundles of $M$ by $NM$ and $TM$, respectively, and will denote the fibers at $x \in M$ for each bundle by $N_xM$ and $T_xM$. As a Riemannian submanifold of $\mathbb{R}^{n+m+1}$, the induced metrics on $NM$ and $TM$ are the Euclidean dot product restricted to each bundle when they are viewed as elements of $\mathbb{R}^{n+m+1}$.

We will denote the connections on $NM$ and $TM$ by $\nabla$ and $\nabla^\prime$, respectively. Precisely, if $X,Y$ are $C^1$ sections of $TM$ and $U$ is a $C^1$ section of $NM$, then
$$
    \nabla_X U = (D_X U)^\perp \text{ and } \nabla_X^\prime Y = (D_X Y)^\top,
$$
where $D$ is the directional derivative on Euclidean space. In local coordinates $x^1, \ldots, x^n$, we have
\begin{equation*}
    D_{\pdv{x^i}} U := \pdv{U}{x^i} = U_{x^i} \text{ and } \nabla_i U := \nabla_{\pdv{x^i}} U.
\end{equation*}

For each $k = 1,2, \ldots$, we will write $\nabla^k$ for the $k$-th order covariant derivative
$$
    \nabla^k: C^\infty(NM) \overset{\nabla}{\rightarrow} C^{\infty}(T^*M \otimes NM)\overset{\nabla}{\rightarrow} \cdots \overset{\nabla}{\rightarrow} C^\infty(T^*M^{\otimes k} \otimes NM),
$$
where $C^\infty(T^*M^{\otimes k} \otimes NM)$ is the space of $C^\infty$ sections of $T^*M^{\otimes k} \otimes NM$. We can similarly define $\nabla^k$ on $C^k$ sections of $NM$. For each $k = 0 ,1,2, \ldots$, we define the space $C^k(NM)$ to be the space of sections of $NM$ whose components relative to any smooth local orthonormal frame for $NM$ are $C^k$ and satisfy
\begin{equation}
    \norm{U}_{C^k(NM)} := \sum_{j = 0}^k \sup_{x \in M} |\nabla^j U(x)| < \infty. \label{Cknorm}
\end{equation}
The spaces $C^k(NM)$ are Banach spaces with norm defined by \eqref{Cknorm}. When $M$ is compact with non-empty boundary, we will denote by $C_0^k(NM) \subset C^k(NM)$ the space of $C^k$ normal sections that vanish on $\partial M$. For each $k = 0,1,2,\ldots,$ the H\"older space $C^{k,\gamma}(NM)$ is the space of sections $U$ of $NM$ whose components relative to any smooth local orthonormal frame are $C^{k,\gamma}$ and satisfy
\begin{equation}
    \norm{U}_{C^{k,\gamma}(NM)} := \norm{U}_{C^k(NM)} + |U|_{k,\gamma,M} < \infty. \label{ckanorm}
\end{equation}
Here, the H\"older semi-norm $|\cdot|_{\gamma,M} := |\cdot|_{0,\gamma, M}$ is defined by
\begin{equation}
    |U|_{\gamma, M} := \sup_{\substack{x,y \in M \\ x \neq y}} \Bigg\{\frac{|U(x) - U(y)|}{d(x,y)^\gamma} \Bigg\}, \label{seminorm}
\end{equation}
where $d(x,y)$ is the Riemannian distance between $x,y \in M$ (see \cite{A}). The H\"older semi-norms $|U|_{k,\gamma, M}$ for $k = 1,2, \ldots$ are defined similarly, with $\nabla^k U$ replacing $U$ in \eqref{seminorm}. However, in this case $|\nabla^k U(x) - \nabla^kU(y)|$ should be understood via parallel translation (e.g. see \cite{N}). The spaces $C^{k,\gamma}(NM)$ are Banach spaces when given the norm \eqref{ckanorm}. We define $C_0^{k,\gamma}(NM)$ similarly to the $C^k$ case.

Let $\mathcal{H}^n$ represent the $n$-dimensional Hausdorff measure on $M$. For $1 \leq p < \infty$, the space $L^p(NM)$ will denote the class of Borel measurable sections of $NM$ such that 
$$
    \norm{U}_{L^p(NM)} := \Big( \int_M |U|^p \, d\mathcal{H}^n\Big)^{\frac{1}{p}} < \infty.
$$
For each $k = 0 ,1,2, \ldots$, let $\mathcal{C}_2^k(NM)$ be the subset of $C^\infty(NM)$ for which the norm
\begin{equation}
    \norm{U}_{H^k(NM)} := \sum_{j = 0}^k \Big(\int_M |\nabla^j U|^2 \, d\mathcal{H}^n\Big)^{\frac{1}{2}}  \label{snorm}
\end{equation}
is finite. The Sobolev space $H^k(NM)$ is the completion of $\mathcal{C}_2^k(NM)$ in $L^2$ for $\norm{\cdot}_{H^k}$. It is a Hilbert space when given the inner product
\begin{equation}
    \langle U, V \rangle_{H^k} := \sum_{j = 0}^k \int_M \langle \nabla^j U , \nabla^j V \rangle\, d\mathcal{H}^n. \label{sprod}
\end{equation}
If $M$ is compact with boundary $\partial M$, then the subspace $H_0^k(NM) \subset H^k(NM)$ is the closure of $C_0^\infty(NM)$ in $H^k(NM)$. When there is no confusion, we will drop the $d\mathcal{H}^n$ term in the integral expressions and will write $C^k(M)$, $L^p(M)$, etc., for the spaces $C^k(NM)$, $L^p(NM)$, and so on.

\subsubsection{Normal Laplacian and Second Fundamental Form}
Let $U$ be a $C^2$ section of $NM$ and let $e_1,\ldots,e_n$ be a smooth local orthonormal frame in a neighborhood of $x \in M$. Then
\begin{equation}
    \Delta U := \tr \nabla^2 U = \sum_{i = 1}^n (\nabla_{e_i} \nabla_{e_i} U - \nabla_{\nabla_{e_i}^\prime e_i} U). \label{laplace}
\end{equation}
The operator $\Delta$ is called the \emph{normal Laplacian}. We will also use $\nabla^\prime$ and $\Delta^\prime$ to denote the gradient and Laplace-Beltrami operator acting on $C^1$ and $C^2$ functions on $M$, respectively.

The \emph{second fundamental form} for $M$ is the symmetric $NM$-valued bilinear form defined on $X,Y \in TM$ by 
$$
    A(X,Y) := (D_{X} Y)^\perp.
$$
We will write $|A|^2$ for the length squared of the second fundamental form. The \emph{mean curvature vector} $H$ for $M$ is the trace of the second fundamental form, and $M$ is a \emph{minimal submanifold} of $\mathbb{R}^{n+m+1}$ if $H \equiv 0$.

\subsubsection{Coordinate Expressions}
Fix $x_0 \in M$, let $x^1, \ldots, x^n$ be local coordinates near $x_0$, and let $n_1, \ldots, n_{m+1}$ be a smooth local orthonormal frame for $NM$ near $x_0$. Let $U \in C^{2}(M)$ and write $U = u^\alpha n_\alpha$ for $C^2$ functions $u^\alpha$ ($\alpha = 1,\ldots, m+1$). For each $i = 1,\ldots, n$ and each $\alpha,\beta = 1,\ldots, m+1$, let $B_{\beta i}^\alpha$ be the connection terms 
$$
    B_{\beta i}^\alpha := (\nabla_i n_\beta) \cdot n_\alpha = (n_\beta)_{x^i} \cdot n_\alpha
$$  
By standard tensor calculus, we have
\begin{align*}
    \nabla U &=  u_{;j}^\alpha \,dx^j \otimes n_\alpha \text{ and }  \\
    \nabla^2 U &=  u_{;ij}^\alpha \, dx^i \otimes dx^j \otimes n_\alpha,
\end{align*}
where we have used Einstein summation and
\begin{align}
    u_{;j}^\alpha &:= u_{x^j}^\alpha + u^\beta B_{\beta j}^\alpha \label{coord1}\\
    u_{;ij}^\alpha &:= u_{x^i x^j}^\alpha + u_{x^j}^\beta B_{\beta i}^\alpha + u_{x^i}^\beta B_{\beta j}^\alpha - \Gamma_{ij}^k u_{x^k}^\alpha + \Big(B_{\beta i, j}^\alpha+ B_{\beta j}^\sigma B_{\sigma i}^\alpha 
 -\Gamma_{ij}^k B_{\beta k}^\alpha \Big)u^\beta,\label{coord2}
\end{align}
$B_{\beta i,j}^\alpha := \pdv{B_{\beta i}^\alpha}{x^j}$, and the $\Gamma_{ij}^k$ are the Christoffel symbols for $\nabla^\prime$ in the given coordinates. In addition, 
\begin{align*}
    |\nabla U|^2 &:= g^{ij} u_{;i}^\alpha u_{;j}^\alpha \\
    |\nabla^2 U|^2 &:= g^{ik}g^{jl} u_{;ij}^\alpha u_{;kl}^\alpha,
\end{align*}
where $g_{ij} := \pdv{x^i} \cdot \pdv{x^j}$, $g := (g_{ij})$, and $g^{-1} := (g^{ij})$. Going forward, Greek indices in summation will always index $1,2,\ldots, m+1$ while $i,j,k$, etc. will always index $1,2, \ldots, n$.

\subsection{Stability}
We define stability of a minimal submanifold in $\mathbb{R}^{n+m+1}$, introduce the stability operator, and provide a brief discussion of stable minimal cones.

\subsubsection{Stability Operator}
Let $M$ be an $n$-dimensional smooth compact minimal submanifold of $\mathbb{R}^{n+m+1}$ with the induced metric and boundary $\partial M$. For a normal section $U \in C_0^2(M)$, the second variation of area if we vary in the direction $U$ is
\begin{equation} 
   Q(U,U) := \int_M \big(|\nabla U|^2 - \langle U , \tilde{A}(U)\rangle ) = - \int_M \langle U ,\Delta U + \tilde{A}(U) \rangle, \label{weak}
\end{equation}
where $\tilde{A}$ is \emph{Simons' operator} on $M$. If $e_1,\ldots, e_n$ is a smooth local orthonormal frame for $TM$, then 
\begin{equation*}
    \tilde{A}(U) = \sum_{i,j = 1}^n\langle A(e_i,e_j) , U \rangle A(e_i,e_j).
\end{equation*}
The operator $L := \Delta + \tilde{A}$ is called the \emph{stability operator} and is defined on $C^2$ normal sections of $M$. It is a (formally) $L^2$ self-adjoint strongly elliptic operator on normal sections in $C_0^\infty(M)$ (\cite{Si}, Proposition 1.2.3; \cite{Sm}). $M$ is said to be \emph{stable} if $Q(U,U) \geq 0$ for all $U \in C_0^\infty(M)$. If $M$ is not compact, we require that $U$ has compact support in $M$ and vanishes on $\partial M$. A minimal submanifold $M$ is \emph{strictly stable} if $L$ is positive definite when identified with $-L$.

We will need to write $L$ in local coordinates. Let $(\Omega, \psi)$ be a smooth local parameterization (i.e. an injective immersion) for $M$ and let $n_1, \ldots, n_{m+1}$ be a smooth local orthonormal frame for $NM\lvert_{\Omega}$. If $U \in C^2(M)$, we can write $U = u^\alpha n_\alpha$ on $\Omega$. In these coordinates, Simons' operator is 
\begin{equation}
    \tilde{A}(U) = (g^{kj}g^{il} A_{kl}^\beta A_{ij}^\alpha u^\beta )n_\alpha, \label{csimons}
\end{equation}
where $A_{ij}^\alpha := \psi_{x^i x^j} \cdot n_\alpha$ and we have summed over repeated indices. Using \eqref{coord1}, \eqref{coord2}, and \eqref{csimons}, for any normal section $U \in C^2(M)$ we can write
\begin{equation}
   LU = (g^{ij}u_{x^ix^j}^\alpha  +  b_{\beta}^{i\alpha} u_{x^i}^\beta + c_{\beta}^\alpha u^\beta)n_\alpha \label{genstab}
\end{equation}
in local coordinates with respect to any local orthonormal frame $n_1,\ldots, n_{m+1}$. 

\subsubsection{Stable Minimal Cones}
Let $\mathbb{S}^{n+m}(p)$ be the unit $n+m$ sphere in $\mathbb{R}^{n+m+1}$ with center at $p$, let $\Sigma$ be a smooth compact closed embedded ($n-1$)-dimensional minimal submanifold of $\mathbb{S}^{n+m}(p)$, and let $C$ be the cone over $\Sigma$ with vertex at $p$ intersected with the closed unit ball $B_1^{n+m+1}(p)$ in $\mathbb{R}^{n+m+1}$. Then $C$ is an $n$-dimensional minimal submanifold of $\mathbb{R}^{n+m+1}$, $\partial C = \Sigma$, and $C$ has an isolated singularity at $p$ provided $\Sigma$ is not a totally geodesic $(n-1)$-sphere. In the latter case, $C$ is flat. Note that, since one dimensional minimal submanifolds of a sphere are pieces of great circles, it is appropriate to assume $n\geq 3$. We will need to work in polar coordinates on $C$:
\begin{equation*}
    x = r\omega \text{ for } x \in C, \text{ } \omega = \theta - p \text{ for } \theta \in \Sigma, \text{ and } r = |x-p|.
\end{equation*}
In these coordinates, $C$ is identified with $C \setminus \{p\}$ and we can write
$$
    C := \{p + t\omega :  t \in (0,1]\}. 
$$ 
The cone $C$ is embedded with this identification.

Defining local orthonormal frames $e_1,\ldots, e_{n-1}$ and $n_1,\ldots,n_{m+1}$ for $T\Sigma$ and $N\Sigma$, respectively, and parallel translating them down the cone, one can compute
\begin{equation}
    A_C(e_i,e_j) = r^{-1}A_\Sigma(e_i,e_j) \text{ for each } i,j = 1,\ldots, n-1, \label{scale}
\end{equation}
where $A_C$ and $A_\Sigma$ are the second fundamental forms for $C$ and $\Sigma$. Using that $\Sigma$ is minimal in $\mathbb{S}^{n+m}(p)$ and the definition of the $e_i$ via parallel transport, we obtain the polar coordinate expression
\begin{equation}
    L_C = \pdv[2]{r} + \frac{n-1}{r} \pdv{r}  +\frac{1}{r^2} L_\Sigma, \label{pstab} 
\end{equation}
where $L_\Sigma := \Delta_{\Sigma} + \tilde{A}_\Sigma$. The operator $L_{\Sigma}$ is strongly elliptic on $C^\infty(\Sigma)$ and self-adjoint with respect to $L^2(\Sigma)$ so the eigenvalues of $L_\Sigma$ are real, countable, and form an increasing sequence
$$
    \mu_1 \leq \mu_2 \leq \mu_3 \leq \cdots 
$$
with $\mu_i \rightarrow \infty$ as $i \rightarrow \infty$. We can define (strict) stability of a minimal cone in $\mathbb{R}^{n+m+1}$ as follows (see \cite{CHS, DL, Si}):

\begin{defn}\label{stabcondition}
    An $n$-dimensional minimal cone $C$ in $\mathbb{R}^{n+m+1}$ is said to be stable (strictly stable) if
    \begin{equation*}
        \frac{(n-2)^2}{4} + \mu_1 \geq 0 \text{ (likewise, } >0). 
    \end{equation*}
\end{defn}
We will write $d_0(C)$ for the quantity on the left-hand side of the inequality above. In this paper, we will be interested in strictly stable minimal cones (i.e. $d_0(C) > 0)$. In codimension one, Lawson's cones (e.g. see \cite{Law, O}) are strictly stable when $n \geq 7$ (\cite{CHS, DL, NS2}). Recently, the author and J. Lee showed that special Langrangian cones in $\mathbb{R}^{2n}$ are strictly stable when $n \geq 5$, and when $n=4$ under the additional assumption that the link is simply connected (\cite{DL}, Theorem 5.4). All coassociative cones in $\mathbb{R}^7$ were proved to be strictly stable (\cite{DL}, Theorem 6.4). 

Of particular interest to the present paper is the \emph{Lawson-Osserman cone} (see \cite{LO}). Let $\eta: \mathbb{S}^3 \rightarrow \mathbb{S}^2$ denote the Hopf map:
\begin{equation}
    \eta(z_1,z_2) = (2\overline{z_1}z_2, |z_1|^2 - |z_2|^2) \label{hopf}
\end{equation}
for $(z_1,z_2) \in \mathbb{C}^2$ with $|z_1|^2 + |z_2|^2 = 1$. After making the identifications $\mathbb{C}^2 \simeq \mathbb{R}^4$ and $\mathbb{C} \times \mathbb{R}\simeq \mathbb{R}^3$, one can show that the graph of $\frac{\sqrt{5}}{2}|x|\eta\Big(\frac{x}{|x|}\Big)$ is a four dimensional minimal cone in $\mathbb{R}^7$ using the symmetries of the Hopf map. In addition, it has an isolated singularity at the origin. The link $\Sigma$ is given by 
\begin{equation}
    \Sigma:= \Big\{\Big(\frac{2}{3}x, \frac{\sqrt{5}}{3}\eta(x)\Big) : |x| = 1\Big\}, \label{LOlink}
\end{equation}
and is a closed three dimensional minimal submanifold of $\mathbb{S}^6$. In \cite{HL}, it was shown that the Lawson-Osserman cone is coassociative; hence, strictly stable by the discussion above. In particular, it is an example of a graphical strictly stable minimal cone with an isolated singularity. 

\subsection{Approximate Solutions}
As in \cite{NS2}, strictly stable minimal cones will be the building blocks in the construction of the approximate solutions $M^\epsilon$. We refer the reader to \cite{NS1} for the construction of $M^\epsilon$ when $n \geq 6$, though a brief description of the construction has been included in Sec. 8.1.1 of the appendix for convenience. The constructions when $n = 3,4,5$ are original, but are inspired by those in \cite{NS1} when $n = 2$ (see Sec. 8.1.2). However, our constructions allow us to obtain stronger $L^p$ mean curvature estimates for the approximate solution without risking blow-up of the coefficients of $L$ on the patching regions as $\epsilon \rightarrow 0$.

We must specify what is meant by an \emph{$\epsilon$-bridge}. Let $C_1$ and $C_2$ be strictly stable cones over $\Sigma_1$ and $\Sigma_2$ and let $q_i \in \Sigma_i = \partial C_i$ for each $i = 1,2$. Going forward, $B_r^n$ will always denote a closed $n$-ball of radius $r$ in Euclidean space and $B_r$ will always denote a geodesic ball in $M^\epsilon$ with radius $r$. In addition, $D_r$ will represent a geodesic disk in $\Sigma_i$ with radius $r$. 

\begin{defn}
An $\epsilon$-bridge $\Gamma_{q_1,q_2}(\epsilon)$ from $q_1$ to $q_2$ is a one-parameter family (defined for small $\epsilon > 0$) of embedded $n$-dimensional strips (i.e. a diffeomorphic copy of $B_1^{n-1}(0) \times [0,1]$) whose ends are smoothly attached to $C_1$ and $C_2$ at the geodesic disks $D_{5\epsilon}(q_1)$ and $D_{5\epsilon}(q_2)$ in $\Sigma_1$ and $\Sigma_2$, respectively. 
\end{defn}

For each $n \geq 3$, we can assume (see Sec. 8.1.1 and Remark \ref{nbdds})
\begin{equation}
    \begin{cases}
        c_0^{-1}\epsilon \leq \diam \Gamma(\epsilon) \leq c_0 \epsilon \\
        \norm{A_{\Gamma(\epsilon)}}_{C^0(\Gamma(\epsilon))} + \norm{\nabla A_{\Gamma(\epsilon)}}_{C^0(\Gamma(\epsilon))} \leq c_0 \\
        \norm{A_{\partial \Gamma(\epsilon)}}_{C^0(\partial \Gamma(\epsilon))} \leq c_0\epsilon^{-1},
    \end{cases} \label{bridge1}
\end{equation}
where $c_0$ is a constant independent of $\epsilon$ and $\nabla$ is the covariant derivative relative to the induced metric, while $A_{\Gamma(\epsilon)}$ and $A_{\partial \Gamma(\epsilon)}$ are the second fundamental forms of $\Gamma(\epsilon):= \Gamma_{q_1,q_2}(\epsilon)$ and $\partial \Gamma(\epsilon)$, respectively.

We now let $C_1,\ldots, C_N$ ($N \geq 2$) be a collection of strictly stable cones with vertices at $p_i$  ($i = 1,\ldots, N$) and smooth links $\Sigma_i \subset \mathbb{S}^{n+m}(p_i)$. Let $\Gamma_1(\epsilon), \ldots, \Gamma_I(\epsilon)$ be a collection of $\epsilon$-bridges, where $\Gamma_l(\epsilon) := \Gamma_{q_{2l-1}, q_{2l}}(\epsilon)$ is an $\epsilon$-bridge from $q_{2l-1}$ to $q_{2l}$, $l = 1,\ldots, I$, and each $q_l$ is in some $\Sigma_i$. Assume that the cones $C_i$ have been joined by the $\Gamma_l(\epsilon)$ so that the new submanifold is connected. Then $I \geq N-1$ and the composition is a smooth submanifold with boundary (not necessarly minimal) in $\mathbb{R}^{n+m+1}$ away from its isolated singularities at the $p_i$. Denote this submanifold by $M^\epsilon$. Then
$$
    \partial M^\epsilon := \Big(\bigcup_{l = 1}^I \partial \Gamma_l(\epsilon) \cup \bigcup_{i = 1}^N \Sigma_i \Big) \setminus \bigcup_{k = 1}^{2I}  D_{5\epsilon}(q_k),
$$
and the points of intersection of the $\Sigma_i$ with the interior of the bridges $\Gamma_l(\epsilon)$ are in the interior of $M^\epsilon$. Furthermore, due to the bridge constructions in Sec. 8.1, we may assume that $M^\epsilon$ satisfies the following $L^p$ mean curvature estimates for each $p \geq 2$ and $\epsilon$ small:
\begin{equation}
    \Big( \int_{M^\epsilon} |H_0|^p  \Big)^{\frac{1}{p}} \leq c_0 \epsilon^{\frac{n-1}{p}} \text{ when } n \geq 6 \label{Hsmall}
\end{equation}
and
\begin{equation}
     \Big( \int_{M^\epsilon} |H_0|^p \,  \Big)^{\frac{1}{p}} \leq c_0 \epsilon^{\frac{5}{p}} \text{ when } n = 3,4,5, \label{Hsmall1}
\end{equation}
 where the constant $c_0$ is independent of $\epsilon$ in each case. In fact, our bridge constructions when $n = 3,4,5$ yield stronger estimates on the mean curvature than assumed in \eqref{Hsmall1}. The estimates \eqref{Hsmall} and \eqref{Hsmall1} make rigorous the notion that $M^\epsilon$ is approximately minimal. Henceforth, we will identify $M^\epsilon$ with $M^\epsilon \setminus \{p_1, \ldots, p_N\}$. Under this identification, $M^\epsilon$ is embedded.

We will write $M_\delta^\epsilon$ to denote the submanifold constructed the same as $M^\epsilon$, but with the cones $C_i$ replaced by the truncated cones $C_{i,\delta}$ where
\begin{equation}
    C_{i,\delta} := \{(1-t)p_i + t \theta: \theta \in \Sigma_i, \text{ } t \in [\delta,1]\}. \label{tcone}
\end{equation}
The submanifolds $M_\delta^\epsilon$ are smooth embedded compact submanifolds with boundary. For $0 < r_1 < r_2 < 1$ and each $i = 1 ,\ldots, N$, we set 
\begin{equation}
    S_{r_1,r_2}^i := \{x \in C_i : r_1 \leq |x - p_i| \leq r_2\} \text{ and } S_{r_1,r_2} := \bigcup_{i = 1}^N S_{r_1,r_2}^i.
\end{equation}
The submanifolds above will be helpful tools for making the necessary Schauder estimates to complete the gluing procedure.

\subsection{Main Results}
We motivate the paper in the spirit of \cite{NS1}. Let $C_c^{2,\gamma}(M^\epsilon)$ denote the $C^{2,\gamma}$ sections of $NM^\epsilon$ supported away from the vertices $p_i$ of $M^\epsilon$ ($i = 1,\ldots, N)$. Here, $\Psi \in C_c^{2,\gamma}$ may take non-zero values on $\partial M^\epsilon$, and will represent our boundary data. The space $C_c^\infty(M^\epsilon)$ is defined similarly, with smooth sections replacing $C^{2,\gamma}$ sections. We seek $U \in C^{2,\gamma}(M^\epsilon)$ such that a small normal perturbation of $M^\epsilon$ by $U$ is minimal in $\mathbb{R}^{n+m+1}$. To do so, we consider sections of $NM^\epsilon$ contained in a suitable closed subspace $\mathscr{K}(\epsilon)$ of $C^{2,\gamma}(M^\epsilon)$ consisting of those sections which decay at least quadratically near the singularities $p_i$ of $M^\epsilon$, are sufficiently small in $C^1$, and are equal to $\Psi \in C_c^{2,\gamma}(M^\epsilon)$ on $\partial M^\epsilon$ for a suitable class of $\Psi$ (in fact, we will assume $\Psi \in C_c^\infty(M^\epsilon))$.\footnote{See the weighted H\"older spaces in Sec. 3.3 and $\mathscr{K}$ at the beginning of Sec. 4.2.} For $U \in \mathscr{K}(\epsilon)$ and $\epsilon$ small, we can associate a perturbed submanifold which we will denote $M_U^\epsilon$. The submanifold $M_U^\epsilon$ is a small normal perturbation of $M^\epsilon$ which is initially $C^{2,\gamma}$ away from $N$ isolated singularities at the $p_i$. After applying a bootstrapping argument, we will update the regularity of $U$ to conclude $M_U^\epsilon$ is smooth away from the $p_i$. 

By \eqref{gest}, there is an $\epsilon_0$ depending only on $c_0$ in \eqref{bridge1} such that, if $\epsilon < \epsilon_0$ and if $(\Omega, \psi)$ is any smooth local parameterization for $M^\epsilon$ defined as in Sec. 8.1, $(\Omega, \psi + U)$ is a local parameterization for $M_U^\epsilon$. For $U \in \mathscr{K}(\epsilon)$, we can define an operator $H(U): \mathscr{K}(\epsilon)\subset C^{2,\gamma}(M^\epsilon)  \rightarrow C^{0,\gamma}(M_U^\epsilon)$ which sends $U$ to the mean curvature vector $H(U)$ for $M_U^\epsilon$. In terms of a local parameterization, we have $H(U) := \Delta_U^\prime(\psi(x) + U(x))$ where $\Delta_U^\prime$ is the Laplace-Beltrami operator on $M_U^\epsilon$ acting component-wise on $\psi + U$ as a vector in $\mathbb{R}^{n+m+1}$. Expanding this expression, we find
\begin{equation}
    H(U) = \frac{1}{\sqrt{\det g(U)}} \pdv{x^i}\Big(\sqrt{\det g(U)}g^{ij}(U)\pdv{x^j}(\psi + U)\Big), \label{HU}
\end{equation}
where $g(U) := (g_{ij}(U))$ is the metric for $M_U^\epsilon$ given by
\begin{equation*}
    g_{ij}(U) := (\psi + U)_{x^i} \cdot (\psi + U)_{x^j} \text{ and } g(U)^{-1} := (g^{ij}(U)). 
\end{equation*}
We would like to solve $H(U) = 0$ for $U \in \mathscr{K}(\epsilon)$. However, $H(U)$ is not an operator on $NM^\epsilon$. To handle this, we project $H(U)$ onto the normal bundle as in \cite{NS1}.

For each $x \in M^\epsilon$, let $\Pi_x: \mathbb{R}^{n+m+1} \rightarrow N_xM^\epsilon$ denote the orthogonal projection onto the fiber $N_x M^\epsilon$. Set $\overline{x}:= x + U(x)$. The following proposition is standard (e.g. see p. 515 in \cite{NS1}):

\begin{prop}\label{Hperp}
    If $U \in \mathscr{K}(\epsilon)$ for $\epsilon < \epsilon_0$, where $\epsilon_0$ is small and depends only on $c_0$, then for $x \in M^\epsilon$ and $V \in N_{\overline{x}}M_U^\epsilon$ we have 
    $$
        \Pi_x V(\overline{x}) = 0 \text{ if and only if } V(\overline{x}) = 0.
    $$
\end{prop}
It now makes sense to define the map $H^\perp: \mathscr{K}(\epsilon) \times M^\epsilon \rightarrow N_xM^\epsilon$ by $H^\perp(U)(x) := \Pi_xH(U(x))$. Note that $H^\perp(U)$ can be viewed as an analytic mapping of Banach spaces from $\mathscr{K} \subset C^{2,\gamma}(M^\epsilon) \rightarrow C^{0,\gamma}(M^\epsilon)$. Moreover, by Proposition \ref{Hperp} we may take $\epsilon_0$ small to conclude that $H^\perp(U) = 0$ if and only if $H(U) = 0$ for $U \in \mathscr{K}(\epsilon)$ and $\epsilon < \epsilon_0$. Hence, to prove that $M_U^\epsilon$ is minimal it is enough to solve the boundary value problem 
\begin{equation*}
    H^\perp(U) = 0 \text{ on } M^\epsilon \text{ and } U = \Psi \text{ on } \partial M^\epsilon \text{ for } U \in \mathscr{K}(\epsilon), \text{ } \epsilon < \epsilon_0.
\end{equation*}
This is achieved by linearizing the mean curvature operator about the approximate solution $M^\epsilon$ as in \cite{CHS, NS1, NS2}.  

Due to the discussion above, for any $U \in \mathscr{K}(\epsilon)$ with $\epsilon$ small enough we find that $H^\perp(U)$ is well-defined and we can expand $H^\perp(U)$ in a Taylor polynomial as 
\begin{equation*}
    H^\perp(U)(x) = H_0(x) + \dv{H^\perp(tU(x))}{t} \Big\lvert_{t = 0} + \int_0^1(1-t) \dv[2]{H^\perp(tU(x))}{t} \, dt,
\end{equation*}
where $H_0$ is the mean curvature of $M^\epsilon$. It is well known (e.g. see \cite{CHS, Si, NS1, NS2}) that 
$$
    LU = \dv{H^\perp(tU)}{t} \Big\lvert_{t = 0}. 
$$
Thus, if we identify $H_0$ with $-H_0$ and set
\begin{equation}
    E(U) := -\int_0^1(1-t) \dv[2]{H^\perp(tU(x))}{t} \, dt, \label{EU}
\end{equation}
then $M_U^\epsilon$ is a minimal submanifold if and only if 
\begin{equation}\label{jacobi2}
    LU = H_0 + E(U).
\end{equation}

The idea is to invert $L$ in \eqref{jacobi2} for $U \in \mathscr{K}(\epsilon)$ and appropriate boundary data $\Psi \in C_c^\infty(M^\epsilon) \subset C_c^{2,\gamma}(M^\epsilon)$, then solve \eqref{jacobi2} as a fixed point problem using the Schauder fixed point theorem.\footnote{The Schauder fixed point theorem says that if $\mathscr{K}$ is a non-empty convex compact subset of a Banach space $\mathcal{B}$, and if $T : K \rightarrow K$ is a
continuous map, then $T$ has a fixed point (see \cite{GT}).} In the present case, this means:
\begin{equation}
    \begin{cases}
        U =  L^{-1}(H_0 + E(U)) \text{ on } M^\epsilon \\
        U = \Psi \text{ on } \partial M^\epsilon.
    \end{cases} \label{fixed}
\end{equation}
Thus, we need to study the existence and regularity of the Dirichlet problem
\begin{equation}
    \begin{cases}
        LV = F \text{ on } M^\epsilon \\
        V = \Psi \text{ on } \partial M^\epsilon
    \end{cases}\label{dp}
\end{equation}
for appropriate $F \in C_{\text{loc}}^{0,\gamma}(M^\epsilon)$. We can now state the theorems. 

Let $C_1,\ldots, C_N$ be strictly stable minimal cones with links $\Sigma_i$. Theorem \ref{mainthm} below is the natural generalization of Theorem 2.1 in \cite{NS2} to high codimension.
\begin{thm}\label{mainthm}
    Let $n \geq 3$, let $\gamma \in (0,1)$, and let $\nu_1,\ldots, \nu_N$ be any set of $N$ numbers such that $\nu_i \geq 2$ for each $i$. Then there is an $\epsilon_0 > 0$ and a constant $c > 0$ each depending only on $n$, $m$, $\Sigma_1 \dots, \Sigma_N$, $c_0$, $\gamma$, $\nu_i$, $d_0(C_i)$, and $\mathcal{H}^n(M^\epsilon)$ such that, for all $\epsilon \in (0, \epsilon_0)$, there exists a unique smooth solution to \eqref{fixed} satisfying  
    \begin{equation*}
        |U(x)| \leq c|x-p_i|^{\nu_i} \text{ for each } x \in C_i, \text{ } i = 1,\ldots, N 
    \end{equation*}
    for appropriate $\Psi \in C_c^{\infty}(M^\epsilon)$. The perturbed submanifold $M_U^\epsilon$ is smooth away from $N$ isolated singularities at the points $p_i$. Furthermore, it is close to $M^\epsilon$, strictly stable, and is embedded provided $\epsilon_0$ is sufficiently small. If $n \geq 4$ and $M^\epsilon$ is constructed from graphical cones so that it is graphical, we may choose $\epsilon_0$ so that $M_U^\epsilon$ is graphical. 
\end{thm}

Let $\mathcal{G}$ be an $n$-dimensional Lipschitz minimal graph in $\mathbb{R}^{n+m+1}$. If $u: \Omega \subset \mathbb{R}^n \rightarrow \mathbb{R}^{m+1}$ is the defining function for $\mathcal{G}$, then the graph map $\psi(x) := (x,u(x))$ is a weak (integral) solution to the \emph{minimal surface system}:
\begin{equation}
    \begin{cases}
        \sum_{i = 1}^n \pdv{x^i}(\sqrt{g}g^{ij}) = 0, \text{ } j = 1,\ldots, n \\
        \sum_{i,j = 1}^n \pdv{x^i}(\sqrt{g}g^{ij} u_{x^j}^\alpha) = 0, \text{ } \alpha = 1, \ldots, m+1, \label{MSS}
    \end{cases}
\end{equation}
where $g$ is the usual metric on the graph of $u$ defined component-wise by
\begin{equation*}
    g_{ij} = \delta_{ij} + \sum_{\alpha = 1}^{m+1} u_{x^i}^\alpha u_{x^j}^\alpha.
\end{equation*}
When $u$ is $C^2$, the system \eqref{MSS} is equivalent to the following quasilinear elliptic system in non-divergence form:
\begin{equation}
    g^{ij}u_{x^i x^j}^\alpha = 0, \text{ } \alpha = 1,\ldots, m+1. \label{MSS2} 
\end{equation}

In general, any Lipschitz function $u$ such that the graph map $\psi$ solves \eqref{MSS} in the weak sense is said to be a \emph{stationary solution} to the minimal surface system, and $u$ is a stationary solution if and only if its graph is minimal. A stationary solution is a \emph{(strictly) stable solution} if and only if its graph is also (strictly) stable. Applying Theorem \ref{mainthm} to copies of the Lawson-Osserman cone and appealing to the discussion above leads to our most interesting result.
\begin{thm}\label{gsolution}
    Let $N \in \mathbb{N}$. Then there is a smooth compact domain $\Omega \subset \mathbb{R}^4$ and a strictly stable stationary solution $u: \Omega \rightarrow \mathbb{R}^3$ to the minimal surface system which is smooth away from $N$ isolated singularities.
\end{thm}

In \cite{D}, it was shown that the singular set of any stationary solution has Hausdorff dimension at most $n-4$ using Savin's Allard Theorem (\cite{Sa}) and a dimension reduction argument. As a consequence, Theorem \ref{gsolution} is sharp. We further remark that if $n \geq 4$, Theorem \ref{mainthm} and the graphical bridge constructions in Sec. 6 imply that given any suitable finite collection $C_1,\ldots, C_N$ of $n$-dimensional graphical strictly stable cones in $\mathbb{R}^{n+m+1}$, there is an associated infinite family of strictly stable stationary solutions with $N$ islolated singularites having prescribed rates of decay to their tangent cones (i.e. the $C_i$) at their singularities.

\section{Existence and Regularity for the Stability Operator}

In this section, we verify the results in Sec. 3 of \cite{NS2} in the high codimension case. We first derive $L^2$ and Schauder estimates for the equation $LU = F$ on $M_\delta^\epsilon$ for $U \in C^{2,\gamma}(M_\delta^\epsilon)$ and $F \in C^{0,\gamma}(M_\delta^\epsilon)$. From the estimates, we obtain existence and regularity to the Dirichlet problem for $L$ on $M^\epsilon$. Using the existence and regularity theory, we can then argue by separation of variables as in \cite{NS2} to prove an asymptotic formula for solutions to the Dirichlet problem which is essential in the proof of Theorem \ref{mainthm}. 

\subsection{Schauder Estimates}
The stability operator on functions is given by 
$$
    L^\prime := \Delta^\prime + |A|^2.
$$
To obtain $C^0$ bounds, we show that if $U \in C_{\text{loc}}^{2,\gamma}(M^\epsilon)$ solves $LU = F$ pointwise on $M^\epsilon$ for $F \in C_{\text{loc}}^{0,\gamma}(M^\epsilon)$, then $|U|$ is an $H^1(M_\delta^\epsilon)$ subsolution to the elliptic scalar equation $L^\prime |U| = - |F|$ on $M_\delta^\epsilon$ for each $\delta > 0$ and apply De Giorgi-Nash-Moser $C^0$ estimates. 

Going forward, we will say a constant $c>0$ is universal if it depends on $n$, $m$, $\Sigma_1, \ldots, \Sigma_N$, $c_0$, $\gamma$, $\nu$, $d_0(C_i)$, and $\mathcal{H}^n(M^\epsilon)$ and is independent of $\epsilon$. We will sometimes allow universal constants to depend on additional parameters, but will make it clear when this is the case. In addition, we can assume that the coefficients of \eqref{genstab} are uniformly bounded independent of $\epsilon$ on any compact subset of $M^\epsilon$ due to the bridge constructions in Sec. 8.1.
\begin{lem}\label{weaksoln}
    Suppose $F \in C_{\text{loc}}^{0,\gamma}(M^\epsilon)$ and that $U \in C_{\text{loc}}^{2,\gamma}(M^\epsilon)$ satisfies $LU = F$ pointwise on $M^\epsilon$. Then for each $\delta > 0$, we have $|U| \in H^1(M_\delta^\epsilon)$ (viewed as a space of functions) and $|U|$ is a weak subsolution to
    $$
        L^\prime|U| = -|F| \text{ on } M_\delta^\epsilon.
    $$
\end{lem}
\begin{proof}
    We adapt the argument given in Lemma 4.6 in \cite{HL}. The fact that $|U| \in H^{1}(M_\delta^\epsilon)$ is well known. When $x_0 \in \{|U| > 0\}$, then $|U|$ is $C^2$ near $x_0$ so we may compute $\Delta^\prime |U|(x_0)$ directly. Let $e_1, \ldots, e_n$ be an orthonormal frame for $TM^\epsilon$ near $x_0$. Then 
    $$
        \Delta^\prime |U| = \frac{\langle \Delta U , U \rangle}{|U|} + \sum_{i = 1}^n \frac{|\nabla_{e_i} U|^2}{|U|} - \sum_{i = 1}^n \frac{\langle \nabla_{e_i} U , U\rangle^2}{|U|^3}.
    $$
    In particular,
    $$
        -\Delta^\prime|U| \leq |\Delta U| \leq |A|^2|U| +|F| \text{ at } x_0
    $$
    implying 
    $$
        \Delta^\prime |U| + |A|^2|U| \geq -|F| \text{ at } x_0.
    $$
   Hence, $|U|$ is a $C^2$ subsolution to $L^\prime w = -|F|$ on $\{|U| > 0\}$. For each $\xi > 0$, set
   $$
        \phi_\xi(s) := \max\{0, s-\xi\}.
   $$
    By writing
   $$
        \phi_\xi(s) = \frac{s-\xi + |s-\xi|}{2},
   $$
   we see that $\phi_\xi^\prime(s) = 0$ almost everywhere on $\{s \leq \xi\}$ and $\phi_\xi^\prime(s) = 1$ on $\{s > \xi\}$. Let $v \in C_0^\infty(M_\delta^\epsilon)$ be a function such that $v\geq0$ and set $\phi_{\xi,\tau} := \eta_\tau * \phi_\xi$ where $\eta_\tau$ is the standard mollifier. Using that $\phi_{\xi,\tau}$ is $C^2$ and convex and that $|U|$ is a subsolution on $\{|U| > 0\}$, we may compute
    \begin{align*}
         \int_{M_\delta^\epsilon} \langle \nabla^\prime \phi_{\xi,\tau}(|U|) , \nabla^\prime v \rangle &= \int_{M_\delta^\epsilon} \langle \phi_{\xi,\tau}^\prime(|U|) \nabla^\prime|U|, \nabla^\prime v\rangle \\
         &=\int_{M_\delta^\epsilon} \big(\langle \nabla^\prime |U|, \nabla^\prime(\phi_{\xi,\tau}^\prime(|U|)v) \rangle - \langle \nabla^\prime|U|, \nabla^\prime |U| \rangle v \phi_{\xi,\tau}^{\prime \prime}(|U|) \big)\\
         &\leq \int_{M_\delta^\epsilon} |F|\phi_{\xi,\tau}^\prime(|U|) v   + \int_{M_\delta^\epsilon} |A|^2|U|\phi_{\xi,\tau}^\prime(|U|) v 
    \end{align*}
    for $\tau$ small enough. Letting $\tau \rightarrow 0^+$, we see that 
    $$
         \int_{M_\delta^\epsilon} \langle \nabla^\prime \phi_{\xi}(|U|), \nabla^\prime v \rangle  \leq \int_{M_\delta^\epsilon} |F|\phi_{\xi}^\prime(|U|) v + \int_{M_\delta^\epsilon} |A|^2|U|\phi_{\xi}^\prime(|U|) v.
    $$
    Note that $|\nabla^\prime \phi_\xi(|U|)| = |\phi_\xi^\prime(|U|)\nabla^\prime |U|| \leq |\nabla^\prime|U||$ and $\phi_\xi^\prime(|U|)$ increases to the characteristic function $\chi_{\{|U|>0\}}$ as $\xi \rightarrow 0$. Letting $\xi \rightarrow 0$ and applying the dominated convergence theorem shows 
    $$
        \int_{M_\delta^\epsilon} \langle \nabla^\prime |U| , \nabla^\prime v \rangle \leq \int_{M_\delta^\epsilon} |F| v + \int_{M_\delta^\epsilon} |A|^2|U| v,
    $$
    so that $|U|$ is a subsolution.
\end{proof}

As a consequence of Lemma \ref{weaksoln}, we easily obtain both local and global bounds for $L$. 
\begin{coro}[Global $C^0$ Bound] \label{sest2}
    Suppose $LU = F$ on $M_\delta^\epsilon$ for some $\delta \in (0,1)$ and $U \in C^2(M_\delta^\epsilon)$. If $n \geq 4$, then for any $\tau > 0$ there is a universal constant $c := c(\tau,\delta)$ such that
    $$
        \norm{U}_{C^0(M_\delta^\epsilon)} \leq c\big(\norm{F}_{L^{\frac{n}{2} + \tau}(M_\delta^\epsilon)} + \norm{U}_{C^0( \partial M_\delta^\epsilon)} + \norm{U}_{L^2(M_\delta^\epsilon)}\big).
    $$
    If $n = 3$, then 
    $$
        \norm{U}_{C^0(M_\delta^\epsilon)} \leq c\big(\norm{F}_{L^{2}(M_\delta^\epsilon)} + \norm{U}_{C^0( \partial M_\delta^\epsilon)} + \norm{U}_{L^2(M_\delta^\epsilon)}\big)
    $$
    for some universal constant $c := c(\delta)$.
\end{coro}
\begin{proof}
    Since the coefficients of $L^\prime$ are uniformly bounded in $C^0$ independent of $\epsilon$ for each $n \geq 3$, $M_\delta^\epsilon$ is compact, and $M_\delta^\epsilon$ satisfies a uniform Sobolev inequality (see \cite{MS}), we may apply Lemma \ref{weaksoln}, \eqref{scale}, and \eqref{bridge1}
    to conclude:
    \begin{align*}
        \norm{U}_{C^0(M_\delta^\epsilon)} &\leq c(\tau)\big(\norm{F}_{L^{\frac{n}{2} + \tau}(M_\delta^\epsilon)} + \norm{|A|^2|U|}_{L^{\frac{n}{2} + \tau}(M_\delta^\epsilon)} + \norm{U}_{C^0( \partial M_\delta^\epsilon)}\big) \\
        &\leq c\big(\tau,\delta)(\norm{F}_{L^{\frac{n}{2} + \tau}(M_\delta^\epsilon)} + \norm{U}_{L^{\frac{n}{2} + \tau}(M_\delta^\epsilon)} + \norm{U}_{C^0( \partial M_\delta^\epsilon)}\big).
    \end{align*}
    The estimate for $n \geq 4$ follows from Young's inequality $ab \leq \sigma a^p + c_\sigma b^q$ with
    $$
        p = \frac{n + 2\tau}{n - 4 + 2\tau} \text{ and } q = \frac{p}{p-1} \text{ (e.g. see Proposition 3.2 in \cite{NS2})}.
    $$
    When $n = 3$, we choose $\tau = 4^{-1}$ and use H\"older's inequality with $p = \frac{8}{7}$ and $q = 8$ in the first and second terms on the right-hand side.
\end{proof}
\begin{coro}[Local $C^0$ Bound]\label{sest1}
    Let $K \subset M^\epsilon$ be a compact region and let $F \in C^{0,\gamma}(K)$. Suppose $U \in C^{2,\gamma}(K)$ satisfies $LU = F$ on $K$. Let $B_{2r} \subset K$ be a geodesic ball in $M^\epsilon$ of radius $2r$. Then there is a universal constant $c := c(K)$ such that
    $$
        \norm{U}_{C^0(B_r)} \leq c\big(r^{-\frac{n}{2}}\norm{U}_{L^2(B_{2r})} + r\norm{F}_{L^n(B_{2r})}\big),
    $$
\end{coro}
\begin{proof}
   Since $K \subset M_\delta^\epsilon$ for some $\delta > 0$, this follows immediately from the Lemma \ref{weaksoln} and Theorem 8.17 in \cite{GT}, as well as the uniform boundedness, independent of $\epsilon$, of the coefficients of $L^\prime$ on $K$ for $n \geq 3$.
\end{proof}

We will need local and global Schauder estimates on $M_\delta^\epsilon$. The global Schauder estimates below are proved in Sec. 3 of \cite{NS1} (see Lemma 2 and Lemma 3).

\begin{prop}[Global Schauder Estimates] \label{sest3}
    Suppose $U \in C^{2,\gamma}(M_\delta^\epsilon)$ for some $\delta \in (0,1)$ and $\gamma \in (0,1)$. Let $F \in C^{0,\gamma}(M^\epsilon)$ and suppose $U$ solves the Dirichlet problem
    \begin{equation*}
        \begin{cases}
            LV = F \text{ on } M_\delta^\epsilon \\
            V = 0 \text{ on } \partial M_\delta^\epsilon.
        \end{cases}
    \end{equation*}
    Then there is a universal constant $\epsilon_0$ and a universal constant $c(\delta)$ such that, if $\epsilon < \epsilon_0$, then
    \begin{itemize}
     \itemsep = 3pt
        \item[(a)] $\norm{\nabla^k U}_{C^0(M_\delta^\epsilon)} \leq c(\delta)\big(\epsilon^{-k}\norm{U}_{C^{0}(M_\delta^\epsilon)} + \epsilon^{2-k}\norm{F}_{C^{0,\gamma}(M_\delta^\epsilon)}\big)$ for $k = 1,2$ and
        \item[(b)] $|U|_{k, \gamma ,M_\delta^\epsilon} \leq c(\delta)\big(\epsilon^{-k-\gamma}\norm{U}_{C^0(M_\delta^\epsilon)} + \epsilon^{2-k-\gamma}\norm{F}_{C^{0,\gamma}(M_\delta^\epsilon)} \big)$ for $k = 0,1,2$.
    \end{itemize}
\end{prop}

Using the local Schauder estimates in Sec. 8.2 for linear elliptic systems, we can prove local Schauder estimates for $L$ on $M^\epsilon$. To do so, we need some interpolation inequalities for $C^{2,\gamma}$ functions on a closed $n$-ball $B_r^n$ in Euclidean space: 
\begin{align}
    \sum_{i,j} \norm{u_{x^i x^j}}_{C^0(B_r^n)} &\leq r^\gamma \sum_{i,j} |u_{x^i x^j}|_{\gamma,B_r^n} + cr^{-2}\norm{u}_{C^0(B_r^n)}, \label{i1} \\
    \sum_i|u_{x^i}|_{\gamma, B_r^n} &\leq r\sum_{i,j}|u_{x^i x^j}|_{\gamma,B_r^n} + cr^{-1-\gamma}\norm{u}_{C^0(B_r^n)}, \label{i2}\\
    \sum_i \norm{u_{x^i}}_{C^0(B_r^n)} &\leq r^{1+\gamma}\sum_{i,j} |u_{x^i x^j}|_{\gamma, B_r^n} + c r^{-1}\norm{u}_{C^0(B_r^n)}, \label{i3}\\
    |u|_{\gamma,B_r^n} &\leq r^2 \sum_{i,j}|u_{x^i x^j}|_{\gamma,B_r^n} + c r^{-\gamma} \norm{u}_{C^0(B_r^n)}. \label{i4}
\end{align}
For the details, see \cite{GT}.
\begin{prop}[Local Schauder Estimates] \label{sest4}
    Let $K$ be a compact region contained in the interior of $M^\epsilon$, let $B_{2r} \subset K$ be a geodesic ball with $r \in (0, 1]$, and let $F \in C^{0,\gamma}(K)$. Suppose $U \in C^{2,\gamma}(K)$ satisfies $LU = F$ on $K$. Then there is a universal constant $c:=c(K)$ such that
    \begin{itemize}
     \itemsep = 3pt
        \item[(a)] $\norm{\nabla^k U}_{C^0(B_r)} \leq c\big(r^{-k}\norm{U}_{C^0(B_{2r})} + r^{2-k}\norm{F}_{C^{0,\gamma}(B_{2r})}\big)$ for $k=1,2$, and 
        \item[(b)] $|U|_{k,\gamma, B_r} \leq c\big( r^{-k-\gamma}\norm{U}_{C^0(B_{2r})} + r^{2-k-\gamma}\norm{F}_{C^{0,\gamma}(B_{2r})}\big)$ for $k=0,1,2$.
    \end{itemize}
\end{prop}
\begin{proof}
Let $n_1, \ldots, n_{m+1}$ be a local orthonormal frame for $B_{2r}$ and write $U = u^\alpha n_\alpha$. The constant $c_1,c_2,\ldots$ will denote universal constants.
    \begin{itemize}
        \item[(a)] Using \eqref{coord1}, we find
    \begin{align*}
    \sup_{B_r}|\nabla U| &\leq c_1 \sum_{i,\alpha,\beta} \norm{u_{x^i}^\alpha + u^\beta B_{\beta i}^\alpha}_{C^0(B_r^n)} \\
    &\leq c_2\Big( \sum_{i,\alpha} \norm{u_{x^i}^\alpha}_{C^0(B_r^n)} + \sum_{\alpha} \norm{u^\alpha}_{C^0(B_r^n)}\Big)
    \end{align*}
    Applying the interpolation inequality \eqref{i3} to the first term above and using that $r \leq 1$ in the second, we get
    \begin{equation}
         \sup_{B_r}|\nabla U| \leq c_3\Big(r^{1 + \gamma}\sum_{i,j,\alpha} |u_{x^i x^j}^\alpha|_{\gamma, B_r^n}  + r^{-1} \norm{U}_{C^0(B_{r})}\Big).\label{gradest}
    \end{equation}
    Similarly, we can use \eqref{coord2} to find
    \begin{equation}
        \sup_{B_r}|\nabla^2 U| \leq c_4 \Big(r^\gamma \sum_{i,j,\alpha} |u_{x^i x^j}^\alpha|_{\gamma, B_r^n} + r^{-2} \norm{U}_{C^0(B_{r})}\Big). \label{hessest}
    \end{equation}
    By the coordinate formula \eqref{genstab}, we can apply Corollary \ref{cor1} in \eqref{gradest} and \eqref{hessest} to conclude (a).

    \item[(b)] We have
   \begin{align*}
       |U|_{2,\gamma, B_{r}} &\leq c_5\sum_{i,j,\alpha}\Big(\norm{u_{x^i x^j}^\alpha}_{C^0(B_r^n)} + |u_{x^i x^j}^\alpha|_{\gamma, B_r^n} + \norm{u_{x^i}^\alpha}_{C^0(B_r^n)}\nonumber  \\
       &\qquad+ |u_{x^i}^\alpha|_{\gamma, B_r^n}  + |u_{x^j}^\alpha|_{\gamma, B_r^n} + \norm{u^\alpha}_{C^0(B_r^n)} + |u^\alpha|_{\gamma,B_r^n}\Big).
   \end{align*}
   Applying the interpolation inequalities \eqref{i1}-\eqref{i4} on the right-hand side above and using that $r \leq 1$ yields
   \begin{equation}
         |U|_{2,\gamma, B_{r}} \leq c_6\Big(\sum_{i,j,\alpha} |u_{x^i x^j}^\alpha|_{\gamma, B_r^n} + r^{-2-\gamma}\norm{U}_{C^0(B_{2r})}\Big) \label{lio}
   \end{equation}
   By the coordinate expression \eqref{genstab}, we can apply Corollary \ref{cor1} to the first term in \eqref{lio} to find
   \begin{equation*}
       |U|_{2,\gamma, B_{r}} \leq c_7 \Big( r^{-\gamma}\norm{F}_{C^{0,\gamma}(B_{2r})} + r^{-2-\gamma}\norm{U}_{C^0(B_{2r})}\Big),
   \end{equation*}
   proving the desired inequality when $k = 2$. The other cases are similar and follow by interpolation.
    \end{itemize}
    Combining parts (a) and (b) completes the proof.
\end{proof}

\subsection{Existence and Regularity for the Dirichlet Problem}
We briefly summarize the existence and regularity theory for the Dirichlet problem for $L$ on $M^\epsilon$. Set
\begin{equation}
   h(x) := \begin{cases}
                |x-p_i| \text{ for } x \in C_i \\
                1 \text{ for } x \in \cup_l \Gamma_l(\epsilon).
            \end{cases} \label{h}
\end{equation}
Let $H_{1}^1(M_\delta^\epsilon)$ be the set $H^1(M_\delta^\epsilon)$ equipped with the weighted Sobolev norm
\begin{equation}
    \norm{U}_{H_{1}^1(M_\delta^\epsilon)} := \Big(\int_{M_\delta^\epsilon} |\nabla U|^2 \, d\mathcal{H}^n + \int_{M_\delta^\epsilon} |U|^2h^{-2} \, d\mathcal{H}^n\Big)^{\frac{1}{2}}, \label{wsobolev} 
\end{equation}
and let $H_{0,1}^1(M_\delta^\epsilon)$ be the closure of $C_0^\infty(M_\delta^\epsilon)$ in $H_{1}^1(M_\delta^\epsilon)$ with respect to \eqref{wsobolev}.
Then $H_{1}^1(M_\delta^\epsilon)$ and $H_{0,1}^1(M_\delta^\epsilon)$ are Hilbert spaces when given the inner product associated to \eqref{wsobolev}. Furthermore, the spaces $H_{0,1}^1(M_\delta^\epsilon)$ and $H_0^1(M_\delta^\epsilon)$ coincide. Hence, the norm \eqref{wsobolev} is equivalent to the $H^1$ norm \eqref{snorm} on $M_\delta^\epsilon$. The definition of $H_{1}^1(M_\delta^\epsilon)$ can be extended to normal sections on $M^\epsilon$ by restricting $H^1(M^\epsilon)$ to the subclass of normal sections for which the norm \eqref{wsobolev} is finite, where the integrals are taken over $M^\epsilon$. In this case, the $H^1(M^\epsilon)$ and $H_{1}^1(M^\epsilon)$ norms are not equivalent, though we do have the inequality $\norm{U}_{H^1(M^\epsilon)} \leq \norm{U}_{H_{1}^1(M^\epsilon)}$. We refer the reader to \cite{B} and \cite{PR} Ch. 8 for more on the weighted Sobolev spaces above. 

By \eqref{param1}-\eqref{nest}, Remark \ref{nbdds}, and the polar coordinate expression for $L$, $h^2L$ has bounded smooth coefficients independent of $\delta$ for each $n \geq 3$, and the symmetric bilinear form on $H_{0,1}^1(M_\delta^\epsilon)$ associated to $h^2L$ (i.e. the bilinear form associated to \eqref{weak}) is continuous and weakly coercive on $H_{0,1}^1(M_\delta^\epsilon)$ (see Definition 3.41 in \cite{Ma}). Thus, the Fredholm alternative applies. In addition, $h^2 L$ is strongly elliptic on $C_0^\infty(M_\delta^\epsilon)$ and self-adjoint with respect to 
\begin{equation}
    L_{1}^2(M_\delta^\epsilon) := L^2(NM_\delta^\epsilon; h^{-2}d\mathcal{H}^n). \label{hminus}
\end{equation}
Therefore, $h^2 L$ has a discrete set of eigenvalues
$$
    \lambda_1 \leq \lambda_2 \leq \lambda_3 \leq \cdots \text{ } (\lambda_i := \lambda_i(\epsilon,\delta))
$$
with $\lambda_i \rightarrow \infty$ and a corresponding orthonormal basis of smooth eigensections for $L_{1}^2(M_\delta^\epsilon)$ given by
$$
    \Phi_{1}, \Phi_{2}, \Phi_3, \ldots,
$$
where $\Phi_{i}$ vanishes on $\partial M_\delta^\epsilon$ for each $i \in \mathbb{N}$, each $\delta \in (0,1)$, and $\epsilon$ small. 

We define $L_{1}^2(M^\epsilon)$ the same as in \eqref{hminus} with $M^\epsilon$ replacing $M_\delta^\epsilon$. It will be helpful to also define the spaces
$$
    L_{-1}^2(M_\delta^\epsilon) := L^2(M_\delta^\epsilon; h^{2}d\mathcal{H}^n) \text{ and } L_{-1}^2(M^\epsilon) := L^2(M^\epsilon; h^{2}d\mathcal{H}^n).
$$
Of course, the $L_{-1}^2$ norms are always bounded by the $L^2$ norms and the $L_{1}^2$ norms always bound the $L^2$ norms. The main results in this section are as follows:

\begin{lem}[\cite{NS2}, Lemma 3.1]\label{deltastable}
    Let $d_1 := \min_{1,\ldots, N} d_0(C_i) >0$. There are universal constants $d_0 \in (0,1)$ and $\epsilon_0$ (depending on $d_1$) such that for all $\epsilon \in (0,\epsilon_0)$, all $\delta \in (0,1)$, and all $U \in C_{0}^2(M_\delta^\epsilon)$ we have
    \begin{equation*}
        \lambda_1(\delta,\epsilon) \geq d_0 \text{ and } \norm{U}_{H_{1}^1(M_\delta^\epsilon)} \leq d_0^{-1} \norm{LU}_{L_{-1}^2(M_\delta^\epsilon)}.
    \end{equation*}
\end{lem}

\begin{lem}[\cite{NS2}, Lemma 3.2]\label{existence}
    Let $F \in C_{\text{loc}}^{0,\gamma}(M^\epsilon) \cap L_{-1}^2(M^\epsilon)$, $\Psi \in C_c^{2,\gamma}(M^\epsilon)$, $\epsilon \leq \epsilon_0$, and $\gamma \in (0,1)$. Then there exists a unique solution $U \in C_{\text{loc}}^{2,\gamma}(M^\epsilon) \cap L_{1}^2(M^\epsilon)$ to the Dirichlet problem
    \begin{equation*}
        \begin{cases}
            LV = F \text{ on } M^\epsilon \\
            V = \Psi \text{ on } \partial M^\epsilon \label{DP}
        \end{cases}
    \end{equation*}
    satisfying
     \begin{equation*}
         \norm{U}_{H_{1}^1(M^\epsilon)} 
        \leq d_0^{-1} \Big( \norm{F}_{L_{-1}^2(M^\epsilon)} + \norm{\Psi}_{H_{1}^1(M^\epsilon)} + \norm{L\Psi}_{L_{-1}^2(M^\epsilon)}\Big).
    \end{equation*}
\end{lem}

Replacing $L^\prime$ with $L^\prime + |\lambda_1| h^{-2}$ in the proof of Lemma \ref{weaksoln}, we see that Lemma \ref{weaksoln} holds for $L + \lambda_1h^{-2}$ as well, along with Corollary \ref{sest2}. Combining this with the local Schauder estimates in the previous section, we may argue precisely as in Lemma 3.1 in \cite{NS2} to prove Lemma \ref{deltastable}. In order to prove Lemma \ref{existence} via the compactness argument in Lemma 3.2 in \cite{NS2}, we need to study existence and regularity of the Dirichlet problem for $L$ on $M_\delta^\epsilon$. Note that it is enough to prove Lemma \ref{existence} for $\Psi = 0$ by linearity of $L$. Hence, we only consider the case of zero boundary data.

\begin{lem}\label{reg1}
    Fix $\delta \in (0,1)$ and let $F \in C^{0,\gamma}(M_\delta^\epsilon)$. Suppose $0 < \epsilon < \epsilon_0$, where $\epsilon_0$ is as in Lemma \ref{deltastable}. Then the Dirichlet problem 
    \begin{equation*} 
        \begin{cases}
             L V = F \text{ on } M_\delta^\epsilon \\
             V = 0 \text{ on } \partial M_\delta^\epsilon \label{DP2}
        \end{cases}
    \end{equation*}
    has a unique solution $U \in C^{2,\gamma}(M_\delta^\epsilon)$. Furthermore, $U$ satisfies 
    \begin{equation*}
        \norm{U}_{H_{1}^1(M_\delta^\epsilon)} \leq d_0^{-1} \norm{F}_{L_{-1}^2(M_\delta^\epsilon)}.  \label{DPineq}
    \end{equation*}
\end{lem}
\begin{proof}
    Existence and uniqueness of an $H_{0,1}^1(M_\delta^\epsilon)$ solution $U$ with $L$ replaced by $h^2L$ and $F$ replaced by $h^2F$ follows from Lemma \ref{deltastable} and the Fredholm alternative. To see this, note that any solution to $h^2 L U= 0$ with $U = 0$ on $\partial M_\delta^\epsilon$ is smooth by elliptic regularity, and Lemma \ref{deltastable} says that such a $U$ must be zero. Since $H_{0,1}^1(M_\delta^\epsilon)$ and $H_0^1(M_\delta^\epsilon)$ coincide, a $H_{0,1}^1(M_\delta^\epsilon)$ weak solution of $h^2 L U = h^2 F$ satisfies $LU = F$ in the weak sense on $H_0^1(M_\delta^\epsilon)$. We show $U \in C^{2,\gamma}(M_\delta^\epsilon)$ by approximation with smooth normal sections.

    Let $\{F_j\}_1^\infty$ be a sequence in $C^\infty(M_\delta^\epsilon)$ converging uniformly to $F$ which is uniformly bounded in the $C^{0,\gamma}$ norm. Let $\{U_j\}_1^\infty$ be the corresponding sequence of unique $H_{0,1}^1(M_\delta^\epsilon)$ weak solutions to the problem $h^2L V = h^2F_j$ for each $j \in \mathbb{N}$. Then the $U_j$ are $H_0^1(M_\delta^\epsilon)$ weak solutions to $LV = F_j$ for each $j$. Since the $F_j$ are smooth, the compactness of $M_\delta^\epsilon$, $L^2$ elliptic regularity, and the Sobolev embedding theorem together show the $U_j$ are in  $C^k(M_\delta^\epsilon)$ for each $k$ and each $j$. Since uniform convergence implies convergence in $L^p(M_\delta^\epsilon)$ for any $p \geq 1$, Corollary \ref{sest2} together with Lemma \ref{deltastable} imply that the $U_j$ are uniformly bounded in the $L^\infty$ norm on $M_\delta^\epsilon$ for fixed $\epsilon$ giving a uniform $C^0$ bound for the $U_j$. Proposition \ref{sest3} applied to the $U_j$ then gives a uniform $C^{2,\gamma}$ bound so by the Arzel\'a-Ascoli theorem there is a subsequence of the $U_j$ (which we do not relabel) converging to some $\tilde{U} \in C^{2,\gamma}(M_\delta^\epsilon)$ in the $C^{2,\frac{\gamma}{2}}$ norm. By construction, we must have $h^2L\tilde{U} = h^2F$ in the weak sense on $H_{0,1}^1(M_\delta^\epsilon)$ with $\tilde{U} \in H_{0,1}^1(M_\delta^\epsilon)$ since $\tilde{U} \in C^{2,\gamma}$ on the truncated submanifold $M_\delta^\epsilon$. Thus, $U = \tilde{U}$ by uniqueness so $U \in C^{2,\gamma}(M_\delta^\epsilon)$. The desired inequality is now immediate from Lemma \ref{deltastable}.
\end{proof}
Using Lemma \ref{reg1} and the estimates derived in Sec. 3.1, we can now follow the arguments in \cite{NS2} to prove Lemma \ref{existence}. 

\subsection{Asymptotic Formula}
We now state an important lemma due to \cite{CHS,NS2} concerning the asymptotic behavior of solutions to \eqref{dp} near the singularities $p_1,\ldots,p_N$ (see Lemma \ref{eigenexp}). The normal sections $F$ are chosen so that they decay sufficiently rapidly near the $p_i$. This will allow us to expand our solution $U$ to \eqref{dp} in a Fourier series and solve an ODE for the coefficient functions. The asymptotic formula then follows from the prior estimates. For a comprehensive discussion of the material in this section, we refer the reader to \cite{CHS, Oo, NS2} and ch. 2 in \cite{PR}. 

For an $N$-vector $\nu := (\nu_1,\ldots, \nu_N)$ with $\nu_i >0$ for each $i$, set 
\begin{align*}
    C_\nu^{k,\gamma}(M^\epsilon) := \Big\{ U \in C_{\text{loc}}^{k,\gamma}(M^\epsilon) : &\max_{r \in (0,\frac{1}{2}]} \sum_{j = 0}^k \norm{\nabla^j U}_{C^0(S_{r,2r}^i)}r^{j - \nu_i} < \infty \text{ and } \\
    &\max_{r \in (0,\frac{1}{2}]} \sum_{j = 0}^k| U|_{j,\gamma, S_{r,2r}^i}r^{j + \gamma - \nu_i} < \infty, \text{ } i = 1,\ldots,N \Big\}
\end{align*}
where $\gamma \in (0,1)$ and $k = 0,1,2$. The spaces $C_\nu^{k,\gamma}(M^\epsilon)$ are Banach spaces with the norm
\begin{multline*}
    \norm{U}_{C_\nu^{k,\gamma}} := \max_{\substack{r \in (0, \frac{1}{2}] \\ i = 1,\ldots, N}} \Big( \sum_{j = 0}^k \norm{\nabla^j U}_{C^0( S_{r,2r}^i)}r^{j - \nu_i} + \sum_{j = 0}^k|U|_{j,\gamma, S_{r,2r}^i}r^{j + \gamma - \nu_i}\Big)  \\+ \norm{U}_{C^{k,\gamma}(M_{\frac{1}{2}}^\epsilon)}.
\end{multline*}
One can of course define the $C_\nu^k$ norm for $k = 0,1,2$ and the corresponding spaces $C_\nu^k(M^\epsilon)$. Now, suppose $\nu_i \geq 2$ for each $i$ and let $\overline{\nu} := \nu - \big(\frac{3}{2}, \ldots, \frac{3}{2}\big)$ where the second vector has $N$ components. As discussed in Sec. 2.4, to prove Theorem \ref{mainthm} we look for $U \in \mathscr{K} \subset C_{\nu}^{2,\gamma}(M^\epsilon)$ satisfying \eqref{fixed} for appropriate $\Psi \in C_c^\infty(M^\epsilon)$. 

Let $F \in C_{\overline{\nu}}^{0,\gamma}(M^\epsilon)$, let $\Psi \in C_c^{2,\gamma}(M^\epsilon)$, and let $U$ be the solution to the Dirichlet problem \eqref{dp} given by Lemma \ref{existence}. Fix a cone $C_i$ and recall that in polar coordinates $x = r\omega$ on $C_i$ we have
\begin{equation}
    LU(r\omega) = \pdv[2]{U(r\omega)}{r} + \frac{n-1}{r} \pdv{U(r\omega)}{r}  +\frac{1}{r^2} L_{\Sigma_i} U(r\omega), \label{Lcone}
\end{equation}
where $L_{\Sigma_i} := \Delta_{\Sigma_i} + \tilde{A}_{\Sigma_i}$. As before, $L_{\Sigma_i}$ is strongly elliptic on $C^\infty(\Sigma_i)$ and self-adjoint with respect to $L^2(\Sigma_i)$ so $L_{\Sigma_i}$ has a discrete set of eigenvalues for each $i = 1,\ldots, N$ 
$$
    \mu_1^i \leq \mu_2^i \leq \mu_3^i \leq \cdots
$$
with $\mu_j^i \rightarrow \infty$ as $j \rightarrow \infty$, and a corresponding smooth orthonormal basis of eigensections for $L^2(\Sigma_i)$ 
$$
    \eta_1^i, \eta_2^i, \eta_3^i, \ldots.
$$
That is,
$$
    \Delta_{\Sigma_i} \eta_j^i +  \tilde{A}_{\Sigma_i}(\eta_j^i) + \mu_j^i \eta_j^i = 0 \text{ for } i = 1,\ldots, N \text{ and } j = 1,2,\ldots
$$
with 
$$
    \int_{\Sigma_i} \langle \eta_j^i, \eta_k^i \rangle \, d\mathcal{H}^{n-1}(\omega) = \delta_{jk} \text{ for each } i,j,k,
$$
where $\mathcal{H}^{n-1}(\omega)$ is the $(n-1)$-dimensional Hausdorff measure on $\Sigma_i$. 

For $r > 0$, we can expand $U$ uniquely in a Fourier series (see \cite{CHS, DL, Si}) as 
$$
    U(r\omega) := \sum_{j = 1}^\infty a_j^i(r) \eta_j^i(\omega) 
$$
for coefficient functions $a_j^i(r)$, where $i = 1,\ldots,N$ and $j=1,2,\ldots$. Furthermore, the $a_j^i(r)$ are square summable since $U(r\cdot) \in L^2(\Sigma_i)$. Substituting the above expression for $U$ into \eqref{Lcone}, we see that the $a_j^i$ satisfy the ordinary differential equations
\begin{equation}
    r^2w_j^{\prime \prime} +(n-1)rw_j^\prime - \mu_j^i w_j = r^2 f_j^i \text{ for each } j = 1,2,\ldots \label{ODE}
\end{equation}
where 
\begin{equation}
    f_j^i(r) := \int_{\Sigma_i} \langle F(r\omega) , \eta_j^i(\omega) \rangle \, d\mathcal{H}^{n-1}(\omega). \label{fij}
\end{equation}

Note that $(n-2)^2 + 4\mu_j^i > 0$ by Definition \ref{stabcondition} and the strict stability of the $C_i$, so we may define the real numbers
\begin{align}
    \gamma_j^i(+) &:= \frac{(2-n) + \sqrt{(n-2)^2 + 4\mu_j^i}}{2},\label{gamma+} \\
    \gamma_j^i(-) &:= \frac{(2-n) - \sqrt{(n-2)^2 + 4\mu_j^i}}{2}.
\end{align}
By direct computation, it is easy to check that $r^{\gamma_j^i(+)}$ and $r^{\gamma_j^i(-)}$ are two linearly independent solutions to the homogeneous version of \eqref{ODE} (i.e. $f_j^i \equiv 0$) for each $i,j$. We now let $J_i$ be a positive integer such that 
$$
    \gamma_{J_i}^i(+) < \nu_i \leq \gamma_{J_i+1}^i(+) \text{ for } i = 1,\ldots,N,
$$
which can always be done since $\gamma_1^i(+) \leq 0$ for each $i$ and $\gamma_j^i(+) \rightarrow \infty$ as $j \rightarrow \infty$ (see Lemma 5.1.3 in \cite{Si}). In addition, we let
\begin{equation}
    F_j^i(F)(r) := \begin{cases}
                        r^{\gamma_j^i(+)} \int_0^r \tau^{1-n-2\gamma_j^i(+)}\int_0^\tau s^{n-1 + \gamma_j^i(+)} f_j^i(s) \, ds d\tau \text{ for } j \leq J_i \\
                        r^{\gamma_j^i(+)}\int_1^r \tau^{1-n-2\gamma_j^i(+)} \int_0^\tau s^{n-1 + \gamma_j^i(+)} f_j^i(s) \, dsd\tau \text{ for } j > J_i.
                    \end{cases}\label{F}
\end{equation}
Note that we can take zero to be the lower limit of integration in $s$ in the first line in \eqref{F} since $F \in C_{\overline{\nu}}^{0,\gamma}(M^\epsilon)$ and $2^{-1}(2-n) \leq \gamma_j^i(+) < \nu_i$ for $j \leq J_i$. The functions $F_j^i(F)$ are particular solutions to \eqref{ODE} for each $i,j$ (\cite{CHS,Oo}). We thereby obtain our desired solution to \eqref{ODE}:
$$
    a_j^i(r) = \alpha_j^ir^{\gamma_j^i(+)} + \beta_j^ir^{\gamma_j^i(-)} + F_j^i(F)(r)
$$
for $r \in (0,1]$ and some numbers $\alpha_j^i,\beta_j^i$, where $i = 1,\ldots,N$ and $j = 1,2,\ldots$. In fact, we have $\beta_j^i \equiv 0$ (see Lemma \ref{eigenexp} below). Note also that we have the continuous embedding $C_{\overline{\nu}}^{0,\gamma}(M^\epsilon) \hookrightarrow L_{-1}^2(M^\epsilon)$ (\cite{PR}, p. 168) so Lemma \ref{existence} applies to the Dirichlet problem for $L$ on $M^\epsilon$ with $F \in C_{\overline{\nu}}^{0,\gamma}(M^\epsilon)$. Since the proof of Lemma 3.3 in \cite{NS2} only relies on the estimates proved in Sec. 3.1 and 3.2, we have: (see also \cite{CHS})

\begin{lem}[\cite{NS2}, Lemma 3.3]\label{eigenexp}
    Let $\gamma \in (0,1)$, $F \in C_{\overline{\nu}}^{0,\gamma}(M^\epsilon)$, $\Psi \in C_c^{2,\gamma}(M^\epsilon)$, and let $U$ be the solution of the Dirichlet problem \eqref{DP} guaranteed by Lemma \ref{existence}. Then on each $C_i$, $i = 1,\ldots, N$, $U$ has a unique eigensection expansion
    \begin{equation*}
        U(r\omega) = \sum_{j = 1}^\infty \big( \alpha_j^i r^{\gamma_j^i(+)} + F_j^i(F)(r)\big)\eta_j^i(\omega) \text{ for } r \in (0,1]
    \end{equation*}
    which is convergent in $L^2(\Sigma_i)$, where the $\alpha_j^i$ satisfy $\sum_{j=1}^\infty|\alpha_j^i|^2 < \infty$ for each $i$. In addition, there is a universal constant $c > 0$ such that 
\begin{multline*}
    \norm{\sum_{j > J_i} \big( \alpha_j^ir^{\gamma_j^i(+)} + F_j^i(F)(r)\big)\eta_j^i}_{L^2(\Sigma_i)}^2 \leq \\
    c\Big(\int_0^1 s^{3-2\nu_i}\norm{F(s)}_{L^2(\Sigma_i)}^2 \, ds \Big)r^{2\nu_i} + \Big(\sum_{j > J_i} |\alpha_j^i|^2\Big)r^{2\nu_i}.
\end{multline*}
\end{lem}

Before we continue, we make a few remarks. First, note that since $\beta_j^i \equiv 0$ we can write $\gamma_j^i$ for $\gamma_j^i(+)$. The asymptotic expansion in Lemma \ref{eigenexp} will be used to set up a fixed point problem that will allow us to kill off the first $J_i$ $\alpha_j^i$ terms ($i$ fixed) in the eigensection expansion on each cone $C_i$ for our solution $V$ of \eqref{DPnu} below. The number of terms we need to kill off will depend on the rates $\nu_i$, with larger values of $\nu_i$ corresponding to a larger number of terms (Lemma \ref{lambdastar1}). This is due to how $J_i$ above is chosen depending on $\nu_i$. Since $F = H_0 + E(U)$ for $U \in \mathscr{K}$ (see \eqref{jacobi2}), we see from the proof of Lemma \ref{lambdastar1} that the first $J_i$ terms in the eigensection expansion on each $C_i$ for our solution $V$ of \eqref{DPnu} are rigid and depend continuously on $U$. As a consequence, our boundary data will depend continuously on $U$ (Proposition \ref{psilambda} and Lemma \ref{lambdastar1}). On the other hand, the asymptotic $L^2$ estimate will be key in making the Schauder estimates required to show that the unique solution to the Dirichlet problem \eqref{DPnu} remains in $C_\nu^{2,\gamma}(M^\epsilon)$ (Proposition \ref{n4reg}).

\vspace{-2.8pt}

\section{The Fixed Point Problem}
Our goal in this section is to prove Theorem \ref{mainthm}, excluding strict stability of the solution and the graphical case. The main difference between the cases $n \geq 6$ and $n = 3,4,5$ is that one needs a better $L^p$ estimate for the mean curvature vector on $M^\epsilon$ when $n = 3,4,5$ in order to obtain the required estimates to solve \eqref{DPnu} below as a fixed point problem. We proceed using Sec. 4 in \cite{NS2} as a guide.

\subsection{The Dirichlet Problem}
For $U \in C_\nu^{2,\gamma}(M^\epsilon)$ and $\Psi \in C_c^{\infty}(M^\epsilon)$, consider the Dirichlet problem
\begin{equation}
    \begin{cases}
        LV = H_0 + E(U) \text{ on } M^\epsilon \\
        V = \Psi \text{ on } \partial M^\epsilon. \label{DPnu}
    \end{cases}
\end{equation}
To apply the Schauder fixed point theorem, we need suitable $\Psi \in C_c^\infty(M^\epsilon)$ and an appropriate convex compact subset $\mathscr{K}$ of $C_\nu^{2,\gamma}(M^\epsilon)$ for which a solution $V$ of \eqref{DPnu} will be in $\mathscr{K}$ if $U \in \mathscr{K}$ also. Due to Lemma \ref{eigenexp}, we know that $V \in C_\nu^{2,\gamma}(M^\epsilon)$ if and only if $\alpha_j^i = 0$ when $j \leq J_i$. Thus, we need boundary conditions for which the $\alpha_j^i = 0$ for $j \leq J_i$.

Choose $\delta_0 < \frac{1}{20}$ so that $D_{5\delta_0}(q_k) \cap D_{5\delta_0}(q_l) = \emptyset$ for $k \neq l$ ($k,l = 1,\ldots 2I$) and choose $\epsilon$ small depending on $\delta_0$ so that the bridges are attached inside $D_{5\delta_0}(q_l)$ (e.g. $\epsilon < \frac{\delta_0}{10}$ will do). By making only minor adjustments in the proof of Proposition 4.1 in \cite{NS2}, we obtain:

\begin{prop}\label{psilambda}
    For any $p > 0$, any integer $K \geq 1$, and $i = 1,\ldots N$, there exists a $K$-parameter family $\{\Psi_\lambda\} \subset C^\infty(\Sigma_i)$ $(\lambda := (\lambda_1, \ldots, \lambda_K))$ satisfying the following properties:
    \begin{itemize}
    \itemsep = 3pt
        \item[(a)] $\Psi_\lambda$ is supported in $\bigcup_{q_l \in \Sigma_i}\big(\Sigma_i \setminus D_{5 \delta_0}(q_l)\big)$;
        \item[(b)] $\Psi_\lambda = \lambda_1 \eta_1^i + \cdots + \lambda_K \eta_K^i + \Psi_\lambda^\perp$, where $-\epsilon^p \leq \lambda_k \leq \epsilon^p$ and 
        $$
            \Psi_\lambda^\perp \in \Span\{\eta_1^i,\ldots,\eta_K^i\}^\perp;
        $$
        \item[(c)] $\norm{\Psi_\lambda}_{C^{2,\gamma}(\Sigma_i)} \leq c \epsilon^p$ where $\gamma \in (0,1)$;
        \item[(d)] $\max_{\lambda_k \in [-\epsilon^p, \epsilon^p]} \norm{\pdv{\Psi_\lambda}{\lambda_k}}_{C^3(\Sigma_i)} \leq c$ for $k = 1, \ldots, K$, where $c$ is universal. Thus, the parameter space is $\lambda \in \Omega_\epsilon := [-\epsilon^p,\epsilon^p]^K$.
    \end{itemize}
\end{prop}

Set $K = J_i$ and, for each $\Sigma_i$ $(i = 1,\ldots, N)$, let $\Psi_{\lambda^i}^i$ with $\lambda^i := (\lambda_1^i, \ldots, \lambda_{J_i}^i)$ be a $J_i$-parameter family of $C^\infty$ normal sections on $\Sigma_i$ as in Proposition \ref{psilambda}. Suppose also that 
\begin{equation}
    \begin{cases}
        p_0 = n - 3 - \frac{3}{n^4} \text{ when } n \geq 6 \\
        p_0 = 3 - \frac{3}{n^4} \text{ when } n = 3,4,5.
    \end{cases}\label{p}
\end{equation}
Set $J_0 := \sum_{i = 1}^N J_i$, let
$$
    \lambda := (\lambda_1^1, \ldots, \lambda_{J_1}^1, \lambda_1^2,\ldots, \lambda_{J_2}^2, \ldots, \lambda_1^N, \ldots, \lambda_{J_N}^N) \text{ for } \lambda \in \Omega_\epsilon,
$$
and let $\Psi_\lambda$ be the $J_0$-parameter family of normal sections on $\partial M^\epsilon$ given by:
\begin{equation}
    \Psi_\lambda(x) := \begin{cases}
                            \Psi_{\lambda^i}^i(x), \text{ } x \in \Sigma_i \text{ and } i = 1,\ldots, N \\
                            0, \text{ } x \in \partial M^\epsilon \setminus \cup_{i = 1}^N \Sigma_i.
                        \end{cases}\label{pl}
\end{equation}
We obtain the desired boundary conditions by extending $\Psi_\lambda$ to all of $M^\epsilon$ by multiplying the $\Psi_\lambda$ by a smooth cut-off function which is radial about $p_i$ on each of the cones $C_i$. 

Let $\varphi: [0, 1] \rightarrow [0,1]$ be a smooth function which is zero on $[0,\frac{3}{4}]$, increases from zero to one on $[\frac{3}{4},1]$, and whose derivatives of all orders tend to zero in the limit as $r \rightarrow 1^-$ and $r \rightarrow \frac{3}{4}^+$. In polar coordinates $x = r\omega$ on $C_i$, we can set 
$$
    \hat{\Psi}_{\lambda} := \varphi(r)\Psi_{\lambda}(\omega)
$$
for each $i = 1,\ldots, N$ and extend $\hat{\Psi}_\lambda$ to be zero on $\cup_l \Gamma_l(\epsilon)$. Then $\hat{\Psi}_\lambda \in C_c^\infty(M^\epsilon) \subset C_c^{2,\gamma}(M^\epsilon)$ and properties (c) and (d) in Proposition \ref{psilambda} are satisfied with $M^\epsilon$ replacing $\Sigma_i$ (\cite{NS2}, p. 622). The normal sections $\hat{\Psi}_\lambda$ are the desired boundary data for the Dirichlet problem and have support in $M_{\frac{3}{4}}^\epsilon$. Going forward, we will denote the normal sections $\hat{\Psi}_\lambda$ by $\Psi_\lambda$.

We can now solve the Dirichlet problem \eqref{DPnu} with $\Psi_\lambda$ replacing $\Psi$, where $U \in\mathscr{K}$ (with $\mathscr{K}$ to be determined), $\gamma \in (0,1)$, and $\lambda \in \Omega_\epsilon$. In what follows, we first define the sets $\mathscr{K} \subset C_\nu^{2,\gamma}(M^\epsilon)$ for which we can solve the problem \eqref{DPnu} as a fixed point problem using the Schauder fixed point theorem. The rest of the section is devoted to verifying the hypotheses of the Schauder fixed point theorem for the operator $T : \mathscr{K} \rightarrow C_\nu^{2,\gamma}(M^\epsilon)$ sending $U$ to the unique solution $V$ of \eqref{DPnu} for $\epsilon$ small, where $\Psi$ is replaced by $\Psi_{\lambda(U)}$ and $\Psi_{\lambda(U)}$ is unique for $U \in \mathscr{K}$ and depends continuously on $U$.

\subsection{Solving the Fixed Point Problem}
When $n \geq 6$, let $\beta_0,\beta_1,\ldots, \beta_K$ be an increasing sequence of numbers such that 
$$
    \beta_0 := 2 - \frac{2}{n} - \frac{1}{n^4} \text{ and } \beta_K := \frac{(n-1)}{2} - \frac{1}{n^4}.
$$
When $n = 3,4,5$, we assume
$$
    \beta_0 = 2 - \frac{1}{n^4}, \text{ } \beta_K = \frac{5}{2} - \frac{1}{n^4}.
$$
Suppose further that
$$
    \beta_l \geq \beta_{l + 1} - \frac{1}{n^2} \text{ for } l = 0,1,\ldots, K-1. 
$$
Then we can assume $K:= K(n) \leq n^3$ in each case since $n^2(\beta_K - \beta_0) < n^3$. Let 
$$
    \frac{1}{2} = \sigma_0 < \sigma_1 < \cdots < \sigma_{K-1} = \frac{3}{4}
$$
be a subdivision of $[\frac{1}{2}, \frac{3}{4}]$ such that $\sigma_l - \sigma_{l-1} \leq \frac{1}{n^3}$ and set
$$
     \mathcal{V}_0 := M_{\frac{3}{4}}^\epsilon, \text{ }  \mathcal{V}_K := S_{\frac{1}{4}, \frac{1}{2}}, \text{ } \mathcal{V}_l := S_{\sigma_{K - l - 1}, \sigma_{K - l}} \text{ } (l = 1,\ldots, K-1).
$$
For $\epsilon < \epsilon_0$, $\gamma \in (0,1)$, and $\nu_i \geq 2$, define the space $\mathscr{K}:= \mathscr{K}(n,\epsilon, \gamma, \nu)$ to be the set of all $U \in C_{\nu}^{2,\gamma}(M^\epsilon)$ such that $U = \Psi_\lambda$ on $\partial M^\epsilon$ for some $\lambda \in \Omega_\epsilon$ and $U$ satisfies the following estimates when $n \geq 6$: 
\begin{itemize}
\itemsep = 3pt
    \item[(i)] $\norm{\nabla^k U}_{C^0(M_{\frac{1}{2}}^\epsilon)} \leq \epsilon^{2 - k - \frac{2}{n} - \frac{1}{n^4}}$ for $k = 0,1,2$; 
    \item[(ii)] $|U|_{k,\gamma, M_{\frac{1}{2}}^\epsilon} \leq \epsilon^{2 - k - \frac{2}{n} - \frac{1}{n^4} - \gamma}$ for $k = 0,1,2$;
    \item[(iii)] $\norm{\nabla^kU}_{C^0(\mathcal{V}_l)} \leq \epsilon^{\beta_l - k}$ for $k = 0,1,2$ and $l = 1,\ldots, K$; 
    \item[(iv)] $|U|_{k,\gamma, \mathcal{V}_l} \leq \epsilon^{\beta_l - k - \gamma}$ for $k = 0,1,2$ and $l = 1, \ldots, K$; 
    \item[(v)] $\norm{\nabla^kU}_{C^0(S_{r,2r}^i)}r^{k - \nu_i} \leq \epsilon^{\frac{n-1}{2} - \frac{1}{n^4} - k}$ for $k = 0,1,2$, $r \in (0, \frac{1}{4}]$, and $i = 1,\ldots, N$;
    \item[(vi)] $|U|_{k,\gamma, S_{r,2r}^i}r^{k + \gamma - \nu_i} \leq \epsilon^{\frac{n-1}{2} - \frac{1}{n^4} - \gamma - k}$ for $k = 0,1,2$, $r \in (0,\frac{1}{4}]$, and $i = 1,\ldots, N$;
    \item[(vii)] $\norm{U}_{H_{1}^1(M^\epsilon)} \leq \epsilon^{\frac{n-1}{2} - \frac{1}{n^4}}$;
    \item[(viii)] $\norm{\nabla^2 U}_{L^2(M_{\frac{1}{2}}^\epsilon)} \leq \epsilon^{\frac{n^2-n-4}{2n} - \frac{1}{n^4}}$. 
\end{itemize}
When $n = 3,4,5$, we can assume (iii) and (iv) but we must replace the others with: 
\begin{itemize}
\itemsep = 3pt
     \item[(i$^\prime$)] $\norm{\nabla^k U}_{C^0(M_{\frac{1}{2}}^\epsilon)} \leq \epsilon^{2 - \frac{1}{n^4} - k }$ for $k = 0,1,2$; 
     
    \item[(ii$^\prime$)] $| U|_{k, \gamma, M_{\frac{1}{2}}^\epsilon} \leq \epsilon^{2- \frac{1}{n^4} - k - \gamma}$ for $k = 0,1,2$;
    
    \item[(v$^\prime$)] $\norm{\nabla^kU}_{C^0(S_{r,2r}^i)}r^{k - \nu_i} \leq \epsilon^{\frac{5}{2} - \frac{1}{n^4} - k}$ for $k = 0,1,2$, $r \in (0, \frac{1}{4}]$, and $i = 1,\ldots, N$;
    
    \item[(vi$^\prime$)] $|U|_{k,\gamma, S_{r,2r}^i}r^{k + \gamma - \nu_i} \leq \epsilon^{\frac{5}{2} - \frac{1}{n^4} - \gamma - k}$ for $k = 0,1,2$, $r \in (0,\frac{1}{4}]$, and $i = 1,\ldots, N$;

    \item[(vii$^\prime$)] $\norm{U}_{H_{1}^1(M^\epsilon)} \leq \epsilon^{\frac{5}{2} - \frac{1}{n^4}}$;
    
    \item[(viii$^\prime$)] $\norm{\nabla^2 U}_{L^2(M_{\frac{1}{2}}^\epsilon)} \leq \epsilon^{\frac{n-1}{2} - \frac{1}{n^4}}$ when $n = 4,5$;

    \item[(viii$^{\prime \prime}$)] $\norm{\nabla^2 U}_{L^2(M_{\frac{1}{2}}^\epsilon)} \leq \epsilon^{\frac{3}{2} - \frac{1}{3^4}}$ when $n = 3$.
\end{itemize}

The normal sections $\Psi_\lambda$ constructed in Sec. 4.1 are in $\mathscr{K}(n, \epsilon, \gamma, \nu)$ for small $\epsilon < \epsilon_0$, where $\epsilon_0$ is universal. Hence, the sets $\mathscr{K}$ are non-empty for small $\epsilon < \epsilon_0$. Note that $\mathscr{K}$ is a convex, closed, and bounded subset of $C_{\nu}^{2,\gamma}(M^\epsilon)$. To apply the Schauder fixed point theorem, we must show that solutions $V$ of \eqref{DPnu} remain in $\mathscr{K}$ for some unique $\lambda(U) \in \Omega_\epsilon$ depending continuously on $U$ provided $U \in \mathscr{K}$. 

The conditions on the $\mathcal{V}_l$ help bridge the estimates on $M_{2^{-1}}^\epsilon$ with those on $M^\epsilon \setminus M_{2^{-1}}^\epsilon$. Furthermore, the improved estimates (v), (v$^\prime$), (vi), and (vi$^\prime$) on $S_{r,2r}^i$ have been included since better control of $E(U)$ is needed near the singularities of $M^\epsilon$ to be sure solutions remain in $\mathscr{K}$. In fact, we need $E(U)$ to decay faster than a quadratic polynomial in $\epsilon$ near the singularities $p_i$. This is due to issues of compatibility between the estimates assumed on $M_{2^{-1}}^\epsilon$ and those on $S_{r_1,r_2}^i$ near the singularities $p_i$, and motivates the improved bridge constructions when $n = 3,4,5$ since the asymptotic estimates for $E(U)$ compete with the $L^p$ mean curvature bound in low dimensions. 

\begin{prop}\label{Eest1}
    If $U \in \mathscr{K}(n,\epsilon, \gamma,\nu)$, then there exists universal constants $c$ and $c(\delta)$ such that:
    \begin{itemize}
    \itemsep = 3pt
        \item[(a)] For each $x \in M_\delta^\epsilon$, we have 
        \begin{align*}
            |E(U)(x)| \leq c&(\delta)\Big(|U(x)|^2 + |\nabla U(x)|^2 \\
            &+ |U(x)||\nabla^2U(x)| + |\nabla U(x)|^2|\nabla^2U(x)|\Big);
        \end{align*}

        \item[(b)] $\norm{E(U)}_{C^{0,\gamma}(M_\delta^\epsilon)} \leq c(\delta) \epsilon^{2 - \frac{4}{n} - \frac{2}{n^4} - \gamma}$ when $n \geq 6$;

        \item[(b$^\prime$)] $\norm{E(U)}_{C^{0,\gamma}(M_\delta^\epsilon)} \leq c(\delta) \epsilon^{2 -  \frac{2}{n^4} - \gamma}$ when $n =3,4,5$;
        
        \item[(c)] $|E(U)|_{\gamma,\mathcal{V}_l} \leq c(\delta)(\epsilon^{2\beta_l - 2 - \gamma} + \epsilon^{3\beta_l - 4 - \gamma})$ for $l = 0,1, \ldots, K$;
        \item[(d)] For $x \in C_i$, $r = |x - p_i|$, and $i = 1,\ldots,N$, we have 
        \begin{align*}
            |E(U)(x)| \leq c &\Big(r^{-3}|U(x)|^2 + r^{-1}|\nabla U(x)|^2 \\
            &+ r^{-1} |U(x)||\nabla^2U(x)| + |\nabla U(x)|^2|\nabla^2U(x)|\Big);
        \end{align*}

        \item[(e)] $\norm{E(U)}_{C^0(S_{r,2r}^i)} \leq c \epsilon^{n - 3 - \frac{2}{n^4} }r^{2\nu_i - 3}$ for $0 < r \leq \frac{1}{4}$ when $n \geq 6$;

        \item[(e$^\prime$)] $\norm{E(U)}_{C^0(S_{r, 2r}^i)} \leq c\epsilon^{3 - \frac{2}{n^4}}r^{2\nu_i - 3}$ for $0 < r \leq \frac{1}{4}$ when $n = 3,4,5$;

        \item[(f)] $|E(U)|_{\gamma, S_{r,2r}^i} \leq c \epsilon^{n - 3 - \frac{2}{n^4} - \gamma}r^{2\nu_i - 3 - \gamma}$ for $0 < r \leq \frac{1}{4}$ when $n \geq 6$;
        
        \item[(f$^\prime$)] $|E(U)|_{\gamma, S_{r,2r}^i} \leq c \epsilon^{3 - \frac{2}{n^4} - \gamma}r^{2\nu_i - 3 - \gamma}$ for $0 < r \leq \frac{1}{4}$ when $n = 3,4,5$.
        \end{itemize}
\end{prop}
\begin{proof}
    Parts (a)-(c) follow from Lemma 5 in Sec. 4 of \cite{NS1} and the fact that $U \in \mathscr{K}$. Part (d) is proved similarly to part (a) and Proposition 4.2 in \cite{NS2} using \eqref{HU}, \eqref{EU}, and that $|r^{-1}U|$, $|\nabla U| <1$ for $|x - p_i| = r$  (see the proof of Lemma 5 in \cite{NS1}). Parts (e)-(f$^\prime$) follow immediately from part (d) and the fact that $U \in \mathscr{K}$. See also \cite{CHS} and \cite{Oo}.
\end{proof}

Observe that, if $U \in \mathscr{K}$, then $E(U) \in C_{\nu^\prime}^{0,\gamma}(M^\epsilon) \subset C_{\overline{\nu}}^{0,\gamma}(M^\epsilon)$ where $\nu^\prime := 2\nu - (3,\ldots, 3)$ and $\overline{\nu} = \nu - (\frac{3}{2}, \ldots, \frac{3}{2})$. Hence, we may apply Lemma \ref{eigenexp} to the solution $V$ of \eqref{DPnu} given by the lemma. Then on each $C_i$ ($i = 1,\ldots, N$), V has the asymptotic expansion in polar coordinates 
\begin{equation}
    V(r\omega) = \sum_{j = 1}^\infty \Big(\alpha_j^i r^{\gamma_j^i} + F_j^i(r)\Big) \eta_j^i(\omega), \label{V}
\end{equation}
where $F_j^i(r) := F_j^i(E(U))(r)$ (see \eqref{F}) and $\alpha_j^i := \alpha_j^i(U,\lambda)$. We define $E_{j}^i(U)$ the same as in \eqref{fij} with $E(U)$ replacing $F$. As previously discussed, we need to show that we can kill off the $\alpha_j^i$ for $j \leq J_i$. With this in mind, set
\begin{equation}
    \Lambda_U(\lambda) := (\alpha_1^1(\lambda), \ldots, \alpha_{J_1}^1(\lambda), \alpha_1^2(\lambda), \ldots, \alpha_{J_2}^2(\lambda), \ldots, \alpha_{1}^N(\lambda), \ldots, \alpha_{J_N}^N(\lambda)).\label{lambda}
\end{equation}
In Lemma \ref{lambdastar1} below, we will show that for each $U \in \mathscr{K}$ there is a unique $\lambda(U) \in \Omega_\epsilon$ such that $\Lambda_U(\lambda(U)) = 0$. This is accomplished by solving a fixed point problem for $\lambda \in \Omega_\epsilon$ using the contraction mapping theorem.

Lemma \ref{lambdastar1} is the most important step in the proof of Theorem \ref{mainthm}, and the estimates obtained in the proof of the lemma serve as a base for the majority of the estimates needed to apply the Schauder fixed point theorem. In addition, this step is where the improved mean curvature estimate on the bridge is most crucial.

\begin{lem}[\cite{NS2}, Lemma 4.1]\label{lambdastar1}
    For each $n \geq 3$, $\gamma \in (0,1)$, and $\nu_i \geq 2$, there exists a universal constant $\epsilon_0$ such that, if $\epsilon < \epsilon _0$ and $U \in \mathscr{K}(n,\epsilon,\gamma,\nu)$, there is a unique $\lambda(U) \in \Omega_\epsilon$ such that the solution $V$ of \eqref{DPnu} satisfies $\Lambda_U(\lambda(U)) = 0$. Moreover, $\lambda(U)$ depends continuously on $U$. 
\end{lem}
\begin{proof}
    We first motivate the proof. Fix $U \in \mathscr{K}$ and let $V := V_\lambda$ be the solution of \eqref{DPnu} guaranteed by Lemma \ref{existence} and Lemma \ref{eigenexp} for $\lambda \in \Omega_\epsilon$. On $\Sigma_i$, we have the $L^2(\Sigma_i)$ expansion
    \begin{equation}
        V \lvert_{\Sigma_i}= \sum_{j = 1}^\infty \Big(\alpha_j^i(\lambda) + F_j^i(1)\Big) \eta_j^i \text{ for } i = 1,\ldots, N, \label{v1}
    \end{equation}
while on $\Sigma_i \setminus \cup_l D_{5\delta_0}(q_l)$ we have $V = \Psi_{\lambda^i}^i$ due to the boundary conditions. The issue is that we do not know what $V$ is on $\Sigma_i \setminus \partial M^\epsilon$ (i.e. on $\cup_l D_{5\epsilon}(q_l)$ for $q_l \in \Sigma_i$) since we only know $V = \Psi_{\lambda^i}^i$ on $\Sigma_i \cap \partial M^\epsilon$. 
    
Take $\delta_0$ as in the paragraph preceding Proposition \ref{psilambda}, let $\epsilon$ small depending on $\delta_0$, and for each $i = 1,\ldots, N$ define $\overline{V}^i := \overline{V}_\lambda^i$ by
$$
    \overline{V}^i(x) = \begin{cases}
                            V(x) \text{ when } x \in D_{5 \epsilon}(q_l), \text{ } q_l \in \Sigma_i \\
                                0 \text{ when } x \in \Sigma_i \setminus \cup_l D_{5\epsilon}(q_l).
                            \end{cases}
$$
Notice that $\overline{V}^i \in C^{\infty}(\Sigma_i)$ since the support of $\Psi_\lambda$ does not intersect $D_{5\epsilon}(q_l)$ for any $l = 1,\ldots, 2I$. In addition, $\overline{V}^i$ has an $L^2(\Sigma_i)$ expansion 
\begin{equation*}
    \overline{V}^i = \sum_{j = 1}^{J_i} \overline{v}_j^i(\lambda) \eta_j^i + (\overline{V}^i)^\perp, \label{vlambda1}
\end{equation*}
where $(\overline{V}^i)^\perp$ is defined as in Proposition \ref{psilambda}(b). It follows that on $\Sigma_i$ we have $V = \Psi_{\lambda^i}^i + \overline{V}^i$. Explicitly,
\begin{equation}
    V\lvert_{\Sigma_i} = \sum_{j \leq J_i} \big(\lambda_j^i + \overline{v}_j^i(\lambda)\big) \eta_j^i + (\overline{V}^i + \Psi_{\lambda^i}^i)^\perp. \label{v2}
\end{equation}
Setting \eqref{v1} and \eqref{v2} equal, we find
$$
    \alpha_j^i(\lambda) + F_j^i(1) = \lambda_j^i + \overline{v}_j^i(\lambda) \text{ for } j = 1,\ldots, J_i, \text{ } i = 1,\ldots, N. 
$$
Rearranging terms, we see that $\alpha_j^i(\lambda) = 0$ if and only if 
\begin{equation}
    \lambda_j^i - F_j^i(1) + \overline{v}_j^i(\lambda) = 0 \Leftrightarrow \lambda - F + \overline{v}_\lambda = 0 \label{fixedpt}
\end{equation}
where 
$$
    F = (F_1^1(1), \ldots, F_{J_1}^1(1), \ldots, F_1^N(1),\ldots, F_{J_N}^N(1))
$$
and 
\begin{equation*}
    \overline{v}_\lambda := (\overline{v}_1^1(\lambda), \ldots, \overline{v}_{J_1}^1(\lambda), \ldots, \overline{v}_1^N(\lambda),\ldots, \overline{v}_{J_N}^N(\lambda)). \label{vlambda2}
\end{equation*}
In other words, \eqref{fixedpt} says that we need to solve a fixed point problem uniquely for $\lambda \in \Omega_\epsilon$, provided $\epsilon$ is small. To do so, we show that the map $\lambda \mapsto F - \overline{v}_\lambda$ is a contraction map from $\Omega_\epsilon$ into $\Omega_\epsilon$ when $\epsilon$ is small by showing that $|F - \overline{v}_\lambda| \leq c\epsilon^\theta$ for some $\theta > p_0$, with $p_0$ given by \eqref{p} and $c$ universal. Here, $|\cdot|$ denotes the Euclidean norm in $\mathbb{R}^{J_0}$. Note also that $F$ is independent of $\lambda$. The independence of $F$ from $\lambda$ will be important in proving the continuous dependence of $\lambda(U)$ on $U$. Going forward, the constants $c_1, c_2, \ldots$ are always universal. 

We begin with the estimate for $F$. Using \eqref{F}, for $i= 1,\ldots, N$ and $j = 1,\ldots ,J_i$ we have
    $$
        |F_j^i(1)| = \Big|\int_0^1 \tau^{1-n -2\gamma_j^i} \Big(\int_0^\tau s^{n-1 + \gamma_j^i} E_j^i(U(s)) \, ds\Big) d\tau \Big|.
    $$
Changing the order of integration and using \eqref{fij} along with the fact that $2-n-2\gamma_j^i < 0$ by \eqref{gamma+}, we can split the integral in $s$ into two integrals to find 
\begin{align*}
    |F_j^i(1)| &\leq c_1\Lambda \Big|\int_0^1 s^{1-\gamma_j^i} E_j^i(U(s)) \, ds \Big| \\
    &\leq c_2\Lambda \Bigg(\Big|\int_0^{\frac{1}{2}} s^{1-\gamma_j^i}E_j^i(U(s)) \, ds \Big| + \Big|\int_{\frac{1}{2}}^1 s^{n-1}\!\int_{\Sigma_i} \langle E(U(s\omega)), \eta_j^i(\omega) \rangle \, d\mathcal{H}^{n-1}(\omega) ds \Big| \Bigg)\\
    &\leq c_3 \Lambda \max_{\substack{s \in (0, \frac{1}{2}] \\ \omega \in \Sigma_i}}[s^{1-\nu_i}|E(U(s\omega))|] + c_4 \Lambda \int_{M_\frac{1}{2}^\epsilon}|E(U)| \, d\mathcal{H}^n, 
\end{align*}
where $\Lambda = -(2-n-2\gamma_1^i)^{-1} \geq -(2 - n - 2\gamma_j^i)^{-1}$. Using Proposition \ref{Eest1}(e) and (e$^\prime$), the definition of $\mathscr{K}$, and the fact that $\nu_i \geq 2$, we see that the the first term in the last line is bounded by:
\begin{equation}
    c_3 \Lambda \max_{\substack{s \in (0, \frac{1}{2}] \\ \omega \in \Sigma_i}}[s^{1-\nu_i}|E(U(s\omega))|] \leq
    \begin{cases}
        c_5\epsilon^{n - 3 - \frac{2}{n^4}} \text{ for } n\geq 6 \\
        c_5\epsilon^{3 - \frac{2}{n^4}} \text{ for } n = 3,4,5.
    \end{cases} \label{newest}
\end{equation}
For the second term, Proposition \ref{Eest1}(a) gives 
\begin{equation*}
    \int_{M_\frac{1}{2}^\epsilon }|E(U)|  \leq c_6 \int_{M_\frac{1}{2}^\epsilon} \Big(|U|^2 + |\nabla U|^2 + |U||\nabla^2U| + |\nabla U|^2|\nabla^2U|\Big). 
\end{equation*}
The largest term on the right-hand side in the last line is the third one. Applying the Cauchy-Schwarz inequality and the definition of $\mathscr{K}$, we get 
\begin{equation}
\norm{E(U)}_{L^1(M_{\frac{1}{2}}^\epsilon)} \leq
\begin{cases}
     c_7 \epsilon^{n - 1 - \frac{2}{n} - \frac{2}{n^4}} \text{ for } n \geq 6 \\
    c_7\epsilon^{\frac{n+4}{2} - \frac{2}{n^4}} \text{ for } n = 4,5 \\
    c_7\epsilon^{4 - \frac{2}{3^4}} \text{ for } n = 3.
\end{cases} \label{intest1}
\end{equation}
Hence, \eqref{newest} and \eqref{intest1} show 
\begin{equation}
   |F| \leq
   \begin{cases}
         c_8 \epsilon^{n-3 - \frac{2}{n^4}} \text{ for } n \geq 6 \\
        c_{8} \epsilon^{3 - \frac{2}{n^4}} \text{ for } n = 3,4,5. 
    \end{cases}\label{n43}
\end{equation}
    
    We now estimate $\overline{v}_\lambda$. We have 
    \begin{equation}
        |\overline{v}_j^i(\lambda)| = \sum_{q_l \in \Sigma_i} \Big|\int_{D_{5\epsilon}(q_l)}\langle V(\omega) , \eta_j^i \rangle\, d\mathcal{H}^{n-1}(\omega) \Big| \leq c_{9} \norm{V}_{C^0(M_{\frac{1}{2}}^\epsilon)} \epsilon^{n-1} \label{n5}
    \end{equation}
    as $\mathcal{H}^{n-1}(D_{5\epsilon}(q_l)) \leq c \epsilon^{n-1}$. Thus, we need a $C^0$ estimate for $V$. This is achieved via a cut-off argument. Let $\varphi \in C^2(M^\epsilon)$ be a (real-valued) cut-off function satisfying 
    \begin{equation*}
        \varphi(x) := \begin{cases}
                        1, \text{ } x \in M_{\frac{1}{3}}^\epsilon \\
                        0, \text{ } x \in M^\epsilon \setminus M_{\frac{1}{4}}^\epsilon.
                    \end{cases}
    \end{equation*}
    We first consider when $n \geq 4$. Using Corollary \ref{sest2} and the Minkowski inequality, we get 
    \begin{align}
        \norm{\varphi V}_{C^0(M_{\frac{1}{4}}^\epsilon)} &\leq c(\tau)\Big( \norm{H_0}_{L^{\frac{n}{2} + \tau}(M_{\frac{1}{4}}^\epsilon)} + \norm{E(U)}_{L^{\frac{n}{2} + \tau}(M_{\frac{1}{4}}^\epsilon)}  + \norm{V}_{L^2(M_{\frac{1}{4}}^\epsilon)} \nonumber \\
        &\qquad+ \norm{V}_{C^0(S_{\frac{1}{4}, \frac{1}{3}})} + \norm{\nabla V}_{L^{\frac{n}{2}+\tau}(S_{\frac{1}{4}, \frac{1}{3}})} + \norm{\Psi_\lambda}_{C^0(\partial M^\epsilon)} \Big). \label{extra2}
    \end{align}
    First, we estimate $\norm{V}_{L^2(M_{4^{-1}}^\epsilon)}$. In fact, we will show that $V$ satisfies the Sobolev estimates (vii) and (vii$^\prime$) in the definition of $\mathscr{K}$ for $\epsilon$ small enough.
    
    Using Lemma \ref{existence}, we find 
    \begin{align*}
        \norm{V}_{H_{1}^1(M^\epsilon)} &\leq d_0^{-1}\Big( \norm{\Psi_\lambda}_{H_{1}^1(M^\epsilon)} + \norm{L\Psi_\lambda}_{L_{-1}^2(M^\epsilon)} + \norm{H_0}_{L_{-1}^2(M^\epsilon)} + \norm{E(U)}_{L_{-1}^2(M^\epsilon)}\Big). 
    \end{align*}
    Since $\Psi_\lambda \in C_c^\infty(M^\epsilon)$ and $U \in \mathscr{K}$, Proposition \ref{psilambda}(c), \eqref{bridge1}, \eqref{Hsmall}, \eqref{Hsmall1}, and \eqref{p} imply
    \begin{equation}
        \norm{V}_{H_{1}^1} \leq
        \begin{cases}
            c_{10}\Big(\epsilon^{\frac{n-1}{2}}   + \norm{E(U)}_{L^2(M^\epsilon \setminus M_{\frac{1}{2}}^\epsilon)} + \norm{E(U)}_{L^2(M_{\frac{1}{2}}^\epsilon)}\Big), \text{ }  n \geq 6 \\
            c_{10} \Big( \epsilon^{\frac{5}{2}} + \norm{E(U)}_{L^2(M^\epsilon \setminus M_{\frac{1}{2}}^\epsilon)} + \norm{E(U)}_{L^2(M_{\frac{1}{2}}^\epsilon)}\Big), \text{ } n = 4,5.
        \end{cases}\label{n44}
    \end{equation}
    The estimate for the second term on the right-hand side is similar to the estimate for the first term in the estimate for $|F_j^i(1)|$ (i.e. the estimate for $s \in [0,\frac{1}{2}]$), and is bounded by 
    \begin{equation}
    \norm{E(U)}_{L^2(M^\epsilon \setminus M_{\frac{1}{2}}^\epsilon)} \leq
        \begin{cases}
             c_{11}\epsilon^{n-3 - \frac{2}{n^4}} \text{ for } n \geq 6 \\
            c_{11}\epsilon^{3 - \frac{2}{n^4}} \text{ for } n = 4,5.
        \end{cases} \label{estimate1}
    \end{equation}
    For the third term, we apply Proposition \ref{Eest1}(a): 
    \begin{align*}
        \norm{E(U)}_{L^2(M_{\frac{1}{2}}^\epsilon)}^2 \leq c_{12} \int_{M_{\frac{1}{2}}^\epsilon} \Big(|U|^4 + |\nabla U|^4 
        + |U|^2|\nabla^2U|^2 + |\nabla U|^4|\nabla^2U|^2\Big).
    \end{align*}
    When $n \geq 5$, the second term on the right-hand side is the largest, and when $n = 4$ the third term is the largest. We have:
    \begin{equation}
    \norm{E(U)}_{L^2(M_{\frac{1}{2}}^\epsilon)} \leq
    \begin{cases}
         c_{13} \epsilon^{\frac{n+1}{2} - \frac{2}{n} - \frac{2}{n^4}} \text{ for } n \geq 6 \\
         c_{13} \epsilon^{\frac{7}{2}  - \frac{2}{n^4}} \text{ for } n =4,5.
    \end{cases}\label{L2est}
    \end{equation}
    Substituting \eqref{estimate1} and \eqref{L2est} into \eqref{n44} yields 
    \begin{equation}
    \norm{V}_{H_{1}^1(M^\epsilon)} \leq
    \begin{cases} 
         c_{14} \epsilon^{\frac{n-1}{2}} \text{ for } n \geq 6 \\ 
         c_{14} \epsilon^{\frac{5}{2}} \text{ for } n =4,5.
    \end{cases} \label{n45}
    \end{equation}
    
    Next, we estimate the fourth and fifth terms on the right-hand side of \eqref{extra2}. Applying Corollary \ref{sest1} on a covering of $S_{\frac{1}{6}, \frac{5}{12}}$ with geodesic balls of radius at most $24^{-1}$ and using that $H_0$ is supported in $\cup_l \Gamma_l(\epsilon)$ due to the minimality of the cones $C_i$, we obtain
    \begin{equation}
        \norm{V}_{C^0(S_{\frac{1}{6}, \frac{5}{12}})} \leq c_{15} \Big(\norm{V}_{L^2(S_{\frac{1}{12}, \frac{1}{2}})} + \norm{E(U)}_{L^n(S_{\frac{1}{12}, \frac{1}{2}})}\Big). \label{est2}
    \end{equation}
    Since the second term on the right-hand side satisfies the same estimates in \eqref{n43}, \eqref{n45} gives
    \begin{equation}
    \norm{V}_{C^0(S_{\frac{1}{6}, \frac{5}{12}})} \leq
    \begin{cases}
         c_{16} \epsilon^{\frac{n-1}{2}} \text{ for } n \geq 6 \\
        c_{16} \epsilon^{\frac{5}{2}} \text{ for } n = 4,5.
    \end{cases} \label{extra}
    \end{equation}
    Since $S_{\frac{1}{4},\frac{1}{3}} \subset \subset  S_{\frac{1}{6}, \frac{5}{12}}$, this gives an estimate for the fourth term on the right-hand side in \eqref{extra2}. We can now apply Proposition \ref{sest4}(a) and a covering argument to conclude
    \begin{equation*}
        \norm{\nabla V}_{C^0(S_{\frac{1}{4}, \frac{1}{3}})} \leq c_{17}\Big( \norm{V}_{C^0(S_{\frac{1}{6}, \frac{5}{12}})} + \norm{E(U)}_{C^{0,\gamma}(S_{\frac{1}{6}, \frac{5}{12}})}\Big).
    \end{equation*}
    The first term on the right-hand side is estimated by \eqref{extra}. For the second, first note that
    \begin{equation*}
    \norm{E(U)}_{C^{0,\gamma}(S_{\frac{1}{6}, \frac{5}{12}})} \leq
        \begin{cases}
             c_{18} \epsilon^{n - 3 - \frac{2}{n^4} - \gamma} \text{ for } n \geq 6 \\
             c_{18} \epsilon^{3 - \frac{2}{n^4} - \gamma} \text{ for } n = 4,5
        \end{cases}
    \end{equation*}
    due to Proposition \ref{Eest1}(e)-(f$^\prime$). Thus,
    \begin{equation}
     \norm{\nabla V}_{C^0(S_{\frac{1}{4}, \frac{1}{3}})} \leq
        \begin{cases}
            c_{19} \epsilon^{n-4 - \frac{2}{n^4}} \text{ for } n \geq 6 \\
            c_{19} \epsilon^{2 - \frac{2}{n^4}} \text{ for } n =4,5.
        \end{cases}\label{nabla}
    \end{equation}
    We have:
    \begin{equation*}
        \norm{\nabla V}_{L^{\frac{n}{2} +\tau}(S_{\frac{1}{4}, \frac{1}{3}})} \leq c_{20} \norm{\nabla V}_{C^0(S_{\frac{1}{4}, \frac{1}{3}})}^{\frac{n-4}{n+2\tau} + \frac{2\tau}{n + 2\tau}}\norm{\nabla V}_{L^2(S_{\frac{1}{4}, \frac{1}{3}})}^{\frac{4}{n+2\tau}}.
    \end{equation*}
    Therefore, \eqref{n45} and \eqref{nabla} give
    \begin{equation*}
     \norm{\nabla V}_{L^{\frac{n}{2} +\tau}(S_{\frac{1}{4}, \frac{1}{3}})} \leq
        \begin{cases}
            c_{21}\epsilon^{(n-4 - \frac{2}{n^4})\frac{n-4}{n+2\tau} + (n-4 - \frac{2}{n^4})\frac{2\tau}{n + 2\tau} + \frac{4(n-1)}{2n +4\tau}} \text{ for } n \geq 6 \\
             c_{21}\epsilon^{(2 - \frac{2}{n^4})\frac{n-4}{n+2\tau} + (2 - \frac{2}{n^4})\frac{2\tau}{n + 2\tau} + \frac{20}{2n +4\tau}} \text{ for } n = 4,5.
        \end{cases}
    \end{equation*}
    Choosing $\tau$ small depending on $n$ yields 
    \begin{equation}
         \norm{\nabla V}_{L^{\frac{n}{2}+\tau}(S_{\frac{1}{4}, \frac{1}{3}})} \leq c_{22}\epsilon^{2} \text{ for } n \geq 4. \label{nablaV}
    \end{equation}
  
    The remaining terms to estimate on the right-hand side of \eqref{extra2} are the first and second terms (i.e. the $L^{\frac{n}{2} + \tau}$ norms of $H_0$ and $E(U)$, respectively). By Proposition \ref{Eest1}(a) and the Minkowski inequality, for each $n \geq 4$ we have
    \begin{align*}
        \norm{E(U)}_{L^{\frac{n}{2} + \tau}(M_{\frac{1}{4}}^\epsilon)} \leq c_{23}\Bigg(\Big(&\int_{M_{\frac{1}{4}}^\epsilon}|U|^{n + 2\tau}  \Big)^{\frac{2}{n +2\tau}} + \Big( \int_{M_{\frac{1}{4}}^\epsilon}|\nabla U|^{n + 2\tau}  \Big)^{\frac{2}{n+2\tau}} \\
        &+ \Big(\int_{M_{\frac{1}{4}}^\epsilon}|\nabla U|^{n + 2\tau} |\nabla^2 U|^{\frac{n}{2} + \tau} \Big)^{\frac{2}{n +2\tau}} \\
        &+ \Big(\int_{M_{\frac{1}{4}}^\epsilon}|U|^{\frac{n}{2} + \tau} |\nabla^2 U|^{\frac{n}{2} + \tau}  \Big)^{\frac{2}{n+2\tau}}\Bigg).
    \end{align*}
    The worst term is the third term. We have 
    \begin{align*}
        \Big(\int_{M_{\frac{1}{4}}^\epsilon}|\nabla U|^{n + 2\tau} |\nabla^2 U|^{\frac{n}{2} + \tau} \Big)^{\frac{2}{n+2\tau}} &\leq \norm{\nabla U}_{C^0(M_{\frac{1}{4}}^\epsilon )}^2\norm{\nabla^2 U}_{C^0(M_{\frac{1}{4}}^\epsilon)}^{\frac{n-4}{n + 2\tau} + \frac{2\tau}{n + 2\tau}} \norm{\nabla^2 U}_{L^2(M_{\frac{1}{4}}^\epsilon)}^{\frac{4}{n + 2\tau}} \\
        &\leq \epsilon^{\frac{5}{2}} \text{ for each } n \geq 4,
    \end{align*}
    where the second inequality follows from taking $\tau$ small depending on $n$ and using that $U \in \mathscr{K}$. We conclude
    \begin{equation}
        \norm{E(U)}_{L^{\frac{n}{2} + \tau}(M_{\frac{1}{4}}^\epsilon)} \leq c_{24} \epsilon^{\frac{5}{2}} \text{ for } n \geq 4. \label{E(U)}
    \end{equation}
    Using \eqref{Hsmall} and \eqref{Hsmall1}, the mean curvature term is estimated by:
    \begin{equation}
    \norm{H_0}_{L^{\frac{n}{2} + \tau}(M_{\frac{1}{4}}^\epsilon)} \leq
        \begin{cases}
             c_{25}\epsilon^{2 - \frac{2}{n} - \tau} \leq c_{26} \epsilon^{2- \frac{2}{n} - \frac{1}{2n^4}} \text{ when } n \geq 6 \\
             c_{27}\epsilon^{2 - 4\tau} \leq c_{28} \epsilon^{2 - \frac{1}{2(5^4)}} \text{ when } n = 5 \\
             c_{29} \epsilon^{\frac{5}{2} - \frac{1}{2(4^4)}} \text{ when } n = 4,
        \end{cases}\label{H04}
    \end{equation}
    where we have chosen $\tau$  small depending on $n$ in each case.
    Inserting the \eqref{E(U)} and \eqref{H04} into \eqref{extra2} for each $n \geq 4$, we conclude
    \begin{equation}
    \norm{V}_{C^0(M_{\frac{1}{3}}^\epsilon)} \leq
        \begin{cases}
             c_{30} \epsilon^{2 - \frac{2}{n} - \frac{1}{2n^4}} \text{ for } n \geq 6 \\
            c_{31} \epsilon^{2 - \frac{1}{2n^4}} \label{n49} \text{ for } n = 4,5.
        \end{cases}
    \end{equation}

    The case $n = 3$ is simpler. Define $\varphi \in C^2(M^\epsilon)$ as before. Then Corollary \ref{sest2} gives
     \begin{align}
        \norm{\varphi V}_{C^0(M_{\frac{1}{4}}^\epsilon)} &\leq c(\tau)\Big( \norm{H_0}_{L^{2}(M_{\frac{1}{4}}^\epsilon)} + \norm{E(U)}_{L^{2}(M_{\frac{1}{4}}^\epsilon)}  + \norm{V}_{L^2(M_{\frac{1}{4}}^\epsilon)} \nonumber \\
        &\qquad+ \norm{\nabla V}_{L^{2}(M_{\frac{1}{4}}^\epsilon)} + \norm{\Psi_\lambda}_{C^0(\partial M^\epsilon)} \Big). \label{newextra2}
    \end{align}
    Note that we can use the argument that was used to prove the estimate \eqref{L2est} to estimate the second term on the right-hand side in \eqref{newextra2}. Hence,
    \begin{equation}
         \norm{E(U)}_{L^2(M_{\frac{1}{4}}^\epsilon)} \leq c_{32} \epsilon^{\frac{7}{2} -\frac{2}{3^4}}. \label{n2}
    \end{equation}
    Both the $\Psi_\lambda$ term and the mean curvature term are bounded by $c_{33}\epsilon^{\frac{5}{2}}$ due to \eqref{Hsmall1} and \eqref{p}. As a consequence of this, \eqref{p}, and \eqref{n2} we can follow the argument used to prove \eqref{n45} to conclude 
      \begin{equation}
        \norm{V}_{H_{1}^1(M^\epsilon)} \leq 
         c_{34} \epsilon^{\frac{5}{2}}. \label{n3}
    \end{equation}
    Therefore, the third and fourth terms on the right-hand side of \eqref{newextra2} are bounded by the same quantity in \eqref{n3} and $V$ satisfies the estimate (vii$^\prime$) for $\epsilon$ small when $n = 3$. It follows that 
    \begin{equation}
        \norm{V}_{C^0(M_{\frac{1}{3}}^\epsilon)} \leq c_{35} \epsilon^{\frac{5}{2}}. \label{n4}
    \end{equation}
    Substituting \eqref{n49} and \eqref{n4} into \eqref{n5}, we find
    \begin{equation}
    \norm{\overline{v}_\lambda}_{C^0(\Omega_\epsilon)}   \leq
        \begin{cases}
             c_{36} \epsilon^{n + 1 - \frac{2}{n} - \frac{1}{2n^4}} \text{ for } n \geq 6\\
             c_{36} \epsilon^{n + 1 - \frac{1}{2n^4}} \text{ for } n = 4,5 \\
             c_{36}\epsilon^{\frac{9}{2}} \text{ for } n = 3.
        \end{cases} \label{n47}
    \end{equation}
    Notice also that \eqref{n49} and \eqref{n4} together show that $V$ satisfies the estimates (i) and (i$^\prime$) in the definition of $\mathscr{K}$ for $\epsilon$ small when $k = 0$ and each $n \geq 3$. Combining \eqref{n43} and \eqref{n47} shows 
    \begin{equation*}
    \norm{F - \overline{v}_\lambda}_{C^0(\Omega_\epsilon)} \leq
    \begin{cases}
         c_{37} \epsilon^{n - 3 - \frac{2}{n^4}} \text{ for } n \geq 6 \\
         c_{37} \epsilon^{3 - \frac{2}{n^4}} \text{ for } n = 3,4,5.
    \end{cases}
    \end{equation*}
    Thus, we may take $\theta = n - 3 - \frac{2}{n^4}$ when $n \geq 6$, $\theta = 3 - \frac{2}{n^4}$ when $n = 3,4,5$, and $c = c_{37}$ to conclude that $\lambda \mapsto F - \overline{v}_\lambda$ maps $\Omega_\epsilon$ into $\Omega_\epsilon$ for small $\epsilon < c_{37}^{-n^4}$ and each $n \geq 3$. Since Lemma \ref{contract} below holds independent of dimension, this map is a contraction. The continuous dependence of $\lambda(U)$ on $U$ follows immediately from the fact that $F$ is independent of $\lambda$, $F$ depends continuously on $U \in \mathscr{K}$ (since $E(U)$ is continuous as a map $\mathscr{K} \rightarrow C_{\overline{\nu}}^{0,\gamma}(M^\epsilon)$ by \eqref{HU}, \eqref{EU}, and Proposition \ref{Eest1}; see \eqref{F} also), and Lemma \ref{contract}. 
\end{proof}

\begin{lem}[\cite{NS2}, Proposition 4.2]\label{contract}
    For each $n \geq 3$ there is a universal constant $\epsilon_0$ such that, for all $\epsilon < \epsilon_0$, we have 
    $$
        \max_{\lambda \in \Omega_\epsilon}|D\overline{v}_\lambda| < 1. 
    $$
    Here, $D\overline{v}_\lambda$ is the Jacobian matrix of the map $\lambda \mapsto \overline{v}_\lambda$ of $\Omega_\epsilon$ into $\mathbb{R}^{J_0}$.
\end{lem}
\begin{proof}
    The proof is identical to Proposition 4.3 in \cite{NS2}, with the estimates in Sec. 3 above replacing those in \cite{NS2}.
\end{proof}

Next, we show that the unique solution given by Lemma \ref{existence} and Lemma \ref{lambdastar1} is in $C_\nu^{2,\gamma}(M^\epsilon)$ for small enough $\epsilon$.

\begin{prop}[\cite{NS2}, Proposition 4.4]\label{n4reg}
    Suppose $U \in \mathscr{K}(n,\epsilon, \gamma, \nu)$ for $n \geq 3$, $\gamma \in (0,1)$, $\nu_i \geq 2$, and $\epsilon 
< \epsilon_0$. Then the unique solution $V$ of \eqref{DPnu} with $\Psi = \Psi_{\lambda(U)}$ given by Lemma \ref{existence} and Lemma \ref{lambdastar1} is in $C_{\nu}^{2,\gamma}(M^\epsilon)$.
\end{prop}
\begin{proof}
    Combining Lemma \ref{eigenexp} and Lemma \ref{lambdastar1} shows that on each $C_i$ the solution $V$ given by Lemma \ref{existence} has the Fourier expansion in polar coordinates:
    \begin{equation*}
        V(r\omega) = \sum_{j \leq J_i} F_j^i(r) \eta_j^i(\omega) + \sum_{j > J_i}\Big(\alpha_j^i r^{\gamma_j^i} + F_j^i(r) \Big) \eta_j^i(\omega)
    \end{equation*}
    for unique $\alpha_j^i$, $i = 1,\ldots, N$ and $j = J_i + 1, \ldots$. We need to estimate $\norm{V(r\omega)}_{L^2(\Sigma_i)}$ for $r \leq \frac{1}{2}$. Using Proposition \ref{Eest1}(e), we compute for $r \in (0,\frac{1}{2}]$, $j \leq J_i$, and $n \geq 6$:
     \begin{align*}
        |F_j^i(r)| &\leq \Big|r^{\gamma_j^i} \int_0^r \tau^{1-n-2\gamma_j^i} \int_0^\tau s^{n-1 + \gamma_j^i}E_j^i(U(s)) \, ds d\tau \Big| \\
        &\leq c_1 \epsilon^{n-3 - \frac{2}{n^4}} r^{\gamma_j^i} \int_0^r \tau^{1-n-2\gamma_j^i} \int_0^\tau s^{n-4 + \gamma_j^i + 2\nu_i} \, ds d\tau \\
        &\leq c_2 \epsilon^{n-3   - \frac{2}{n^4}} r^{\gamma_j^i} \int_0^r \tau^{-2 - \gamma_j^i + 2\nu_i} \, d\tau \\ 
        &\leq c_3 \epsilon^{n-3  - \frac{2}{n^4}} r^{2\nu_i -1}.
    \end{align*}
    When $n = 3,4,5$, the same computation using Proposition \ref{Eest1}(e$^\prime$) shows
    $$
        |F_j^i(r)| \leq c_4\epsilon^{3 - \frac{2}{n^4}}r^{2\nu_i - 1}.
    $$
    Hence, for each $n \geq 3$ and $r \in (0,\frac{1}{2}]$ we have
    \begin{equation}
     \norm{\sum_{j \leq J_i} F_j^i(r) \eta_j^i(\omega)}_{L^2(\Sigma_i)} \leq
     \begin{cases}
        c_5 \epsilon^{n-3 - \frac{2}{n^4}} r^{2\nu_i -1} \text{ when } n \geq 6 \\
        c_5 \epsilon^{3 - \frac{2}{n^4}} r^{2\nu_i -1} \text{ when } n = 3,4,5.
    \end{cases}\label{n48}
    \end{equation}
    
    We need to estimate the right-hand side in Lemma \ref{eigenexp}. Recall that
    \begin{multline*}
        \norm{\sum_{j > J_i} \Big(\alpha_j^ir^{\gamma_j^i} + F_j^i(r)\Big)\eta_j^i(\omega)}_{L^2(\Sigma_i)} \leq \\
        \Bigg(\sum_{j > J_i} |\alpha_j^i|^2 \Bigg)^{\frac{1}{2}} r^{\nu_i} + c \Big(\int_0^1 s^{3 - 2\nu_i} \norm{E(U(s))}_{L^2(\Sigma_i)}^2 \, ds \Big)^{\frac{1}{2}} r^{\nu_i}.
    \end{multline*}
    Setting $r=1$ and noting that $F_j^i(1)$ vanishes for $j > J_i$ (see \eqref{F}) shows
    \begin{equation}
        V \lvert_{\Sigma_i} = \sum_{j \leq J_i} F_j^i(1) \eta_j^i + \sum_{j > J_i} \alpha_j^i \eta_j^i \text{ and }  V \lvert_{\Sigma_i} = \Psi_{\lambda^i}^i + \overline{V}^i. \label{n410}
    \end{equation}
    Since $\overline{V}^i = V$ on $D:= \cup_l D_{5\epsilon}(q_l) \cap \Sigma_i$, is zero on $\Sigma_i \setminus D$, and $\mathcal{H}^{n-1}(D) \leq c \epsilon^{n-1}$ the estimates \eqref{n49} and \eqref{n4} along with the second expression for $V \lvert_{\Sigma_i}$ in \eqref{n410} give
    \begin{equation*}
    \norm{V \lvert_{\Sigma_i}}_{L^2(\Sigma_i)} \leq
    \begin{cases}
         c_6\epsilon^{\frac{n+3}{2} - \frac{2}{n} - \frac{1}{2n^4}} \text{ for } n \geq 8 \\
         c_6\epsilon^{n-3 - \frac{3}{n^4}} \text{ for } n = 6,7 \\
         c_6\epsilon^{3 - \frac{3}{n^4}} \text{ for } n = 3,4,5.
    \end{cases}
    \end{equation*}
    Writing 
     $$
        \sum_{j > J_i} \alpha_j^i \eta_j^i = V \lvert_{\Sigma_i} - \sum_{j \leq J_i} F_j^i(1) \eta_j^i
    $$
    and using \eqref{n43}, \eqref{n410}, and that the $\eta_j^i$ $(j = 1,2,\ldots)$ are an $L^2(\Sigma_i)$ orthonormal set for each $i = 1,\ldots, N$, we find
    \begin{equation}
       \Bigg(\sum_{j > J_i} |\alpha_j^i|^2 \Bigg)^{\frac{1}{2}} \leq  
       \begin{cases}
         c_7 \epsilon^{\frac{n+3}{2} - \frac{2}{n} - \frac{1}{2n^4}} \text{ for } n \geq 8 \\
         c_7 \epsilon^{n-3 - \frac{3}{n^4}} \text{ for } n = 6,7 \\
        c_7 \epsilon^{3 - \frac{3}{n^4}} \text{ for } n = 3,4,5.
    \end{cases}\label{n411}
    \end{equation}
 
    By the argument used to estimate \eqref{n43} (i.e. splitting the integral), we get
    \begin{equation*}
        \Big(\int_0^1 s^{3 - 2\nu_i} \norm{E(U(s\omega))}_{L^2(\Sigma_i)}^2 \, ds \Big)^{\frac{1}{2}} \leq c_8 \Bigg( \norm{E(U)}_{L^2(M_{\frac{1}{2}}^\epsilon)} + \max_{\substack{0 < s \leq \frac{1}{2} \\ \omega \in \Sigma_i}}\Big(s^{\frac{3}{2} - \nu_i} |E(U(s\omega))|\Big)\Bigg). 
    \end{equation*}
    Applying \eqref{L2est} and \eqref{n2} to the first term on the right-hand side and Proposition (e) and (e$^\prime$) in the second while combining these with \eqref{n48} and \eqref{n411}, we deduce that for $r \in (0, \frac{1}{2}]$, and each $i = 1,\ldots, N$
    \begin{equation}
    \norm{V(r\omega)}_{L^2(\Sigma_i)} \leq
    \begin{cases}
         c_9 \epsilon^{\frac{n+1}{2} -\frac{2}{n} - \frac{1}{2n^4}} r^{\nu_i} \text{ for } n \geq 7 \\
        c_9 \epsilon^{3 - \frac{3}{n^4}} r^{\nu_i} \text{ for } n = 3,4,5,6. 
    \end{cases}\label{n412}
    \end{equation}
    We can now use \eqref{n412} to get a $C^0$ estimate for $V$ near the $p_i$. Note that by Corollary \ref{sest1} together with a covering argument with geodesic balls of radius at most $\frac{r}{4}$, for $0 < r \leq \frac{1}{8}$ we find
    \begin{multline}
        \norm{V}_{C^0(S_{r,2r}^i)} \leq \\c_{10}\Bigg(r^{-\frac{n}{2}}\Big(\int_{\frac{r}{2}}^{4r} s^{n-1} \norm{V(s\omega)}_{L^2(\Sigma_i)}^2ds\Big)^{\frac{1}{2}} + r \norm{E(U)}_{L^n(S_{\frac{r}{2}, 4r})}\Bigg). \label{n413}
    \end{multline}
    By \eqref{n412} and Proposition \ref{Eest1}(e)-(e$^\prime$), both terms are bounded by a similar quantity as in \eqref{n412} so we have 
    \begin{equation}
        \norm{V}_{C^0(S_{r,2r}^i)} \leq 
        \begin{cases}
            c_{11} \epsilon^{\frac{n+1}{2} -\frac{2}{n} - \frac{1}{2n^4}} r^{\nu_i} \text{ for } n \geq 7 \\
            c_{11} \epsilon^{3 - \frac{3}{n^4}}r^{\nu_i} \text{ for } n = 3,4,5,6.
        \end{cases} \label{c0}
    \end{equation}
    
   By covering $S_{r,2r}^i$ by geodesic balls of radius at most $\frac{r}{4}$ once more and applying Proposition \ref{sest4} together with \eqref{c0}, Proposition \ref{Eest1}(f), and the fact that $\nu_i \geq 2$ for each $i$, we find that for $n \geq 7$, each $r \in (0,\frac{1}{8}]$, and each $i = 1,\ldots, N$
   \begin{equation}
        \begin{cases}
            \norm{V}_{C^0(S_{r,2r}^i)} \leq c_{12} \epsilon^{\frac{n+1}{2} - \frac{2}{n} - \frac{1}{2n^4}}r^{\nu_i} \\
            \norm{\nabla^k V}_{C^0(S_{r,2r}^i)} \leq c_{13} \epsilon^{\frac{n-1}{2} - \frac{2}{n} - \frac{1}{2n^4}}r^{\nu_i - k}  \text{ for } k = 1,2 \\
            |V|_{k,\gamma, S_{r,2r}^i} \leq c_{14} \epsilon^{\frac{n+1}{2} - \frac{2}{n} - \frac{1}{2n^4} - \gamma}r^{\nu_i -k - \gamma} \text{ for } k = 0,1,2.
        \end{cases}\label{singest}
    \end{equation}
    When $n = 3,4,5,6$, we get
    \begin{equation}
        \begin{cases}
            \norm{V}_{C^0(S_{r,2r}^i)} \leq c_{15} \epsilon^{3 - \frac{3}{n^4}}r^{\nu_i} \\
            \norm{\nabla^k V}_{C^0(S_{r,2r}^i)} \leq c_{16} \epsilon^{2 - \frac{3}{n^4}}r^{\nu_i - k}  \text{ for } k = 1,2 \\
            |V|_{k,\gamma, S_{r,2r}^i} \leq c_{17} \epsilon^{3 - \frac{3}{n^4} - \gamma}r^{\nu_i -k - \gamma} \text{ for } k = 0,1,2.
        \end{cases}\label{n414}
    \end{equation}
    We have therefore shown that $V \in C_\nu^{2,\gamma}(M^\epsilon)$ for each $n \geq 3$. Note that \eqref{n414} also shows that, for $r \in (0, \frac{1}{8}]$, $V$ satisfies (v) and (vi) in the definition of $\mathscr{K}$ when $n \geq 6$ and satisfies (v$^\prime$) and (vi$^\prime$) when $n = 3,4,5$.
\end{proof}

Using Lemma \ref{existence}, Lemma \ref{lambdastar1}, and Proposition \ref{n4reg} we can define a mapping 
$$
    T: \mathscr{K}(n,\epsilon, \gamma, \nu) \rightarrow C_\nu^{2,\gamma}(M^\epsilon) \text{ for } \epsilon < \epsilon_0 \text{ and } \gamma \in (0,1),
$$
where $T(U)$ is the solution $V$ of \eqref{DPnu} with $\Psi_{\lambda(U)}$ replacing $\Psi_\lambda$. One can show that $T$ is continuous in the $C_\nu^{2,\gamma}$ topology as in \cite{NS2}. To see this, observe that the map $U \mapsto \lambda(U)$ is continuous as a map $\mathscr{K} \rightarrow \Omega_\epsilon$ and $\lambda \mapsto \Psi_\lambda$ is continuous as a map $\Omega_\epsilon \rightarrow C_{\nu}^{2,\gamma}(M^\epsilon)$ by the construction of $\Psi_\lambda$ and Proposition \ref{psilambda}(d). Furthermore, the map $U \mapsto H_0 + E(U)$ is continuous as a map $\mathscr{K} \rightarrow C_{\overline{\nu}}^{0,\gamma}(M^\epsilon)$. The continuity of $T$ then follows from the global and local $C^0$ bounds (Corollaries \ref{sest2} and \ref{sest1}), the global and local Schauder estimates (Propositions \ref{sest3} and \ref{sest4}), the asymptotic estimate \eqref{n413}, and Lemma \ref{existence}. The next lemma shows that $T$ maps $\mathscr{K} \rightarrow \mathscr{K}$ for small $\epsilon$, and is the final step in proving existence of a small perturbation of $M^\epsilon$ that is minimal in $\mathbb{R}^{n+m+1}$.
\begin{lem}[see \cite{NS2}, Lemma 4.2]\label{fixedpt2}
    For any $n \geq 3$, $\gamma \in (0,1)$, and $\nu_i\geq 2$, there is a universal constant $\epsilon_0 > 0$ such that
    $$
        T: \mathscr{K}(n,\epsilon, \gamma, \nu) \rightarrow \mathscr{K}(n, \epsilon, \gamma, \nu) \text{ for all } \epsilon < \epsilon_0.
    $$
\end{lem}
\begin{proof}
We need to show that the solution $V$ to the Dirichlet problem in Proposition \ref{n4reg} is in $\mathscr{K}$ for $\epsilon < \epsilon_0$, where $\epsilon_0$ is universal. More precisely, we need to show
 \begin{itemize}
 \itemsep = 3pt
        \item[(a)] For $k = 0,1,2$ 
        \begin{equation*}
            \begin{cases}
                \norm{\nabla^k V}_{C^0(M_{\frac{1}{2}}^\epsilon)} \leq c \epsilon^{\theta - k} \\
                |V|_{k,\gamma, M_{\frac{1}{2}}^\epsilon}\leq c \epsilon^{\theta - k - \gamma},
            \end{cases}
        \end{equation*}
        where $\theta > 2 - \frac{2}{n} - \frac{1}{n^4}$ when $n \geq 6$ and $\theta > 2 - \frac{1}{n^4}$ when $n =3,4,5$;

        \item[(b)] For $\rho_l > \beta_l$, $l = 1,\ldots, K$, and $k = 0,1,2$
        \begin{equation*}       
        \begin{cases}\norm{\nabla^k V}_{C^0(\mathcal{V}_l)} \leq c \epsilon^{\rho_l - k} \\
            |V|_{k,\gamma, \mathcal{V}_l} \leq c \epsilon^{\rho_l - k - \gamma};
        \end{cases}
        \end{equation*}

        \item[(c)] For $i = 1,\ldots, N$, $k = 0,1,2$, and $0 < r \leq \frac{1}{4}$
        \begin{equation*}
            \begin{cases}
                \norm{\nabla^k V}_{C^0(S_{r,2r}^i)}r^{k - \nu_i} c \epsilon^{\tau - k} \\
                |V|_{k,\gamma, S_{r,2r}^i}r^{k + \gamma - \nu_i} \leq c \epsilon^{\tau - k - \gamma},
            \end{cases}
        \end{equation*}
        where $\tau > \frac{n-1}{2} - \frac{1}{n^4}$ when $n \geq 6$ and $\tau > \frac{5}{2} - \frac{1}{n^4}$ when $n = 3,4,5$;

        \item[(d)] We have 
        $$
            \norm{V}_{H_{1}^1(M^\epsilon)}  \leq c \epsilon^{\zeta_1},
        $$
        where $\zeta_1 > \frac{n-1}{2} - \frac{1}{n^4}$ when $n \geq 6$ and $\zeta_1 > \frac{5}{2} - \frac{1}{n^4}$ when $n = 3,4,5$;

        \item[(e)] We have 
        $$
            \norm{\nabla^2 V}_{L^2(M_{\frac{1}{2}}^\epsilon)} \leq c \epsilon^{\zeta_2},
        $$
        where $\zeta_2 > \frac{n^2-n-4}{2n} - \frac{1}{n^4}$ when $n \geq 6$, $\zeta_2 > \frac{n-1}{2} - \frac{1}{n^4}$ when $n = 4,5$, and $\zeta_2 > \frac{3}{2} - \frac{1}{3^4}$ when $n = 3$.
    \end{itemize}
    In each case, $c$ is a universal constant and $\zeta_1$, $\zeta_2$, $\zeta_3$, $\theta$, $\rho_1, \ldots, \rho_K$, $\tau$ depend only on $n$. Note that (d) is just \eqref{n45} and \eqref{n3} with $\zeta_1 = \frac{n-1}{2}$ when $n \geq 6$ and $\zeta_1 = \frac{5}{2}$ when $n = 3,4,5$. Also, when $k = 0$ \eqref{n49} shows we may take $\theta = 2 - \frac{2}{n} - \frac{1}{2n^4}$ in (a) when $n \geq 6$ and $\theta = 2 - \frac{1}{2n^4}$ in (a) when $n = 4,5$, while \eqref{n4} shows we may take $\theta = \frac{5}{2}$ when $n = 3$. We now let $V = W + \Psi_{\lambda(U)}$ so that $LW = H_0 + E(U) - L\Psi_{\lambda(U)}$ and $W = 0$ on $\partial M^\epsilon$. 
    
    Let $\varphi$ be a smooth cutoff function satisfying
    \begin{equation*}
        \varphi(x) := \begin{cases}
                        1, \text{ } x \in M_{\frac{1}{2}}^\epsilon \\
                        0, \text{ } x \in M^\epsilon \setminus M_{\frac{1}{3}}^\epsilon.
                    \end{cases}
    \end{equation*}
    Applying Proposition \ref{sest3}(a) to $\varphi W$ for $k = 1,2$ gives
    \begin{align}
        \norm{\nabla^k W}_{C^0(M_{\frac{1}{2}}^\epsilon)} &\leq c_1 \Big( \epsilon^{2-k}\big(\norm{H_0}_{C^{0,\gamma}(M_{\frac{1}{4}}^\epsilon)} + \norm{E(U)}_{C^{0,\gamma}(M_{\frac{1}{4}}^\epsilon)} + \norm{L \Psi_{\lambda(U)}}_{C^{0,\gamma}(M_{\frac{1}{4}}^\epsilon)} \nonumber\\
        &\qquad+\norm{W}_{C^{0,\gamma}(S_{\frac{1}{3},\frac{1}{2}})} +\norm{\nabla W}_{C^{0,\gamma}(S_{\frac{1}{3},\frac{1}{2}})}\big) + \epsilon^{-k} \norm{W}_{C^{0}(M_{\frac{1}{4}}^\epsilon)} \Big) \nonumber\\
        &\leq c_2\Big( \epsilon^{2 - k} + \epsilon^{2 - k}\big( \norm{E(U)}_{C^{0,\gamma}(M_{\frac{1}{4}}^\epsilon)} \label{final3} \\
        &\qquad+ \norm{L \Psi_{\lambda(U)}}_{C^{0,\gamma}(M_{\frac{1}{4}}^\epsilon)}\big) + \epsilon^{-k} \norm{W}_{C^0(M_{\frac{1}{4}}^\epsilon)}\Big), \nonumber
    \end{align}
where the second inequality comes from the bounds on the second fundamental form (see \eqref{bridge1}) and from Proposition \ref{sest4} applied with a covering by $\epsilon$-balls. Using that $U \in \mathscr{K}$, $W = V - \Psi_{\lambda(U)}$, \eqref{p}, Proposition \eqref{Eest1}(b)-(b$^\prime$), and part (a) of this proof with $k = 0$, we get 
\begin{align}
    \epsilon^{2 - k}\norm{E(U)}_{C^{0,\gamma}(M_{\frac{1}{4}}^\epsilon)} &+ \epsilon^{-k} \norm{W}_{C^0(M_{\frac{1}{4}}^\epsilon)}  \leq \nonumber \\
    &\begin{cases}
        c_3(\epsilon^{3 - k - \frac{4}{n} - \frac{2}{n^4}} + \epsilon^{2 -k - \frac{2}{n} - \frac{1}{2n^4}}, \text{ } n \geq 6\\
        c_3(\epsilon^{3 - k - \frac{3}{n^4}} + \epsilon^{2 -k - \frac{1}{2n^4}}), \text{ } n = 4,5 \\
        c_3(\epsilon^{3 - k - \frac{3}{n^4}} + \epsilon^{\frac{5}{2} - k}), \text{ }n = 3.
    \end{cases}\label{final1}
\end{align}
Applying \eqref{bridge1} and \eqref{p}, we find 
\begin{equation}
    \epsilon^{2-k}\norm{L \Psi_{\lambda(U)}}_{C^{0,\gamma}(M_{\frac{1}{4}}^\epsilon)}  \leq
    \begin{cases}
        c_4 \epsilon^{n - 1 - k - \frac{3}{n^4}} \text{ for } n \geq 6 \\
        c_4 \epsilon^{5 - k - \frac{3}{n^4}} \text{ for } n = 3,4,5.
    \end{cases} \label{final2}
\end{equation}
Thus, \eqref{p}, \eqref{final1}, and \eqref{final2} together give
\begin{equation*}
    \norm{\nabla^k V}_{C^0(M_{\frac{1}{2}}^\epsilon)} \leq 
    \begin{cases}
        c_5 \epsilon^{2 - \frac{2}{n} - \frac{1}{2n^4} - k} \text{ for } n \geq 6 \\
        c_5 \epsilon^{2 - \frac{1}{2n^4} - k} \text{ for } n =4,5 \\
        c_5 \epsilon^{\frac{5}{2} - k} \text{ for } n = 3,
    \end{cases}
\end{equation*}
where $k = 1,2$. The proof of the estimate for the H\"older norms of $V$ is identical with Proposition \ref{sest3}(b) replacing Proposition \ref{sest3}(a). Hence, condition (a) is proved with $\theta$ as in the first paragraph. 

Next, we prove condition (e). Set $R_\epsilon := \bigcup_{l = 1}^I \Gamma_l(\epsilon) \cup \bigcup_{k = 1}^{2I} B_{10\epsilon}(q_k)$ and note that on $R_\epsilon$ part (a) shows that
\begin{equation}
    \norm{\nabla^2V}_{L^2(R_\epsilon)} \leq c_6 \epsilon^{-\frac{2}{n} - \frac{1}{2n^4}} \mathcal{H}^{n}(R_\epsilon)^{\frac{1}{2}} \leq c_7 \epsilon^{\frac{n-1}{2} - \frac{2}{n} - \frac{1}{2n^4}} \text{ for } n \geq 6. \label{Repsilon1}
\end{equation}
Likewise, we have
\begin{equation}
     \norm{\nabla^2 V}_{L^2(R_\epsilon)} \leq
     \begin{cases}
        c_8\epsilon^{\frac{n-1}{2} - \frac{1}{n^4}} \text{ for } n = 4,5\\
        c_8\epsilon^{\frac{3}{2}} \text{ for } n = 3.
     \end{cases}\label{n8}
\end{equation}
Let $W$ and $\varphi$ be as in the proof of (a) and let $\tilde{\varphi}$ be a smooth cutoff function such that 
    $$
        \tilde{\varphi}(x) := \begin{cases}
                                1 \text{ for } x \in M^\epsilon \setminus R_\epsilon \\
                                0 \text{ for } \tilde{R}_\epsilon,
                            \end{cases}
$$
where $\tilde{R}_\epsilon := \bigcup_{l = 1}^I \Gamma_l(\epsilon) \cup \bigcup_{k = 1}^{2I} B_{7\epsilon}(q_k)$. Suppose further that 
$$
    \norm{\nabla^\prime \tilde{\varphi}}_{C^0(M^\epsilon)}\epsilon + \norm{(\nabla^\prime)^2 \tilde{\varphi}}_{C^0(M^\epsilon)}\epsilon^2 \leq c_{9}
$$
and note that $\mathcal{H}^n\big({\supp ((\nabla^\prime)^k \tilde{\varphi})}\big) \leq c_{10} \epsilon^n$ for $k = 1,2$. Then $\varphi \tilde{\varphi} W \in C^2(M^\epsilon)$ and is supported on $M_{3^{-1}}^\epsilon \setminus \tilde{R}_\epsilon$. Hence, it is supported on $S_{3^{-1},1}$ and vanishes on $\partial S_{3^{-1},1}$ in a neighborhood of where the bridges meet the cones. Since $S_{3^{-1},1}$ is a smooth submanifold independent of $\epsilon$, we can apply standard elliptic $L^2$ estimates for $L$ (e.g. \cite{Ma, MO}) to find 
    \begin{align*}
        \norm{\nabla^2 (\varphi \tilde{\varphi} W)}_{L^2(S_{\frac{1}{3},1})} &\leq c_{11} \big( \norm{L(\varphi \tilde{\varphi} W)}_{L^2(S_{\frac{1}{3},1})} + \norm{\varphi \tilde{\varphi} W}_{L^2(S_{\frac{1}{3},1})}\big) \\
        &\leq c_{12} \big( \norm{E(U)}_{L^2(M_{\frac{1}{3}}^\epsilon)} + \norm{L\Psi_{\lambda(U)}}_{L^2(M^\epsilon)}  \\
        &\qquad+ \norm{\Psi}_{H_1^1(M^\epsilon)} +\norm{V}_{H_1^1(M^\epsilon)} \\
        &\qquad+ \norm{\nabla W}_{C^0(R_\epsilon \setminus \tilde{R}_\epsilon)}\epsilon^{\frac{n}{2} - 1} + \norm{W}_{C^0(R_\epsilon \setminus \tilde{R}_\epsilon)}\epsilon^{\frac{n}{2} - 2}\big).
    \end{align*}
    To see how to get the second inequality, recall that $W = V - \Psi_{\lambda(U)}$, the fact that $V$ solves $LV = E(U)$ on $S_{3^{-1},1}$, and the definition of $\varphi$ and $\tilde{\varphi}$. Each of the terms on the right-hand side we can bound using previous estimates. By the argument used to estimate \eqref{L2est} and \eqref{n2}, we have
    \begin{align*}
        \norm{E(U)}_{L^2(M_{\frac{1}{3}}^\epsilon)} \leq \begin{cases}
            c_{13}\epsilon^{\frac{n+1}{2} - \frac{2}{n} - \frac{2}{n^4}} \text{ for } n \geq 6 \\
            c_{13} \epsilon^{\frac{7}{2} - \frac{2}{n^4}} \text{ for } n = 3,4,5
        \end{cases}
    \end{align*}
    and the definition of $\Psi_{\lambda(U)}$ along with \eqref{bridge1}, \eqref{p}, \eqref{n45}, and \eqref{n3} gives 
    \begin{equation*}
        \norm{V}_{H_1^1(M^\epsilon)} + \norm{L\Psi_{\lambda(U)}}_{L^2(M^\epsilon)} + \norm{\Psi}_{H_1^1(M^\epsilon)} \leq 
        \begin{cases}
            c_{14} \epsilon^{\frac{n-1}{2}} \text{ for } n \geq 6 \\
            c_{14} \epsilon^{\frac{5}{2}} \text{ for } n = 3,4,5.
        \end{cases}
    \end{equation*}
    Using the definition of $W$, part (a) of this proof, and \eqref{p} we find
    \begin{equation*}
        \norm{W}_{C^0(R_\epsilon \setminus \tilde{R}_\epsilon)} \epsilon^{\frac{n}{2} -2} + \norm{\nabla W}_{C^0(R_\epsilon \setminus \tilde{R}_\epsilon)} \epsilon^{\frac{n}{2} - 1} \leq
        \begin{cases}
            c_{15} \epsilon^{\frac{n}{2} - \frac{2}{n} -\frac{1}{2n^4}} \text{ for } n \geq 6 \\
            c_{15} \epsilon^{\frac{n}{2} - \frac{1}{2n^4}} \text{ for } n = 4,5 \\
            c_{15} \epsilon^{2} \text{ for } n = 3.
        \end{cases}
    \end{equation*}
    It follows that 
    \begin{equation}
        \norm{\nabla^2 (\varphi \tilde{\varphi} W)}_{L^2(S_{\frac{1}{3},1})} \leq 
        \begin{cases}
            c_{16} \epsilon^{\frac{n}{2} - \frac{2}{n} - \frac{1}{2n^4}} \text{ for } n \geq 6 \\
            c_{16} \epsilon^{\frac{n}{2} - \frac{1}{2n^4}} \text{ for } n = 4,5 \\
            c_{16} \epsilon^{2} \text{ for } n = 3.
        \end{cases}\label{final4}
    \end{equation}
    Since $M_{2^{-1}}^\epsilon \subset S_{3^{-1},1} \cup R_\epsilon$, combining \eqref{Repsilon1}, \eqref{n8}, and \eqref{final4} we find
    \begin{equation*}
        \norm{\nabla^2 V}_{L^2(M_{\frac{1}{2}}^\epsilon)} \leq 
        \begin{cases}
            c_{17} \epsilon^{\frac{n^2 - n - 4}{2n} - \frac{1}{2n^4}} \text{ for } n\geq 6 \\
            c_{17} \epsilon^{\frac{n-1}{2} - \frac{1}{2n^4}} \text{ for } n = 4,5 \\
            c_{17} \epsilon^{\frac{3}{2}} \text{ for } n = 3,
        \end{cases}
    \end{equation*}
    which proves (e). 
    
    Part (c) is proved for $r \in (0,\frac{1}{8}]$ with $\tau_0 = \frac{n+1}{2} - \frac{2}{n} - \frac{1}{2n^4}$ when $n \geq 7$ and $\tau_0 = 3 - \frac{3}{n^4}$ when $n = 3,4,5,6$ due to \eqref{singest} and \eqref{n414}. Thus, we need to sufficiently extend these estimates to $r \in (\frac{1}{8}, \frac{1}{4}]$ to prove (c). Observe that, by definition of $\beta_K$ and $\mathcal{V}_K$, to conclude the proof of (c) it suffices to prove (b) since we can then take $\tau = \min\{\tau_0, \rho_K\}$. Hence, we will only prove (b). In addition, we will only consider when $n = 3,4,5$ since the estimates when $n \geq 6$ are similar, and are precisely as in \cite{NS2} (see p. 636).
    
    To prove (b), one first considers when $k = 0$. This is done using an estimate for the $C^0$ norm of $V$ on some larger sets $\hat{\mathcal{V}}_l$ containing $\mathcal{V}_l$ for each $l = 1,\ldots, K$ and patching the estimates together (\cite{NS2}, fig. 5 on p. 635). Recall that $\mathcal{V}_l = S_{\sigma_{K - l - 1}, \sigma_{K - l}}$ ($l = 1,\ldots, K -1$) and $\mathcal{V}_K = S_{4^{-1}, 2^{-1}}$. Set $\overline{\sigma} := \min_{l = 1, \ldots, K-1}|\sigma_l - \sigma_{l-1}|$ and 
    $$\hat{\mathcal{V}}_l := S_{\sigma_{K - l - 1} - \frac{\overline{\sigma}}{2}, \sigma_{K - l} + \frac{\overline{\sigma}}{2}} \text{  } (l = 1,\ldots, K-1), \text{  } \hat{\mathcal{V}}_K := S_{\frac{1}{4} - \frac{\overline{\sigma}}{2}, \frac{1}{2} + \frac{\overline{\sigma}}{2}}.
    $$
    Applying Corollary \ref{sest1} on balls of radius $\frac{\overline{\sigma}}{4}$ in $\hat{\mathcal{V}}_l$ and using that $H_0 = 0$ on these sets, we obtain
    \begin{equation*}
        \norm{V}_{C^0(\hat{\mathcal{V}}_l)} \leq c_{18}\big(\epsilon^{\frac{5}{2}} + \norm{E(U)}_{L^n(\mathcal{V}_{l-1})} + \norm{E(U)}_{L^n(\mathcal{V}_l)} + \norm{E(U)}_{L^n(\mathcal{V}_{l+1})}\big),
    \end{equation*}
    where $\mathcal{V}_{l + 1} = S_{8^{-1}, 4^{-1}}$ when $l = K$. Since the $\beta_l$ increase on each $\mathcal{V}_l$ ($l$ increasing), the first $L^n$ term is the largest. Using Proposition \ref{Eest1}(d), the definition of $\mathscr{K}$, and the Minkowski inequality, we get for $n = 4,5$
    \begin{align*}
       \norm{E(U)}_{L^n(\mathcal{V}_{l-1})} &\leq c_{19}\Bigg(\Big( \int_{\mathcal{V}_{l-1}} |U|^{2n} \Big)^{\frac{1}{n}} + \Big(\int_{\mathcal{V}_{l-1}} |\nabla U|^{2n} \Big)^{\frac{1}{n}} \\
        &\qquad+ \Big(\int_{\mathcal{V}_{l-1}} |\nabla U|^{2n} |\nabla^2 U|^n  \Big)^{\frac{1}{n}} + \Big(\int_{\mathcal{V}_{l-1}}|U|^n|\nabla^2 U|^n \Big)^{\frac{1}{n}}\Bigg) \\
        &\leq c_{20} \big(\norm{U}_{C^0(\mathcal{V}_{l-1})}^{\frac{2n- 2}{n}} \norm{U}_{L^2(\mathcal{V}_{l-1})}^{\frac{2}{n}} + \norm{\nabla U}_{C^0(\mathcal{V}_{l-1})}^{\frac{2n- 2}{n}}\norm{\nabla U}_{L^2(\mathcal{V}_{l-1})}^{\frac{2}{n}} \\
        &\qquad+ \norm{\nabla U}_{C^0(\mathcal{V}_{l-1})}^2\norm{\nabla^2U}_{C^0(\mathcal{V}_{l-1})}^{\frac{n-2}{n}}\norm{\nabla^2 U}_{L^2(\mathcal{V}_{l-1})}^{\frac{2}{n}} \\
        &\qquad+ \norm{U}_{C^0(\mathcal{V}_{l-1})}\norm{\nabla^2U}_{C^0(\mathcal{V}_{l-1})}^{\frac{n-2}{n}}\norm{\nabla^2 U}_{L^2(\mathcal{V}_{l-1})}^{\frac{2}{n}}\big) \\
        &\leq c_{21}\Big( \epsilon^{\beta_{l-1} \frac{2n - 2}{n} + \frac{5}{n} - \frac{2}{n^5}} + \epsilon^{(\beta_{l - 1} - 1)\frac{2n-2}{n} + \frac{5}{n} - \frac{2}{n^5}} \\
        &\qquad+ \epsilon^{2(\beta_{l-1} - 1)  + (\beta_{l-1} - 2)\frac{n-2}{n} + 1 - \frac{1}{n} - \frac{2}{n^5}} + \epsilon^{\beta_{l-1} + (\beta_{l-1} - 2)\frac{n-2}{n} + 1 - \frac{1}{n} - \frac{2}{n^5}}\Big).
    \end{align*}
    When $n = 3$, we have
    \begin{align*}
        \norm{E(U)}_{L^3(\mathcal{V}_{l-1})} &\leq c_{22}\Big( \epsilon^{\frac{4}{3}\beta_{l-1} + \frac{5}{3} - \frac{2}{3^5}} + \epsilon^{\frac{4}{3}(\beta_{l - 1} - 1) + \frac{5}{3} - \frac{2}{3^5}} \\
        &\qquad+ \epsilon^{2(\beta_{l-1} - 1)  + \frac{1}{3} (\beta_{l-1} - 2) + 1  - \frac{2}{3^5}} + \epsilon^{\beta_{l-1} + \frac{1}{3}(\beta_{l-1} - 2) + 1 - \frac{2}{3^5}}\Big).
    \end{align*}
    
    Since $\beta_0 = 2 -\frac{1}{n^2}$ and $\beta_{l-1} \geq \beta_l - \frac{1}{n^4}$, one can check that each of the exponents in the terms above is larger than $\beta_l$ for each $n = 3,4,5$ and each $l = 1,\ldots K$. It follows that for $n = 3,4,5$
    \begin{equation}
        \norm{V}_{C^0(\hat{\mathcal{V}_l})} \leq c_{23} \epsilon^{\tilde{\rho}_l} \text{ for some } \tilde{\rho}_l > \beta_l \text{ } (l = 1,\ldots, K). \label{partb1}
    \end{equation}
    The corresponding estimate for $l = 0$ follows from part (a).
    
    Applying Proposition \ref{sest4}(a) on a covering of $\mathcal{V}_l$ by balls of radius $2\epsilon$ contained in $\hat{\mathcal{V}}_l$ and using \eqref{partb1} along with Proposition \ref{Eest1}(c), we find
    \begin{align*}
        \norm{\nabla^k V}_{C^0(\mathcal{V}_l)} &\leq c_{24} \big( \epsilon^{-k} \norm{V}_{C^0(\hat{\mathcal{V}}_l)} + \epsilon^{2-k}\norm{E(U)}_{C^{0,\gamma}(\hat{\mathcal{V}_l})}\big) \\
        &\leq c_{25}\big(\epsilon^{\tilde{\rho}_l - k} + \epsilon^{2 \beta_l - \frac{2}{n^2} - k - \gamma} + \epsilon^{3 \beta_l - 2 - \frac{3}{n^2} - k - \gamma}\big)
    \end{align*}
    where we have used $\beta_{l-1} \geq \beta_l - \frac{1}{n^2}$. Since $\beta_0 = 2 - \frac{1}{n^4}$, we have 
    $$
        2\beta_l - \frac{2}{n^4} - k - 1 > \beta_l - k \text{ and } 3\beta_l - 3 - \frac{3}{n^4} - k > \beta_l - k.
    $$
    Taking 
    $$
        \rho_l := \min\Big\{\tilde{\rho}_l, 2\beta_l - \frac{2}{n^4} - 1, 3\beta_l - 3 - \frac{3}{n^4}\Big\}
    $$
    proves the $C^0$ estimates in (b) for $k = 0,1,2$ and $l = 1,\ldots, K$, and we once again can defer to part (a) when $l = 0$. The proof of the estimates for the H\"older semi-norms in (b) is identical. Thus, the lemma is proved.
    \end{proof}
    
Lemma \ref{fixedpt2} gives us the desired solution in Theorem \ref{mainthm}. Indeed, since $\mathscr{K}$ is convex, closed, and bounded in $C_\nu^{2,\gamma}(M^\epsilon)$, the set $\mathscr{K}$ is contained in $C_\nu^{2,\frac{\gamma}{2}}(M^\epsilon)$ as a compact subset by standard properties of H\"older norms. Hence, 
$$
    T : \mathscr{K}(n,\epsilon, \gamma, \nu) \subset \mathscr{K}\Big(n,\epsilon, \frac{\gamma}{2}, \nu \Big) \rightarrow C_\nu^{2,\frac{\gamma}{2}}(M^\epsilon)
$$
is continuous in the $C_\nu^{2,\frac{\gamma}{2}}$ topology by the discussion preceding Lemma \ref{fixedpt2}. By the Schauder fixed point theorem, $T$ has a fixed point $U \in \mathscr{K}$ for $\epsilon < \min\{\epsilon_0(\gamma), \epsilon_0(\frac{\gamma}{2})\}$. In particular, $U$ solves the Dirichlet problem \eqref{DPnu} with $\Psi = \Psi_{\lambda(U)}$ and is unique by Lemma \ref{existence}. This implies the perturbed submanifold $M_U^\epsilon$ determined by $U$ is minimal. Since $U \in C_\nu^{2,\gamma}(M^\epsilon)$, we have the pointwise estimate
$$
    |U(x)| \leq c |x - p_i|^{\nu_i} \text{ for } x \in C_i, \text{ } i = 1,\ldots,N.
$$
In addition, $M_U^\epsilon$ is embedded if $\epsilon$ is sufficiently small since $\norm{U}_{C_\nu^1}$ is small relative to $\epsilon$ and $M^\epsilon$ is embedded. We further note that the inclusion of $U \in \mathscr{K}$ implies that the singularities of $M^\epsilon$ are preserved. We need to show that $M_U^\epsilon$ is smooth away from the $p_i$.

\subsection{Regularity}
We will show $U \in C^\infty(M^\epsilon)$ directly. Let $(\Omega, \psi)$ be a local parameterization of an interior patch of $M^\epsilon$ so that $(\Omega, \psi + U)$ is a local parameterization of an interior patch of $M_U^\epsilon$. Let $n_1, \ldots, n_{m+1}$ be a smooth local orthonormal frame for $NM^\epsilon \lvert_{\Omega}$ and write $U = u^\alpha n_\alpha$. Using the coordinate functions for $U$ in this frame, we can define a function $u := (u^1,\ldots, u^{m+1}): \Omega \rightarrow \mathbb{R}^{m+1}$. Define $F$ as follows: 
\begin{align*}
    F(Du, u, x) &:= (\det [D(\psi + U)^TD(\psi+U)])^{\frac{1}{2}} \\
    &= (\det \tilde{g})^{\frac{1}{2}}(\det[I + ((Du + q)\sigma)^T((Du + q)\sigma)])^{\frac{1}{2}},
\end{align*}
where $\tilde{g} := \tilde{g}(u,x) \in \mathbb{R}^{n \times n}$ is an invertible positive symmetric matrix with $\tilde{g} \rightarrow g$ in $C^0$ as $\epsilon \rightarrow 0$ (since $U \in \mathscr{K}$) and $g := (\psi_{x^i} \cdot \psi_{x^j})_{ij}$ is the metric on $M^\epsilon$, $\sigma := \sigma(u,x) \in \mathbb{R}^{n\times n}$ is an invertible positive symmetric matrix, $q := q(u,x) \in \mathbb{R}^{m+1 \times n}$, and $\norm{q}_{C^0} \rightarrow 0$ as $\epsilon \rightarrow 0$ since $U \in \mathscr{K}$. It thereby makes sense to consider the smooth map $F: \mathbb{R}^{m+1 \times n} \times \mathbb{R}^{m+1} \times \mathbb{R}^n \rightarrow \mathbb{R}$ given by 
\begin{equation}
   F(p,z,x) = (\det \tilde{g})^{\frac{1}{2}}(\det[I + ((p + q)\sigma)^T((p + q)\sigma)])^{\frac{1}{2}}, \label{R1}
\end{equation}
where we now view $\tilde{g}$, $\sigma$, and $q$ as matrix valued functions of $z$ and $x$. 

Notice that, for fixed $z,x$, the right-hand side in \eqref{R1} can be viewed as a multiple of the area integrand under an affine change of variable. Define $G: \mathbb{R}^{m + 1 \times n} \rightarrow \mathbb{R}$ by
\begin{equation}
    G(p) := \mathcal{A}((p + q)\sigma), \label{R2}
\end{equation}
where $\mathcal{A}$ denotes the usual area integrand. Using the strong rank-one convexity of $\mathcal{A}$ (\cite{T}, Lemma 6.7; the lemma is stated for $n=2$ but generalizes to arbitrary $n$), one can check that $G$ in \eqref{R2} is strongly rank-one convex for small $\epsilon$. Therefore, $F$ in \eqref{R1} is strongly rank-one convex in the variable $p$; hence, satisfies the strong Legendre-Hadamard condition (\cite{Ma}, Definition 3.36). Since $u \in C^{2,\gamma}$, we can differentiate the Euler-Lagrange system corresponding to \eqref{R1} to see that $Du$ satisfies 
\begin{align}
    \partial_i \Big(F_{p_j^\alpha p_k^\beta}(Du,u,x)u_{x^j x^k}^\beta\Big) &= \partial_i\Big(F_{z^\alpha}(Du,u,x)\delta_{ik} \label{R3}\\
    &\qquad -F_{p_j^\alpha z^\beta}(Du,u,x)u_{x^j}^\beta - F_{p_j^\alpha x^j}(Du,u,x)\Big) \nonumber
\end{align}
for each $\alpha = 1,\ldots, m+1$. Since the equation \eqref{R3} has $C^{1,\gamma}$ coefficients satisfying the strong Legendre-Hadamard condition and $C^{0,\gamma}$ right-hand side, Schauder theory implies $u \in C^{3,\gamma}(\Omega)$. Bootstrapping shows $u \in C^{\infty}(\Omega)$ which gives $U \in C^{\infty}(\Omega)$. In the case $(\Omega,\psi)$ parameterizes a boundary patch, one uses that $\Psi_{\lambda(U)} \in C_c^{\infty}(M^\epsilon)$ and boundary Schauder estimates. 

\begin{rmk}\label{rregular}
    One can alternatively show $M_U^\epsilon$ is smooth away from the $p_i$ by writing $M_U^\epsilon$ locally as the graph of a $C^{2,\gamma}$ function $v: \Omega \subset \mathbb{R}^n \rightarrow \mathbb{R}^{m+1}$ solving \eqref{MSS2} and bootstrapping. However, we presented the proof above since we will need the higher regularity of $U$ in the next section to develop a suitable normal frame for $M_U^\epsilon$. 
\end{rmk}

This completes the proof of Theorem \ref{mainthm}, excluding the strict stability statement and the graphical case when $n \geq 4$.  

\section{Stability}
To conclude the generalization, we must show that the submanifold $M_U^\epsilon$ obtained from Theorem \ref{mainthm} is strictly stable for $\epsilon$ small enough and each $n \geq 3$. The argument is motivated by the stability arguments in \cite{NS2} in the codimension one case (see also \cite{CHS, NS3}). The key difference is that, in the codimension one case, normal vector fields on $M_U^\epsilon$ can be identified with smooth functions on $M^\epsilon$. In high codimension, this is not the case since the contribution from normal derivatives of elements of a smooth local orthonormal frame for $NM_U^\epsilon$ are non-negligible. Thus, we need to derive an appropriate local orthonormal frame for $NM_U^\epsilon$ around each of its points so that we may compare the values of the stability operator $L_U$ on $M_U^\epsilon$ to those of $L$ on $M^\epsilon$ when $\epsilon$ is small. Precisely, we show that given a smooth local orthonormal frame for $NM^\epsilon$ there is an associated local orthonormal frame on $NM_U^\epsilon$ that is a small perturbation of the original. This allows us to directly compare the spectrums of the stability operator on $M^\epsilon$ and $M_U^\epsilon$. Going forward, the constants $c,\eta$ will always denote positive constants independent of $\epsilon$ and $\delta$.

Let $\psi: \Omega \subset \mathbb{R}^n \rightarrow \mathbb{R}^{n + m + 1}$ be a smooth local parameterization for $M^\epsilon$ about a point $p := \psi(x_0)$ satisfying the hypotheses in Sec. 8.1, and let $n_1,\ldots, n_{m+1}$ be a smooth local orthonormal frame for $NM^\epsilon$ in a neighborhood of $p$ contained in $\psi(\Omega)$. Applying Theorem \ref{mainthm} for $\epsilon$ small enough, we obtain $U \in \mathscr{K}$ solving \eqref{DPnu} and an associated minimal submanifold $M_U^\epsilon$. The map $\phi := \psi + U$ is a local parameterization of $M_U^\epsilon$ near $q:= p + U(p)$. A smooth local frame for $TM^\epsilon$ near $p$ is given by $\{\psi_{x^i}\}_{i=1}^n$ and a smooth local frame for $TM_U^\epsilon$ near $q$ is $\{\phi_{x^i}\}_{i = 1}^n$. Let $g_{ij} := \psi_{x^i} \cdot \psi_{x^j}$ be the components of the metric on $M^\epsilon$ and set $g:= (g_{ij})$, $g^{-1} := (g^{ij})$. We first want to find a local frame for $NM_U^\epsilon$ of the form 
$$
    \nu_\alpha := n_\alpha + \epsilon_\alpha \text{ for } \alpha = 1,\ldots, m + 1,
$$
where $\epsilon_\alpha$ is a smooth section of $TM^\epsilon$ near $p$. To do so, we solve the system
$$
    \phi_{x^i} \cdot \nu_\alpha = 0 \text{ for each } i = 1,\ldots, n \text{ and } \alpha = 1,\ldots, m+1.
$$

Write $U = u^\beta n_\beta$. Expanding the above expression gives 
\begin{equation}
    \psi_{x^i} \cdot \epsilon_\alpha + u^\beta(n_\beta)_{x^i} \cdot \epsilon_\alpha = - u_{;i}^\alpha \text{ for each } i,\alpha \text{ fixed}, \label{stable1}
\end{equation}
where $u_{;i}^\alpha$ is defined by \eqref{coord1}. Since $\epsilon_\alpha \in TM^\epsilon$, we may write $\epsilon_\alpha := a^{k\alpha}\psi_{x^k}$. Plugging this into \eqref{stable1}, we get 
\begin{equation}
    g_{ik}a^{k\alpha} + u^\beta((n_\beta)_{x^i} \cdot \psi_{x^k})a^{k\alpha} = -u_{;i}^\alpha. \label{stable2}
\end{equation}
Define the column vectors 
\begin{equation*}
    a_\alpha := \begin{pmatrix}
                a^{1\alpha} \\
                \vdots \\
                a^{n\alpha}
            \end{pmatrix}
    \text{ and }
    b_\alpha := \begin{pmatrix}
                -u_{;1}^\alpha \\
                \vdots \\
                -u_{;n}^\alpha
            \end{pmatrix}.
\end{equation*}
For $\alpha$ fixed, we can then write \eqref{stable2} as 
\begin{equation}
    Ba_\alpha = b_\alpha, \label{stable3}
\end{equation}
 where $B:=B(\epsilon) = g + S$ and $S := S(\epsilon) = (s_{kl}(\epsilon))$ is given by  
\begin{equation}
    s_{kl} = u^\alpha((n_\alpha)_{x^k} \cdot \psi_{x^l}). \label{s}
\end{equation}
Using \eqref{psiest}-\eqref{nest} and Remark \ref{nbdds}, we get the pointwise bounds for each $n \geq 3$:
\begin{equation}
    |s_{kl}| \leq ch^{-1}|U| \text{ on } M^\epsilon, \label{s2}
\end{equation}
where $h$ is the function \eqref{h}. The inclusion of $h$ in the second estimate comes from the $(n_\alpha)_{x^i}$ term in \eqref{s}.

Notice that $S \rightarrow 0$ as $\epsilon \rightarrow 0$ in $C^0$ (i.e. supremum of the Frobenius norm) since $U \in \mathscr{K}$. By continuity and the invertibility of $g$, it follows that $B$ is invertible for small $\epsilon$ with 
\begin{equation}
    B^{-1} = (I + g^{-1}S)^{-1}g^{-1} = \sum_{k = 0}^\infty(-1)^k(g^{-1} S)^kg^{-1}. \label{Binverse}
\end{equation}
Inverting the expression \eqref{stable3} for each $\alpha$, we find $a_\alpha = B^{-1} b_\alpha$. We will prove the following pointwise bounds on $M^\epsilon$ for $n \geq 3$ and small $\epsilon$:
\begin{equation}
    \begin{cases}
        |\epsilon_\alpha| \leq c|\nabla U| \\
        |(\epsilon_\alpha)_{x^i} \cdot n_\beta| \leq ch^{-1}|\nabla U| \\
        |(\epsilon_\alpha)_{x^i}| \leq c(|\nabla^2 U| + h^{-1}|\nabla U| + h^{-2} |U|)
    \end{cases} \label{stable4}
\end{equation}
for each $\alpha = 1,\ldots, m+1$ and each $i = 1,\ldots, n$. As before, the terms involving $h$ in \eqref{stable4} come from differentiating the metric or the $n_\alpha$ on the cones $C_i$.

To prove \eqref{stable4}, it suffices to bound the $a_\alpha$ and their derivatives. We have
$$
    |a_\alpha| \leq |B^{-1}||b_\alpha| \leq c|\nabla U||B^{-1}|.
$$
Using \eqref{s2} and that $U \in \mathscr{K}$, we see that the higher order terms in \eqref{Binverse} are dominated by the lower order terms for $\epsilon$ small. Hence, 
$$
    |B^{-1}| \leq c(1 + h^{-1}|U|).
$$
Combining the previous two estimates and using that $h^{-1}|U| < 1$ for $U \in \mathscr{K}$ and $\epsilon$ small proves the first estimate in \eqref{stable4}. Next, note that 
$$
    (a_\alpha)_{x^i} = \pdv{B^{-1}}{x^i}b_\alpha + B^{-1}(b_\alpha)_{x^i}. 
$$
Using the estimate for $B^{-1}$, \eqref{coord1}, and $h^{-1}|U| < 1$ we get
$$
    |B^{-1}(b_\alpha)_{x^i}| \leq |B^{-1}||(b_\alpha)_{x^i}| \leq c(|\nabla^2 U| + h^{-1}|\nabla U| + h^{-2}|U|).
$$
Recalling that
\begin{equation}
    \pdv{C^{-1}}{x^i} = - C^{-1} \pdv{C}{x^i} C^{-1} \label{dinv}
\end{equation}
for any invertible matrix $C$ and using that $B = g + S$, we obtain 
$$
    \Big|\pdv{B^{-1}}{x^i}b_\alpha\Big| \leq c|\nabla U|\Big|\pdv{B^{-1}}{x^i}\Big| \leq c|\nabla U|\Big|\pdv{B}{x^i}\Big| \leq c|\nabla U|(1 + |\nabla U| + h^{-1}|U|).
$$
Since $|\nabla U| < 1$, this combined with the previous estimate proves the third estimate in \eqref{stable4}, and the second estimate follows from the first and the expression for $(a_\alpha)_{x^i}$. In addition, $U \in \mathscr{K}$ and \eqref{stable4} imply that for small $\epsilon$ the normal sections $\nu_\alpha$ are linearly independent near $q$. Hence, they form a smooth local normal frame for $M_U^\epsilon$ near $q$.

We now perform the Gram-Schmidt procedure on the $\nu_\alpha$ to obtain a smooth local orthonormal frame $\overline{n}_1, \ldots, \overline{n}_{m+1}$ for $NM_U^\epsilon$. In fact, we show that we can write
\begin{equation}
    \overline{n}_\alpha = n_\alpha + \xi_\alpha \text{ for each } \alpha = 1,2,\ldots, m+1 \label{sigma}
\end{equation}
for some smooth vector field $\xi_\alpha \in T\mathbb{R}^{n+m+1}\lvert_{M^\epsilon}$ satisfying the same bounds as the $\epsilon_\alpha$ in \eqref{stable4}. Precisely, on $M^\epsilon$
\begin{equation}
     \begin{cases}
        |\xi_\alpha| \leq c|\nabla U| \\
        |(\xi_\alpha)_{x^i} \cdot n_\beta| \leq ch^{-1}|\nabla U| \\
        |(\xi_\alpha)_{x^i}| \leq c(|\nabla^2 U| + h^{-1}|\nabla U| + h^{-2} |U|)
    \end{cases} \label{xi1}
\end{equation}
for each $\alpha = 1,\ldots, m+1$. 

Indeed, when $\alpha = 1$
$$
     \overline{n}_1 = \frac{n_1 + \epsilon_1}{|n_1 + \epsilon_1|}.
$$
Noting that $|n_1 + \epsilon_1|^{-1} = (1 + |\epsilon_1|^2)^{-\frac{1}{2}}$ and using the power series expansion 
\begin{equation}
    (1 + x^2)^{-\frac{1}{2}} = 1 - \frac{x^2}{2} + \frac{3x^4}{8} - \frac{5x^6}{16} +  \frac{35 x^8}{128} + O(x^{10}) \label{power}
\end{equation}
about zero shows that 
$$
    \overline{n}_1 = n_1 + \epsilon_1 + (n_1 + \epsilon_1)f_1.
$$
where $f_1$ is given by 
$$
    f_1 := \frac{1}{(1 + |\epsilon_1|^2)^{\frac{1}{2}}} - 1 = -\frac{|\epsilon_1|^2}{2} + O(|\epsilon_1|^4).
$$
Set
$$
    \xi_1 := \epsilon_1 + (n_1 + \epsilon_1)f_1.
$$
By direct computation, the estimates \eqref{xi1} then follow from \eqref{stable4} and the fact that $U \in \mathscr{K}$. Iterating this procedure for $\alpha \geq 2$ starting with $\alpha = 2$, the Gram-Schmidt process gives
\begin{equation*}
     \overline{n}_\alpha = \frac{n_\alpha + \epsilon_\alpha - \sum_{\beta = 1}^{\alpha-1}(\epsilon_\alpha \cdot \xi_\beta)(n_\beta + \xi_\beta)}{(1 + |\epsilon_\alpha - \sum_{\beta = 1}^{\alpha-1}(\epsilon_\alpha \cdot \xi_\beta)(n_\beta + \xi_\beta)|^2)^{\frac{1}{2}}}
\end{equation*}
where
$$
    \xi_\beta := \epsilon_\beta - \sum_{\sigma = 1}^{\beta-1}(\epsilon_\beta \cdot \xi_\sigma)(n_\sigma + \xi_\sigma) +(n_\beta + \epsilon_\beta - \sum_{\sigma = 1}^{\beta-1}(\epsilon_\beta \cdot \xi_\sigma)(n_\sigma + \xi_\sigma))f_\beta
$$
and
\begin{align*}
    f_\beta &:= \frac{1}{(1 + |\epsilon_\beta - \sum_{\sigma = 1}^{\beta-1}(\epsilon_\beta \cdot \xi_\sigma)(n_\sigma + \xi_\sigma)|^2)^{\frac{1}{2}}} -1 \\
    &= -\frac{|\epsilon_\beta - \sum_{\sigma = 1}^{\beta-1}(\epsilon_\beta \cdot \xi_\sigma)(n_\sigma + \xi_\sigma)|^2}{2} + O\Big(|\epsilon_\beta - \sum_{\sigma = 1}^{\beta-1}(\epsilon_\beta \cdot \xi_\sigma)(n_\sigma + \xi_\sigma)|^4\Big).
\end{align*}
by \eqref{power}. Using the power series representation \eqref{power} once more then shows $\overline{n}_\alpha = n_\alpha + \xi_\alpha$ for $\xi_\alpha$ defined as above with $\alpha$ replacing $\beta$. Direct computation shows $\xi_\alpha$ satisfies the bounds \eqref{xi1} for each $\alpha = 1,\ldots, m+1$ since the norm of the derivative of $\epsilon_\alpha$ dominates the norm of the derivative of the higher order terms involving the products of $\epsilon_\alpha$ and the $\xi_\beta$, as well as the derivative of $f_\alpha$. We leave the details to the reader (e.g. start with $\alpha = 2$).

Having developed suitable local orthonormal frames for $NM_U^\epsilon$, we can now compare the metrics and second fundamental forms on $M^\epsilon$ and $M_U^\epsilon$. Denote the metric on $M_U^\epsilon$ by $g(U) := (g_{ij}(U))$ and denote the second fundamental forms for $M^\epsilon$ and $M_U^\epsilon$ by $A$ and $A(U)$, respectively. Write 
$$
    A_{ij}^\alpha := \psi_{x^i x^j} \cdot n_\alpha \text{ and } A_{ij}^\alpha(U) := (\psi_{x^i x^j} + U_{x^i x^j}) \cdot \overline{n}_\alpha.
$$
We have 
\begin{align*}
    g_{ij}(U) &= g_{ij} + U_{x^i} \cdot U_{x^j} - 2 \psi_{x^i x^j} \cdot U \\
    &=: g_{ij} + R_{ij}(U) \\
    A_{ij}^\alpha(U) &= A_{ij}^\alpha + \psi_{x^i x^j} \cdot \xi_\alpha + U_{x^i x^j} \cdot n_\alpha + U_{x^i x^j} \cdot \xi_\alpha \\
    &=: A_{ij}^\alpha + R_{ij}^\alpha(U).
\end{align*}
Using \eqref{coord1}, \eqref{coord2}, and \eqref{xi1} we obtain pointwise bounds for the remainders in the expressions above on $M^\epsilon$:
\begin{equation}
    \begin{cases}
        |R_{ij}(U)| \leq ch^{-1}(|\nabla U|^2 + |U|) \\
        |R_{ij}^\alpha(U)| \leq c(|\nabla^2 U| + h^{-1}|\nabla U| + h^{-2}|U|),
    \end{cases} \label{remainder}
\end{equation}
where the lower order terms come from $\xi_\alpha$ and the tangential component of $U_{x^i x^j}$. By \eqref{remainder} and the fact that $U \in \mathscr{K}$, we may compute $g^{-1}$ precisely as in \eqref{Binverse}. This along with the previous estimates shows that on $M^\epsilon$ we have
\begin{equation}
    \begin{cases}
         |g_{ij}(U) - g_{ij}| \leq ch^{-1}(|\nabla U|^2 + |U|) \\
         |g^{ij}(U) - g^{ij}| \leq ch^{-1}(|\nabla U|^2 + |U|).
    \end{cases} \label{stable5}
\end{equation}

Let $\mathcal{H}^n$ be the $n$-dimensional Hausdorff measure on $M^\epsilon$ and write $\mathcal{H}_U^n$ for the $n$-dimensional Hausdorff measure on $M_U^\epsilon$. Since 
$$
    \det g(U) = \det g \det(I + g^{-1}R(U))
$$
for $R(U) := (R_{ij}(U))$, direct computation shows 
\begin{equation}
    d\mathcal{H}_U^n = \sqrt{\det g(U)} \, dx = (1 + \mu) \, d\mathcal{H}^n \label{volume}
\end{equation}
where $\mu$ is some function on $M^\epsilon$ that is a polynomial in the components of $g$, $g^{-1}$, and $R(U)$. Furthermore,
\begin{align*}
    \pdv{R_{ij}(U)}{x^k} &=  U_{x^i x^k} \cdot U_{x^j}  \\
    &\qquad+ U_{x^i} \cdot U_{x^j x^k} -2 (\psi_{x^i x^j x^k} \cdot U + \psi_{x^i x^j} \cdot U_{x^k}). \nonumber
\end{align*}
Using \eqref{coord1}, \eqref{coord2}, and that $U \in \mathscr{K}$ shows 
\begin{equation}
    \Big|\pdv{R_{ij}(U)}{x^k}\Big| \leq c\epsilon^\eta \text{ for some } \eta > 0. \label{rderiv}
\end{equation}
Combining \eqref{rderiv} with \eqref{remainder} gives
\begin{equation}
   |\mu(x)| + |\nabla^\prime \mu(x)| \leq c \epsilon^{\eta_0} \text{ for } x \in M^\epsilon \label{mu} 
\end{equation}
for some constants $c,\eta_0$ independent of $\epsilon$. With the above estimates in hand, we can now prove the strict stability of $M_U^\epsilon$ for $\epsilon$ small.

\begin{proof}[Proof of Strict Stability]
    Set $M_U^\epsilon(\delta) := \{x + U(x) : x \in M_\delta^\epsilon, \text{ } \delta > 0\}$ and let $\lambda_\delta$ be the first eigenvalue of $h^2L_U$ as an operator on $L_{1}^2(M_U^\epsilon(\delta))$. Let $V \in C^\infty(M_U^\epsilon(\delta))$ be a unit eigensection for $h^2L_U$ on $M_U^\epsilon(\delta)$ corresponding to $\lambda_\delta$ so that $V$ satisfies 
    \begin{align*}
        L_U V &= - \lambda_\delta h^{-2} V \text{ on } M_U^\epsilon(\delta) \\
        V &= 0 \text{ on } \partial M_U^\epsilon(\delta), 
    \end{align*}
    and 
    $$
        \int_{M_U^\epsilon(\delta)} h^{-2}|V|^2 \, d\mathcal{H}_U^n = 1.
    $$
    To prove $M_U^\epsilon$ is strictly stable, we need to show $\inf_{\delta>  0} \lambda_\delta > 0$ (e.g. see \cite{CHS, NS2, NS3}). Since $\lambda_\delta$ is decreasing in $\delta$, we may assume $\delta \in (0,\frac{1}{10})$.

    Before we continue, we need to do some preliminary calculations. We can write $V := v^\alpha \overline{n}_\alpha$ in a local orthonormal frame $\overline{n}_1, \ldots, \overline{n}_{m+1}$ for $NM_U^\epsilon(\delta)$ constructed as above. Note that functions on $M_U^\epsilon$ can naturally be viewed as functions on $M^\epsilon$. Let $\tilde{V} \in C_0^{\infty}(M_\delta^\epsilon)$ be the orthogonal projection of $V$ onto $NM^\epsilon$ at each point. Then $\tilde{V} = \tilde{v}^\alpha n_\alpha$ in the constructed frame, where $\tilde{v}^\alpha := v^\alpha + v^\beta(\xi_\beta \cdot n_\alpha)$.
    In addition,
    \begin{equation}
    \begin{cases}
        v_{;i}^\alpha = v_{x^i}^\alpha + B_{\beta i}^\alpha v^\beta + C_{\beta i}^\alpha v^\beta \\
         \tilde{v}_{;i}^\alpha = v_{x^i}^{\alpha} + B_{\beta i}^\alpha v^\beta + v_{x^i}^\beta (\xi_\beta \cdot n_\alpha)  + D_{\beta i}^\alpha v^\beta  
    \end{cases} \label{stable10}
    \end{equation}
    where
    \begin{equation}
        \begin{cases}
            C_{\beta i}^\alpha = (n_\beta)_{x^i} \cdot \xi_\alpha + (\xi_\beta)_{x^i} \cdot n_\alpha + (\xi_\beta)_{x^i} \cdot \xi_\alpha\\
            D_{\beta i}^\alpha =(\xi_\beta)_{x^i} \cdot n_\alpha + \xi_\beta \cdot (n_\alpha)_{x^i}  + \sum_\sigma B_{\sigma i}^\alpha (\xi_\sigma \cdot n_\beta). \label{CD}
        \end{cases}
    \end{equation}
    Using \eqref{xi1} and that $U \in \mathscr{K}$ shows 
    \begin{equation}
        \norm{C_{\beta i}^\alpha}_{C^0} \leq c_1\epsilon^{\eta_1} \text{ and } \norm{D_{\beta i}^\alpha}_{C^0} \leq c_2\epsilon^{\eta_2} \label{CD2}
    \end{equation}
    for each $i,\alpha, \beta$. We thus have the pointwise estimates: 
    \begin{equation}
        \begin{cases}
            |\tilde{V}|^2 \geq |V|^2(1 - c_3\epsilon^{\eta_3}) \\
            |\nabla \tilde{V}|^2 \leq c_4(|\nabla_U V|^2 + |V|^2),
        \end{cases} \label{stable6}
    \end{equation}
    where the first inequality follows from \eqref{xi1} and the second from \eqref{coord1}, \eqref{stable10}, \eqref{CD}, and \eqref{CD2}. Here, $\nabla_U$ is the connection on $NM_U^\epsilon$. We further note that if we represent Simons' operator on $M^\epsilon$ and $M_U^\epsilon$ by $\tilde{A}$ and $\tilde{A}_U$, respectively, then 
    \begin{equation}
        \begin{cases}
            \tilde{A}(\tilde{V}) = (g^{kj}g^{il} A_{kl}^\beta A_{ij}^\alpha (v^\beta + v^\sigma (\xi_\sigma \cdot n_\beta)))n_\alpha \\
            \tilde{A}_U(V) = (g^{kj}(U)g^{il}(U)A_{kl}^\beta(U)A_{ij}^\alpha(U)v^\beta)\overline{n}_\alpha
        \end{cases} \label{simstab}
    \end{equation}
    in the chosen coordinates by \eqref{csimons}, where we have summed over repeated indices.
    
     Let $\Delta_U$ be the normal Laplacian on $NM_U^\epsilon$. Then integration by parts combined with \eqref{volume} and the fact that $V$ is an $L_{1}^2(M_U^\epsilon(\delta))$ unit vector solving the eigenvalue equation shows
    \begin{align*}
        - \int_{M_\delta^\epsilon} \langle \tilde{V}, L\tilde{V} \rangle \, d\mathcal{H}^n &\leq \lambda_\delta + \Big|\int_{M^\epsilon} \big(|\nabla \tilde{V}|^2 - |\nabla_U V|^2 \big) \, d\mathcal{H}^n \Big| \\
        &\qquad+ \Big|\int_{M^\epsilon} \big(\langle V, \tilde{A}_U(V) \rangle- \langle \tilde{V}, \tilde{A}(\tilde{V}) \rangle \big) \, d\mathcal{H}^n \Big| \\
        &\qquad+ \Big| \int_{M^\epsilon}|\nabla \tilde{V}|^2 \mu \, d\mathcal{H}^n \Big| + \Big| \int_{M^\epsilon} |\nabla_U V|^2 \mu \, d\mathcal{H}^n\Big|,
    \end{align*}
    where on the right-hand side we have extended $V$, $\tilde{V}$, and their covariant derivatives to be zero outside $M_U^\epsilon(\delta)$ and $M_\delta^\epsilon$, respectively. Furthermore, standard Hilbert space theory and Lemma \ref{deltastable} show
    $$
        - \int_{M_\delta^\epsilon} \langle \tilde{V}, L\tilde{V} \rangle \, d\mathcal{H}^n \geq d_0\int_{M_\delta^\epsilon} |\tilde{V}|^2 h^{-2} \, d\mathcal{H}^n.
    $$
    Applying \eqref{volume}, \eqref{mu}, and the first inequality in \eqref{stable6}, for small enough $\epsilon$ we then get
    \begin{align}
        0 < \frac{d_0}{2} &\leq \lambda_\delta + \Big|\int_{M^\epsilon} \big(|\nabla \tilde{V}|^2 - |\nabla_U V|^2\big) \, d\mathcal{H}^n\Big| \label{stable8}\\
        &\qquad+ \Big| \int_{M^\epsilon} \big(\langle V, \tilde{A}_U(V) \rangle- \langle \tilde{V}, \tilde{A}(\tilde{V}) \rangle \big) \, d\mathcal{H}^n \Big| \nonumber \\
        &\qquad+ \Big|\int_{M^\epsilon}|\nabla \tilde{V}|^2 \mu \, d\mathcal{H}^n \Big| + \Big| \int_{M^\epsilon} |\nabla_U V|^2 \mu \, d\mathcal{H}^n \Big|. \nonumber
    \end{align}
    If we can show each of the integral terms on the right-hand side is bounded by $c\epsilon^\eta$, we are done. We begin with the second integral term (i.e. the one involving Simons' operator). 

    Using \eqref{simstab}, we deduce 
    \begin{align}
        |\langle \tilde{A}_U(V), V \rangle &- \langle \tilde{A}(\tilde{V}), \tilde{V} \rangle| \leq \nonumber \\
        &|(g^{kj}(U)g^{il}(U)A_{kl}^\beta (U)A_{ij}^\alpha(U) - g^{kj}g^{il}A_{kl}^\beta A_{ij}^\alpha)v^\alpha v^\beta| \nonumber \\
        + &|g^{kj}g^{il} A_{kl}^\beta A_{ij}^\alpha v^\alpha v^\sigma (\xi_\sigma \cdot n_\beta)| +|g^{kj}g^{il} A_{kl}^\beta A_{ij}^\alpha v^\gamma v^\beta(\xi_\gamma \cdot n_\alpha)| \label{Aest1} \\
        + &|g^{kj}g^{il} A_{kl}^\beta A_{ij}^\alpha v^\sigma v^\gamma (\xi_\sigma \cdot n_\beta)(\xi_\gamma \cdot n_\alpha)|. \nonumber
    \end{align}
    Observe that 
    \begin{align*}
        |g^{kj}(U)g^{il}(U)&A_{kl}^\beta (U)A_{ij}^\alpha(U) - g^{kj}g^{il}A_{kl}^\beta A_{ij}^\alpha| \leq \nonumber \\
        &|(g^{kj}(U)g^{il}(U) - g^{kj}g^{il})A_{kl}^\beta A_{ij}^\alpha| 
        +|g^{kj}(U)g^{il}(U)A_{kl}^\beta R_{ij}^\alpha(U)|  \\
        + &|g^{kj}(U)g^{il}(U)A_{ij}^\alpha R_{kl}^\beta(U)| 
        + |g^{kj}(U)g^{il}(U)R_{kl}^\beta(U)R_{ij}^\alpha(U)|. \nonumber
    \end{align*}
    Since $|A_{ij}^\alpha| \leq ch^{-1}$ for some constant $c$ independent of $\epsilon$ when $\epsilon$ is small, \eqref{stable5} shows that
    \begin{equation*}
        |(g^{kj}(U)g^{il}(U) - g^{kj}g^{il})A_{kl}^\beta A_{ij}^\alpha| \leq c_5h^{-3}(|\nabla U|^2 +  |U|) \text{ on } M^\epsilon. \label{z1}
    \end{equation*}
    Since $g^{ij}(U)$ is also uniformly bounded independent of $\epsilon$ on $M^\epsilon$ when $\epsilon$ is small, we have 
    \begin{align*}
        |g^{kj}(U)g^{il}(U)A_{kl}^\beta R_{ij}^\alpha(U)| &\leq c_6h^{-1}(|\nabla^2 U| + h^{-1}|\nabla U| + h^{-2}|U|) \\
        |g^{kj}(U)g^{il}(U)R_{kl}^\beta(U)R_{ij}^\alpha(U)| &\leq c_7(|\nabla^2 U|^2 + h^{-2}|\nabla U|^2 + h^{-4}|U|^2).
    \end{align*}
    Hence, on $M^\epsilon$
    \begin{multline*}
        |(g^{kj}(U)g^{il}(U)A_{kl}^\beta (U)A_{ij}^\alpha(U) - g^{kj}g^{il}A_{kl}^\beta A_{ij}^\alpha)v^\alpha v^\beta| \leq \\ c_8(|\nabla^2 U|^2 + |\nabla^2 U| + h^{-1}|\nabla U| + h^{-2}|U|)|V|^2.
    \end{multline*}
    The remaining terms in \eqref{Aest1} are smaller since
    \begin{align*}
        |g^{kj}g^{il} A_{kl}^\beta A_{ij}^\alpha v^\gamma v^\beta(\xi_\gamma \cdot n_\alpha)| &\leq c_9h^{-2}|\nabla U||V|^2  \\
        |g^{kj}g^{il} A_{kl}^\beta A_{ij}^\alpha v^\sigma v^\gamma (\xi_\sigma \cdot n_\beta)(\xi_\gamma \cdot n_\alpha)| &\leq c_{10}h^{-2}|\nabla U|^2|V|^2
    \end{align*}
    due to \eqref{xi1}. Thus, on $M^\epsilon$
    \begin{equation}
        |\langle \tilde{A}_U(V), V \rangle - \langle \tilde{A}(\tilde{V}), \tilde{V} \rangle| \leq c_{11}(|\nabla^2 U|^2 + |\nabla^2 U| + h^{-1}|\nabla U| +h^{-2}|U|)|V|^2. \label{stable7}
    \end{equation}

    Next, we show that $\norm{V}_{C^0(M_U^\epsilon(\frac{1}{2}))}$ is bounded independent of $\epsilon$. Note that $M_U^\epsilon(\frac{1}{5})$ has a uniform Sobolev inequality independent of $\epsilon$ since it has zero mean curvature and has volume bounded independent of $\epsilon$ (see \cite{MS}). In addition, it is not hard to show that an analogous version of Lemma \ref{weaksoln} holds for equations of the form $L_U W + \lambda_\delta h^{-2} W= F$. Hence, we may use Theorem 8.16 in \cite{GT} as well as Corollary \ref{sest1}. Let $\Delta_U^\prime$ be the Laplace-Beltrami operator on $M_U^\epsilon$. Applying Theorem 8.16 in \cite{GT} to $\Delta_U^\prime |V|$ with $F = 0$, we find 
    \begin{align}
        \norm{V}_{C^0(M_U^\epsilon(\frac{1}{5}))} &\leq c(\tau)\Bigg(\Big( \int_{M_U^\epsilon(\frac{1}{5})} |A(U)|^{n + 2\tau}|V|^{\frac{n}{2} + \tau} \, d\mathcal{H}_U^n \Big)^{\frac{2}{n + 2\tau}} \nonumber \\
        &\qquad+ |\lambda_\delta| \Big(\int_{M_U^\epsilon(\frac{1}{5})} h^{-n - 2\tau} |V|^{\frac{n}{2} + \tau} \, d\mathcal{H}_U^n \Big)^{\frac{2}{n + 2\tau}} + \norm{V}_{C^0(\partial M_U^\epsilon(\frac{1}{5}))} \Bigg). \label{s0} 
    \end{align}
    When $n = 3$, we take $\tau = 2^{-1}$. We estimate the terms on the right-hand side above separately.  

    For any $W \in C_0^\infty(M_U^\epsilon(\delta))$ which is a unit vector in $L_{1}^2(M_U^\epsilon(\delta))$, we have the following inequality
    \begin{equation}
        |\lambda_\delta| \leq \Big|\int_{M_U^\epsilon(\delta)} \langle W, L_UW\rangle \, d\mathcal{H}_U^n\Big|.\label{eigen}
    \end{equation}
    Choosing $W$ vanishing in the interior of $\cup_{i = 1}^N (C_{i,\frac{1}{2}}(U))$, where 
    $$
        C_{i,\frac{1}{2}}(U) := \{ x + U(x) : x \in C_{i,\frac{1}{2}} \}
    $$
    and $C_{i,\frac{1}{2}}$ are the truncated cones \eqref{tcone}, we may integrate by parts in \eqref{eigen} and use \eqref{stable5}, \eqref{stable10}, \eqref{CD}, \eqref{simstab}, and $U \in \mathscr{K}$ to see that $|\lambda_\delta| \leq c^\prime$ for some $c^\prime$ independent of $\delta$ and $\epsilon$ whenever $\epsilon$ is small. Thus,
    \begin{align*}
        |\lambda_\delta| \Big( \int_{M_U^\epsilon(\frac{1}{5})} h^{-n - 2\tau} |V|^{\frac{n}{2} + \tau} \, d\mathcal{H}_U^n \Big)^{\frac{2}{n + 2\tau}} &\leq c_{12}\norm{V}_{C^0(M_U^\epsilon(\frac{1}{5}))}^{\frac{n - 4 + 2\tau}{n + 2\tau}}\norm{V}_{L^2(M_U^\epsilon(\frac{1}{5}))}^{\frac{4}{n+2\tau}}.
    \end{align*}
    Applying Young's inequality $ab \leq \sigma a^p + c_\sigma b^q$ with $\sigma = (4 c(\tau) c_{12})^{-1}$ and
    $$
        p = \frac{n + 2\tau}{n - 4 + 2\tau}, \text{ } q = \frac{p}{p-1}
    $$
    gives
    \begin{align}
        |\lambda_\delta| \Big( \int_{M_U^\epsilon(\frac{1}{5})} h^{-n - 2\tau} &|V|^{\frac{n}{2} + \tau} \, d\mathcal{H}_U^n \Big)^{\frac{2}{n + 2\tau}} \leq \label{s1} \\ 
        &\frac{1}{4c(\tau)}\norm{V}_{C^0(M_U^\epsilon(\frac{1}{5}))} +  c_\sigma \norm{V}_{L^2(M_U^\epsilon(\frac{1}{5}))}. \nonumber
    \end{align}
    Using Young's inequality as above with $\sigma = (4c(\tau))^{-1}$, the Schwarz inequality, and the fact that $\norm{V}_{L_{1}^2(M_U^\epsilon(\delta))} = 1$, we find 
    \begin{align*}
        \Big( &\int_{M_U^\epsilon(\frac{1}{5})} |A(U)|^{n + 2\tau}|V|^{\frac{n}{2} + \tau} \, d\mathcal{H}_U^n \Big)^{\frac{2}{n + 2\tau}} \leq \nonumber\\
        &\norm{V}_{C^0(M_U^\epsilon(\frac{1}{5}))}^{\frac{n - 2 + 2\tau}{n + 2\tau}} \Big( \int_{M_U^\epsilon(\frac{1}{5})} h^2|A(U)|^{2n + 4\tau} \, d\mathcal{H}_U^n \Big)^{\frac{1}{n + 2\tau}}  \Big( \int_{M_U^\epsilon(\frac{1}{5})} h^{-2}|V|^{2} \, d\mathcal{H}_U^n \Big)^{\frac{1}{n + 2\tau}}  \\
        &\leq \frac{1}{4c(\tau)} \norm{V}_{C^0(M_U^\epsilon(\frac{1}{5}))} + c_\sigma \Big( \int_{M_U^\epsilon(\frac{1}{5})} |A(U)|^{2n + 4\tau}\, d\mathcal{H}_U^n \Big)^{\frac{1}{2}}. \nonumber
    \end{align*}
    By \eqref{remainder}, \eqref{stable5}, the boundedness of $|A|$ on $M_{\frac{1}{5}}^\epsilon$, the Minkowski inequality, \eqref{volume}, \eqref{mu}, and the fact that $U \in \mathscr{K}$ we have for small $\epsilon$
    \begin{align}
    \Big( \int_{M_U^\epsilon(\frac{1}{5})} &|A(U)|^{2n + 4\tau}\, d\mathcal{H}_U^n \Big)^{\frac{1}{2}} \leq \nonumber \\ 
    &\qquad c_{13}\Big(\int_{M_{\frac{1}{5}}^\epsilon} \big(1 + |\nabla^2 U| + |\nabla U| + |U|\big)^{2n + 4\tau} (1+\mu) \, d\mathcal{H}^n \Big)^{\frac{1}{2}} \nonumber \\
    &\leq c_{14}\Bigg(1 +  \Big(\int_{M_{\frac{1}{5}}^\epsilon} |\nabla^2 U|^{2n + 4\tau} \, d\mathcal{H}^n \Big)^{\frac{1}{2n + 4\tau}} + \Big(\int_{M_{\frac{1}{5}}^\epsilon} |U|^{2n + 4\tau} \, d\mathcal{H}^n \Big)^{\frac{1}{2n + 4\tau}} \nonumber \\
    &\qquad+ \Big(\int_{M_{\frac{1}{5}}^\epsilon} |\nabla U|^{2n + 4\tau} \, d\mathcal{H}^n \Big)^{\frac{1}{2n + 4\tau}}\Bigg)^{n + 2\tau} \nonumber \\
    &\leq c_{15}\Bigg(1 + \epsilon^{\eta_4} + \norm{\nabla^2 U}_{C^0(M_{\frac{1}{5}}^\epsilon)}^{1-\frac{1}{n+2\tau}}\Big(\int_{M_{\frac{1}{5}}^\epsilon} |\nabla^2 U|^{2} \, d\mathcal{H}^n \Big)^{\frac{1}{2n + 4\tau}}\Bigg)^{n + 2\tau} \nonumber \\
    &\leq c_{16}\big(1 + \epsilon^{ \eta_5}\big) \label{s3}
    \end{align}
    for some $\eta_5 > 0$ depending on $n$, where $\tau$ has been chosen small depending on $n$ also. Hence, the right-hand side is bounded by $2c_{17}$ for $\epsilon$ small. By covering $\partial M_U^\epsilon(\frac{1}{5})$ with geodesic balls having radius at most $10^{-1}$ and applying Corollary \ref{sest1} to $V$ while using that $V$ solves the eigenvalue equation, one finds 
   \begin{equation}
       \norm{V}_{C^0(\partial M_U^\epsilon(\frac{1}{5}))} \leq c_{18}. \label{bddry}
   \end{equation}
   Combining \eqref{s0}-\eqref{bddry} gives 
    \begin{equation}
        \norm{V}_{C^0(M_U^\epsilon(\frac{1}{5}))} \leq c_{19}. \label{deltasch}
    \end{equation}
    We may now combine \eqref{stable7}  and \eqref{deltasch} and use that $U \in \mathscr{K}$ to get
    \begin{multline*}
        \Big| \int_{M^\epsilon} \langle \tilde{A}_U(V), V \rangle - \langle \tilde{A}(\tilde{V}), \tilde{V} \rangle \, d\mathcal{H}^n \Big| \leq \\ 
        c_{20}\int
        _{M_{\frac{1}{2}}^\epsilon} (|\nabla^2 U|^2 + |\nabla^2 U| + |\nabla U| + |U|) \, d\mathcal{H}^n + c_{21}\epsilon^{\eta_6} \int_{M^\epsilon \setminus M_{\frac{1}{2}}^\epsilon} |V|^2 h^{-2}\, d\mathcal{H}^n \leq c_{22} \epsilon^{\eta_7}.
    \end{multline*}
   We have thereby obtained the desired estimate.

    The remaining terms on the right-hand side of \eqref{stable8} are easier. Integrating by parts in the expression 
    \begin{equation*}
        \int_{M_U^\epsilon(\delta)} |\nabla_U V|^2 \, d\mathcal{H}_U^n
    \end{equation*}
    and using that $V$ solves the eigenvalue equation, \eqref{deltasch}, $|\lambda_\delta| \leq c^\prime$, and $U \in \mathscr{K}$, we deduce
    \begin{equation}
        \int_{M_U^\epsilon} |\nabla_U V|^2 \, d\mathcal{H}_U^n \leq c_{23}. \label{stable11}
    \end{equation}
    As a consequence of \eqref{stable6} and \eqref{stable11}, we have 
    \begin{equation}
        \int_{M^\epsilon} |\nabla \tilde{V}|^2 \, d\mathcal{H}^n \leq c_{24}.\label{stable12}
    \end{equation}
    Next, note that by \eqref{stable10}, \eqref{CD}, and \eqref{CD2} we get
    \begin{multline*}
        \big||\nabla \tilde{V}|^2 - |\nabla_U V|^2 \big| \leq \\
        |g^{ij} - g^{ij}(U)||(v_{x^i} + B_{\beta i}^\alpha v^\beta)(v_{x^j} + B_{\beta j}^\alpha v^\beta)| + c_{25}\epsilon^{\eta_8}(|\nabla_U V|^2 + |V|^2).
    \end{multline*}
    Thus, a similar calculation as was done for the terms involving Simons' operator using \eqref{stable5}, \eqref{stable10}, \eqref{CD2}, \eqref{stable6}, and \eqref{stable11} shows
    $$
        \Big|\int_{M^\epsilon} \big(|\nabla \tilde{V}|^2 - |\nabla_U V|^2\big) \, d\mathcal{H}^n\Big| \leq c_{26} \epsilon^{\eta_9}.
    $$
    On the other hand, \eqref{mu}, \eqref{stable11}, and \eqref{stable12} give
    $$
        \Big|\int_{M^\epsilon}|\nabla \tilde{V}|^2 \mu \, d\mathcal{H}^n \Big| + \Big| \int_{M^\epsilon} |\nabla_U V|^2 \mu \, d\mathcal{H}^n \Big| \leq c_{27} \epsilon^{\eta_{10}}.
    $$
   Hence, $M_U^\epsilon$ is strictly stable for $\epsilon$ small enough.
\end{proof}
This completes the generalization of \cite{NS2} to dimensions $n \geq 3$ and arbitrary codimension.

\section{The Graphical Case}
Let $C_1, \ldots, C_N$ be $n$-dimensional graphical strictly stable minimal cones in $\mathbb{R}^{n+m+1}$ with vertices $p_1,\ldots, p_N$ and smooth links $\Sigma_1, \ldots, \Sigma_N$ in $\mathbb{S}^{n+m}(p_i)$, $n \geq 4$. In this case, the $C_i$ can be represented as the graphs of Lipschitz functions $v_i: \Omega_i \subset \mathbb{R}^n \rightarrow \mathbb{R}^{m+1}$ that are stationary solutions to the minimal surface system \eqref{MSS} which are smooth away from $\overline{p}_i \in \Omega_i$, where $\overline{p}_i$ is such that $p_i = (\overline{p}_i,v_i(\overline{p}_i))$. The following proposition completes the proof of Theorem \ref{mainthm}.
\begin{prop}\label{graphical}
    Suppose $M^\epsilon$ has been constructed from the $C_i$ so that it is graphical in $\mathbb{R}^{n+m+1}$ and satisfies the hypotheses of Theorem \ref{mainthm}. Then there is a universal constant $\epsilon_0$ such that $\epsilon < \epsilon_0$ implies the perturbed submanifold $M_U^\epsilon$ guaranteed by Theorem \ref{mainthm} remains graphical.
\end{prop}
\begin{proof}
    Apply Theorem \ref{mainthm} to obtain a solution $U$ to \eqref{dp} for small enough $\epsilon < \epsilon_0$ with corresponding minimal submanifold $M_U^\epsilon$. Choosing $\epsilon_0$ smaller if necessary depending on $\max_{i = 1,\ldots,N} \norm{Dv_i}_{C^0(\Omega)}$, we can ensure that $M_U^\epsilon$ is graphical since $\norm{U}_{C_{\nu}^1}$ can be made as small as we like.
\end{proof}
Using Proposition \ref{graphical}, we can now prove Theorem \ref{gsolution}.

\subsection{Graphical Bridges}
In order to apply Proposition \ref{graphical}, we need to show that for each $n \geq 4$ the approximate solution $M^\epsilon$ can be constructed so that it is graphical and satisfies the hypotheses of Theorem \ref{mainthm} when the $C_i$ are assumed graphical. Applying this to the Lawson-Osserman cone will prove Theorem \ref{gsolution}. We split the construction into two cases: $n \geq 6$ and $n = 4,5$. When $n = 4,5$, we will need to improve the bridge region of $M^\epsilon$ constructed when $n \geq 6$ so that \eqref{Hsmall1} is satisfied. We only consider the case $N = 2$, since the case of general $N$ follows by applying the case $N = 2$ inductively.

\subsubsection{The Case $n \geq 6$}
Assume $n \geq 6$ and that each of the $C_i$ ($i = 1,2$) are graphs over $\mathbb{R}^n$. Choose $\delta > 0$ small so that the extended cone 
$$
    C_{i,\delta}^+ := \{ (1-t)p_i + t\theta : \theta \in \Sigma, \text{ } t \in (0, 1 + \delta]\}
$$
is graphical over $\mathbb{R}^n$ for each $i$. Fix $q_i \in \Sigma_i$ and let $\gamma: [0,\ell_0] \rightarrow \mathbb{R}^{n+m+1}$ be any smooth unit speed curve satisfying:
\begin{itemize}
    \item[(i)] $\gamma(0) = q_1$ and $\gamma(\ell_0) = q_2$;
    \item[(ii)] $\gamma$ only intersects the $C_i$ at its endpoints;
    \item[(iii)] $\gamma$ agrees with the radial paths passing through $p_i$ and $q_i$ in $C_{i,\delta}^+$ for each $i = 1,2$;
    \item[(iv)] $\gamma([0,\ell_0])$ is graphical over $\mathbb{R}^n$;
    \item[(v)] $\gamma^\prime(t)$ is not vertical with respect to $\mathbb{R}^n$ for any $t \in [0,\ell_0]$.
\end{itemize}
Let $\mathcal{N}$ be any tubular neighborhood of $\gamma$ and choose smooth vector fields $\{\mu_i(t)\}_{i = 1}^{n-1}$ along $\gamma$ such that $\{\mu_i(t), \gamma^{\prime}(t)\}_{i = 1}^{n-1}$ is an orthonormal set of vector fields for each $t \in [0,\ell_0]$ that spans $T_{\gamma(t)}C_{i,\delta}^+$ for each $t \in [0,\delta] \cap [\ell_0 - \delta, \ell_0]$. Define $\psi: B_{\epsilon}^{n-1}(0) \times [0,\ell_0] \rightarrow \mathbb{R}^{n+m+1}$ by \eqref{strip1}, where $\epsilon < r_0$ and $r_0$ is chosen small enough that the image of $\psi$, we denote $\tilde{\Gamma}(\epsilon)$, lies inside $\mathcal{N}$. Since $C_{i,\delta}^+$ is a graph for each $i = 1,2$ and $\gamma$ is a graph, we can choose the $\mu_i$ such that the projection map $\Pi_n: T_{\gamma(x^n)}\tilde{\Gamma}(\epsilon) \rightarrow \mathbb{R}^n$ is non-singular for each $x^n \in [0,\ell_0]$. By the inverse function theorem, $\tilde{\Gamma}(\epsilon)$ is graphical over $\mathbb{R}^n$ in a small neighborhood of $\psi(0,\ldots, 0,x^n)$ for each $x^n \in [0,\ell_0]$. Thus, a covering argument shows that we can choose $r_0$ so small that $\tilde{\Gamma}(\epsilon)$ is graphical over $\mathbb{R}^n$ for each $\epsilon < r_0$. 

We need to smooth the ends of $\tilde{\Gamma}(\epsilon)$ onto the extended cones $C_{i,\delta}^+$. This can be done by flattening the $C_{i,\delta}^+$ onto the $n$-planes agreeing with $T_{\gamma(t)} C_{i,\delta}^+$ for each $t \in [1, 1+\delta]$ when $i = 1$, and each $t \in [\ell_0 - \delta, \ell_0]$ when $i = 2$ (e.g. see Sec. 8.1.2). After doing so, we obtain a new bridge $\Gamma(\epsilon)$ consisting of the part of $\tilde{\Gamma}(\epsilon)$ that does not intersect the $C_{i,\delta}^+$ and the part of the $C_{i,\delta}^+$ that has been flattened onto the tangent $n$-planes. It is clear from the computations in Sec. 8.1.2 that $\Gamma(\epsilon)$ will remain graphical provided the $C_{i,\delta}^+$ are graphs. Thus, $M^\epsilon := C_1 \cup C_2 \cup \Gamma(\epsilon)$ is a graphical approximate solution satisfying \eqref{Hsmall} and the hypotheses of Theorem \ref{mainthm} up to a diffeomorphism of $\Gamma(\epsilon)$.

\subsubsection{The Cases $n=4,5$}
The construction when $n = 4,5$ is nearly the same as the case $n \geq 6$. However, in this case \eqref{strip1} is replaced with the perturbed bridge defined by \eqref{psiparam}. In both cases, the mean curvature of $\tilde{\Gamma}(\epsilon)$ (i.e. the image of \eqref{psiparam}) along $\gamma$ is zero due to \eqref{vmc2}. We then follow the smoothing process in Sec. 8.1.2 as above to obtain a graphical bridge $\Gamma(\epsilon)$ and a corresponding graphical approximate solution $M^\epsilon$ satisfying \eqref{Hsmall1} when $n = 4,5$. Since the Lawson-Osserman cone is a graphical strictly stable four dimensional minimal cone in $\mathbb{R}^7$ with an isolated singularity, we can apply Proposition \ref{graphical} to $N$ copies of the Lawson-Osserman cone to conclude Theorem \ref{gsolution}. 

\vspace{-6.1pt}

\section{Remarks}
We first note that the results in \cite{DL} and references therein provide ample examples of strictly stable high codimension cones to apply the present construction to. We further remark that the requirement that $\Psi_\lambda \in C_c^\infty(M^\epsilon)$ (see Theorem \ref{mainthm} and Proposition \ref{psilambda}) can be loosened to require only that $\Psi_\lambda \in C_c^{3}(M^\epsilon)$ and that $\partial C_i$ is $C^4$ for each $i = 1,\ldots, N$ (e.g. see \cite{NS2}). In this case, $U$ is smooth in the interior of $M^\epsilon$ and $C^{3}$ up to the boundary, and the same holds for $M_U^\epsilon$. However, we considered $\Psi_\lambda \in C_c^\infty(M^\epsilon)$ for simplicity of presentation. Finally, we note on each $\Sigma_i$ one can specify any finite number of components in the boundary data $\Psi_\lambda$ in the $L^2(\Sigma_i)$ expansion, taking $\epsilon$ smaller as necessary (see Sec. 5 in \cite{NS2}). 

One limitation of Smale's method is that the bridge constructions force the domain of the defining function for the minimal graph to be molecule-shaped. A natural question is whether a similar gluing technique could be used to produce a Lipschitz stationary solution to the minimal surface system with multiple isolated singularities defined on a convex domain, such as a ball. For example, one could imagine producing an approximate solution by positioning the $C_i$ so that their vertices lie on the $\mathbb{R}^n$ subspace of $\mathbb{R}^{n+m+1}$, they are graphs over $\mathbb{R}^n$, and so that the $C_i$ lie entirely over an $n$-ball $B_r^n(0) \subset \mathbb{R}^n$. One could then flatten small extensions of the $C_i$ onto the portion of a graphical $n$-plane lying at height $1 + \epsilon$ above $B_r^n(0)$ for sufficiently small $\epsilon$ and search for a small perturbation of the approximate solution supported in a neighborhood of each of the defining cones to produce a Lipschitz minimal graph with defining function $v : B_r^n(0) \rightarrow \mathbb{R}^{m+1}$. It is reasonable that this type of construction could even be used to produce an entire Lipschitz minimal graph with multiple isolated singularities. Still, one needs to be careful about what condition replaces the boundary data since generic perturbations of $\partial C$ for a minimal cone $C$ with an isolated singularity can perturb away the singularity (\cite{McI}).

Strict stability is important in two key ways: 
\begin{enumerate}
    \item In \cite{CHS, Oo}, the perturbation methods apply to unstable cones. However, the strict stability condition in \cite{NS2} and in this paper helps counteract against boundary effects stemming from the bridges.
    \item It was crucial in deriving the a priori $L^2$ estimates and subsequent Schauder estimates used to apply the Schauder fixed point theorem and to ensure the singularities were preserved after perturbation.
\end{enumerate}
An interesting question considered by Smale is whether the strict stability condition can be removed. Since
$$
    \spec(L) = \Big[\frac{(n-2)^2}{4} + \mu_1, \infty\Big)
$$
on any minimal cone $C$ by separation of variables (e.g. see Sec. 5 in \cite{NS2}), where $\mu_1$ is the smallest eigenvalue of $L$ restricted to the link, if $C$ is not strictly stable then $0 \in \spec(L)$. Hence, there is a Jacobi field on $C$ vanishing on $\partial C$ and at the origin. Since the bridge principle may fail when Jacobi fields exist, even in the smooth case (see Sec. 2 in \cite{W2}), strict stability would likely need to be replaced by other strong analytic conditions; for example, a restrictive class of boundary data.

In the case of a smooth compact embedded submanifold (with or without boundary), there is a simple relationship between strict stability and area-minimization: For any such strictly stable minimal submanifold $M$, there is an open set $U$ containing $M$ such that $M$ has least area among all surfaces (i.e. integral currents) homologous to $M$ in $U$ (\cite{W3}). In the case that each connected component of $M$ has non-empty boundary, it is uniquely area-minimizing with respect to its boundary. When singularities are introduced, the relationship is not clear. For example, the cone over $\mathbb{S}^1\big(\sqrt{\frac{1}{6}}\big) \times \mathbb{S}^5\big(\sqrt{\frac{5}{6}}\big)$ is strictly stable (\cite{DL, Li, NS2}) but not area-minimizing in any open set (\cite{La, Se1, Se}). On the other hand, if $f(z)$ is a homogeneous complex polynomial of degree $n$ on $\mathbb{C}^{n+1}$, then $f^{-1}(0)$ is a complex calibrated cone with an isolated singularity at the origin provided the complex exterior derivative $\partial f$ vanishes only at zero. Using a unique continuation-type argument along with Almgren's regularity theorem (\cite{Al1, Al2}), it can be shown that all area-minimizing cones are uniquely area-minimizing. Hence, $f^{-1}(0)$ is a uniquely area-minimizing minimal cone with an isolated singularity that is not strictly stable (\cite{DL}). Combining these observations with the discussions in the introduction and preceding paragraph, we see that the strict stability of the Lawson-Osserman cone is essential to the present paper. 

Another problem posed by Smale is determining whether any of the examples in Theorem \ref{mainthm} or in Theorem 2.1 in \cite{NS2} are area-minimizing. In the codimension one case, Smale suggested this is plausible provided the cones $C_i$ are strictly area minimizing (see \cite{HS} for the definition) and are positioned such that $\cup C_i$ is uniquely area-minimizing with respect to the boundary $\cup \Sigma_i$. A different approach is to show that one can glue calibrated cones in a way that preserves the calibrated condition. Since there is a well-developed deformation theory for coassociative and special Lagrangian calibrated submanifolds with conical singularities (e.g. \cite{J, Lo, Mc}) and many strictly stable examples, it makes sense to consider these cases.  The primary challenge one expects to face is choosing the boundary data and constructing the bridges so that the calibrated condition and singularities are simultaneously maintained. 

In \cite{Sim}, Simon showed that, with respect to a smooth ambient metric that is close to the Euclidean metric, any closed subset $K \subset \{0\} \times \mathbb{R}^m$ in $\mathbb{R}^{n+m+1}$ can be realized as the singular set of a strictly stable $(n+m)$-dimensional minimal hypersurface without boundary (see  \cite{NS3, NS4, NS5} for more constructions). It would therefore be interesting to study whether gluing methods can be used to join minimal submanifolds, including complete submanifolds (i.e. without boundary), with more general singular sets in a way that preserves their singularities. Since generic perturbations of the boundary can even perturb away singularities of minimal hypersurfaces with higher dimensional singular sets (\cite{MrS}), it is again expected that the boundary data (or an appropriate replacement) will play a fundamental role. 

One could begin by investigating whether singularity preserving bridge principles can be developed for more general minimal submanifolds with isolated singularities. Nevertheless, this also seems difficult since the rate of convergence of the submanifold to its tangent cones at its singularities may be slow (\cite{AL}), and Smale's method seems to require at least quadratic decay. It is therefore unclear how the boundary data for a general minimal submanfold with isolated singularities would interact with the asymptotic behavior of solutions to \eqref{DPnu} since we are not guaranteed a Fourier expansion and there may be no boundary data available at all. 

\section{Appendix}
\subsection{Constructing Approximate Solutions}
We construct approximate solutions $M^\epsilon$ satisfying \eqref{Hsmall} and \eqref{Hsmall1}. The case $n \geq 6$ is due to Smale in \cite{NS1}, while the cases $n = 3,4,5$ are original. 

\subsubsection{The Case $n \geq 6$}
Fix $n \geq 6$ and let $C_1$ and $C_2$ be $n$-dimensional strictly stable cones in $\mathbb{R}^{n+m+1}$ with smooth links $\Sigma_1$ and $\Sigma_2$. In this case, the $\epsilon$-bridge $\Gamma(\epsilon)$ is constructed as a ruled surface which is the image of a smooth sub-strip $S_\epsilon \subset B_{r_0}^{n-1}(0) \times [0,1] \subset \mathbb{R}^n$ under a smooth embedding $\psi: S_\epsilon \rightarrow \mathbb{R}^{n+m+1}$ (\cite{NS1}, p. 507). 

Fix $q_i \in \Sigma_i$ ($i = 1,2$). One begins by joining the cones $C_i$ by a smooth embedded unit speed curve $\gamma: [0, \ell_0] \rightarrow \mathbb{R}^{n+m+1}$ such that:
\begin{itemize}
    \item[(i)] $\gamma(0) = q_1$ and $\gamma(\ell_0) = q_2$;
    \item[(ii)] $\gamma$ only intersects the $C_i$ at its endpoints;
    \item[(iii)] $\gamma$ is smoothly tangent to $C_i$ at $q_i$ for each $i = 1,2$.
\end{itemize}
Let $\mathcal{N}$ be any tubular neighborhood of $\gamma$ and let $\{\mu_i(t)\}_{i = 1}^{n-1}$ be an orthonormal set of vector fields along $\gamma$ that are orthogonal to $\gamma^\prime(t)$ for each $t \in [0,\ell_0]$. Assume further that $\mu_1(t), \ldots, \mu_{n-1}(t), \gamma^\prime(t)$ form an orthonormal basis for $T_{q_1}C_1$ when $t = 0$, and form an orthonormal basis for $T_{q_2}C_2$ when $t = \ell_0$. Define $\psi: B_{r_0}^{n-1} \times [0,\ell_0] \rightarrow \mathbb{R}^{n+m+1}$ by
\begin{equation}
    \psi(x^1,\ldots, x^n) := \gamma(x^n) + \sum_{i = 1}^{n-1} x^i \mu_i(x^n) \label{strip1}
\end{equation}
and note that, for $r_0$ sufficiently small, $\psi$ is a smooth embedding and its image lies inside $\mathcal{N}$. The image of $\psi$ is tangent to $C_1 \cup C_2$ at the $q_i$ ($i = 1,2$), but is not smoothly attached to $C_1 \cup C_2$ along its ends. To handle this, one smoothly extends the cones $C_i$ near $q_i$ so that they contain a small piece of $\gamma$, and pushes the image of $\psi$ smoothly onto the extended pieces. To smooth the boundary of the strip onto $\Sigma_i$, one then restricts the domain of $\psi$ to a sub-strip $S_\epsilon \subset B_{\epsilon}^{n-1} \times [0,\ell_0]$, $\epsilon < r_0$ (\cite{NS1}, pp. 508-509). The $\epsilon$-bridge is then defined to be $\Gamma(\epsilon) := \psi(S_\epsilon)$. The resulting $n$-dimensional submanifold $M^\epsilon := C_1 \cup C_2 \cup \Gamma(\epsilon)$ is an embedded submanifold of $\mathbb{R}^{n+m+1}$ with boundary that is smooth (including the boundary) away from two isolated singularities at the vertices of the cones $C_i$. We call the submanifolds $M^\epsilon$ \emph{approximate solutions}. We further note that the approximate solutions are constructed so that $M^{\epsilon_1} \subset M^{\epsilon_2} \subset M^{r_0}$ for $0 < \epsilon_1 < \epsilon_2 < r_0$.

Since $\psi$ parameterizes $\Gamma(\epsilon)$ and $\norm{\psi}_{C^{4}}$ is bounded independent of $\epsilon$, for each $\delta > 0$ there is a finite collection of smooth local parameterizations of $M_\delta^\epsilon$, we denote $\mathscr{U} := \{(\Omega_l, \psi_l)\}_{l = 1}^J$, with $\psi_1 = \psi$, $\Omega_1 = S_\epsilon$,
\begin{equation}
    \psi_l : \Omega_l \rightarrow M_\delta^\epsilon \text{ for } l = 2,\ldots, J \text{ covering } C_{1,\delta} \cup C_{2, \delta}, \label{param1}
\end{equation}
where $C_{i,\delta}$ are the truncated cones \eqref{tcone}. Furthermore, 
\begin{equation}
\norm{\psi_l}_{C^{k}} \leq c_0\delta^{1-k} \text{ and } c_0^{-1} \leq \norm{D\psi_l}_{C^0} \text{ for each } l = 2, \ldots, J, \text{ } k = 1,2,3,4 \label{psiest}
\end{equation}
for some constant $c_0$ independent of $\epsilon$ and $\delta$. When $l = 1$, the bounds in \eqref{psiest} are independent of $\delta$. Since 
\begin{equation}
    c_0^{-1} \leq \norm{D\psi_1}_{C^0} \leq c_0, \label{param2}
\end{equation}
lengths and volumes are equivalent in $\Omega_1$ and $\Gamma(\epsilon)$. Furthermore, if $(g_{ij})$ is the metric tensor on $M_\delta^\epsilon$ (i.e. $g_{ij} = \psi_{x^i} \cdot \psi_{x^j}$ in any coordinate system $\psi$) in one of the coordinate systems $(\Omega_l, \psi_l)$ ($l = 2,\ldots, J$) and $(g^{ij}) = (g_{ij})^{-1}$, then 
\begin{equation}
    \begin{cases}
        \sum_{i,j = 1}^n \big( \norm{g_{ij}}_{C^{k}} + \norm{g^{ij}}_{C^{k}}\big) \leq c_0\delta^{-k}, \text{ } k = 0,1,2,3 \\
        \lambda_0^{-1} I < (g_{ij}) < \lambda_0I  \\
        \lambda_0^{-1} I < (g^{ij}) < \lambda_0I,
    \end{cases} \label{gest}
\end{equation}
where $I$ is the identity matrix and $\lambda_0$ is a positive constant independent of $\epsilon$ and $\delta$. When $l = 1$, the first bound in \eqref{gest} is independent of $\delta$. Furthermore, for each $l = 2,\ldots, J$ we can find an orthonormal frame $n_1,\ldots, n_{m+1}$ for $NM_\delta^\epsilon \lvert_{\Omega_l}$ such that 
\begin{equation}
    \begin{cases}
        \sum_{\alpha = 1}^{m+1} \norm{n_\alpha}_{C^{k}} \leq c_0\delta^{-k}, \text{ } k = 0,1,2,3.
    \end{cases} \label{nest}
\end{equation}
When $l = 1$, the bound in \eqref{nest} is again independent of $\delta$.

The resulting submanifold $M^\epsilon$ is an approximate solution with mean curvature satisfying the $L^p$ estimate \eqref{Hsmall}. Indeed, if $H_0$ is the mean curvature of $M^\epsilon$, then $H_0$ is bounded independent of $\epsilon$ and supported in $\Gamma(\epsilon)$ since $C_1$ and $C_2$ are minimal. Thus, when $n \geq 6$
\begin{equation*}
    \Big(\int_{M^\epsilon} |H_0|^p\,   \Big)^{\frac{1}{p}} \leq c \mathcal{H}^n(\Gamma(\epsilon))^{\frac{1}{p}} \leq c\epsilon^{\frac{n-1}{p}} \text{ for } p \geq 1
\end{equation*}
and some constant $c$ independent of $\epsilon$. The mean curvature estimate above plays a critical role in deriving the estimates necessary to solve the fixed point problem. This is related to the fact that the mean curvature estimate is better in higher dimensions. It turns out that we need to improve the mean curvature estimate when $n = 3,4,5$. This is done by constructing a strip whose mean curvature is pointwise small depending on $\epsilon$. The new strip will be a small perturbation of \eqref{strip1}.

\subsubsection{The Cases $n = 3,4,5$}
Let $n = 3,4,5$ and let $C_1$, $C_2$ be strictly stable cones in $\mathbb{R}^{n+m+1}$ with smooth links $\Sigma_1$ and $\Sigma_2$. Fix $q_1 \in \Sigma_1$ and choose coordinates $(x^\prime, x^n, z)$ for $\mathbb{R}^{n+m+1}$, where $x^\prime \in \mathbb{R}^{n-1}$ and $z \in \mathbb{R}^{m+1}$, so that:
\begin{itemize}
    \item[(a)] The vertex $p_1$ of $C_1$ coincides with the origin;
    \item[(b)] The line segment joining $p_1$ and $q_1$ lies along the positive $x^n$-axis, $x^n \in [0,1]$.
\end{itemize}
Set $x := (x^\prime, x^n)$. By construction, the $x$-subspace coincides with $T_{(0,x^n,0)} C_1$ for each $x^n \in \mathbb{R}$, where we have extended $C_1$ to an infinite cone by dilation. Furthermore, there is a small neighborhood $N(\epsilon)$ of $q_1$ and a universal constant $r_0 > 0$ such that, if $\epsilon < r_0$, we can represent $N(\epsilon)$ as the graph of a function $v: B_\epsilon^{n-1}(0) \rightarrow \mathbb{R}^{m+1}$, where $v := v(x^\prime)$, $v(0) = (0, \ldots,0) \in \mathbb{R}^{m+1}$, and $Dv(0) = 0 \in \mathbb{R}^{m+1 \times n-1}$. In addition, there is a universal constant $c > 0$ such that the diameter of $N(\epsilon)$ in $\Sigma_1$ is less than $c\epsilon$, and there is a $\beta \in (0,1)$ such that we can represent a small truncated wedge $\mathcal{W}$ of $C_1$ as the graph of a smooth function $w: B_\epsilon^{n-1}(0) \times [\beta, \infty) \rightarrow \mathbb{R}^{m+1}$ by setting $w(x^\prime,x^n) = x^n v\big(\frac{x^\prime}{x^n}\big)$. 

By the definition of $v$ and the product rule applied in the $x^n$ variable, we see that $w(0,x^n) = (0, \ldots, 0) \in \mathbb{R}^{m+1}$ and $Dw(0,x^n) = 0 \in \mathbb{R}^{m+1 \times n}$ for each $x^n \in [\beta, \infty)$. Parameterize $\mathcal{W}$ by $\psi_0(x) := (x, w(x))$. In these coordinates, $\psi_0(0,x^n)$ lies along the $x^n$-axis in $\mathbb{R}^{n+m+1}$ for each $x^n \in [\beta,\infty)$, where $\psi_0(0,1) = q_1 = (0,1,0) \in \mathbb{R}^{n+m+1}$ and the $1$ occupies the $x^n$ slot in both cases. The metric for $\mathcal{W}$ in these coordinates and its inverse are given by 
\begin{equation}
    g := I + Dw^TDw \text{ and } g^{-1} := I + \sum_{k = 1}^\infty (-1)^k Dw^TDw \label{W0metric}
\end{equation}
when $r_0$ is small enough. This can be done since $Dw \sim O(|x^\prime|)$ by a Taylor expansion about the $x^n$-axis. 

Since $\mathcal{W}$ is minimal, \eqref{MSS2} and \eqref{W0metric} together show
\begin{equation*}
    0 = g^{ij}(x)w_{x^ix^j}^\alpha(x) = \Delta_{\mathbb{R}^n} w^\alpha(x) + O(|Dw(x)|^2) \text{ for } \alpha = 1,\ldots, m+1.
\end{equation*}
Here, $\Delta_{\mathbb{R}^n} w^\alpha$ represents the Euclidean Laplacian of the real-valued function $w^\alpha$. Since $Dw \sim O(|x^\prime|)$, we obtain
$$
    \Delta_{\mathbb{R}^n} w^\alpha(x) \sim O(|x^\prime|^2) \text{ for each } \alpha = 1,\ldots, m+1. 
$$
Let $\varphi: \mathbb{R} \rightarrow \mathbb{R}$ be a smooth function satisfying 
\begin{equation*}
    \varphi(t) = \begin{cases}
                1, \text{ } t \in (-\infty, 1] \\
                0, \text{ } t \in [2, \infty)
            \end{cases}
\end{equation*}
and for $\epsilon < r_0$ define $\tilde{w} : B_\epsilon^{n-1}(0) \times [1,2] \rightarrow \mathbb{R}^{m+1}$ by $\tilde{w}(x^\prime, x^n) := \varphi(x^n)w(x^\prime, x^n)$. Then the graph of $\tilde{w}$ flattens $\mathcal{W}$ onto the $x$-subspace for $x^n \in [1,2]$. Let $\tilde{P}_{1,\epsilon}$ be the graph of $\tilde{w}$ and parameterize $\tilde{P}_{1,\epsilon}$ by $\tilde{\psi}(x) := (x,\tilde{w}(x))$ for $x \in B_\epsilon^{n-1}(0) \times [1,2]$. Observe that 
\begin{equation*}
    w_{x^n}^\alpha(x^\prime,x^n) = v^\alpha \Big(\frac{x^\prime}{x^n} \Big) - \frac{x^\prime}{x^n} \cdot D v^\alpha\Big(\frac{x^\prime}{x^n}\Big),
\end{equation*}
so the definition of $v$ and a Taylor expansion about the $x^n$-axis together imply $w_{x^n}^\alpha \sim O(|x^\prime|^2)$. Since $D_{x^\prime}w \sim O(|x^\prime|)$ about the $x^n$-axis, we conclude that $Dw \sim O(|x^\prime|)$. Using this and the definition of $\tilde{w}$, we deduce that $D\tilde{w} \sim O(|x^\prime|)$. Thus, the metric and its inverse in these coordinates are given by
\begin{equation}
    \tilde{g} := I + D\tilde{w}^TD\tilde{w} \text{ and } \tilde{g}^{-1} := I + \sum_{k = 1}^\infty (-1)^k D\tilde{w}^TD\tilde{w}. \label{tildeWmetric}
\end{equation}
The mean curvature of $\tilde{P}_{1,\epsilon}$ is 
\begin{equation}
    \tilde{H}(x) = (\tilde{g}^{ij}(x)\tilde{\psi}_{x^ix^j}(x))^\perp. \label{tildeH}
\end{equation}
We need to estimate the right-hand side of \eqref{tildeH} in terms of $|x^\prime|$. 

Notice that $\tilde{\psi}_{x^i x^j}^k = 0$ for each $k = 1,\ldots, n$, so it suffices to estimate $\tilde{g}^{ij}\tilde{w}_{x^ix^j}^\alpha$ for each $\alpha = 1,\ldots, m+1$. Using the second expression in \eqref{tildeWmetric}, we find
$$
    \tilde{g}^{ij}(x)\tilde{w}_{x^ix^j}^\alpha(x) = \Delta_{\mathbb{R}^n} \tilde{w}^\alpha(x) + O(|D\tilde{w}(x)|^2).
$$ 
Direct computation shows
\begin{equation*}
    \Delta_{\mathbb{R}^n} \tilde{w}^\alpha(x) = \varphi(x^n) \Delta_{\mathbb{R}^n} w^\alpha(x) + 2\varphi^\prime(x^n) w_{x^n}^\alpha(x) + \varphi^{\prime \prime}(x^n) w^\alpha(x).
\end{equation*}
Each term above is $O(|x^\prime|^2)$ due to prior estimates. Furthermore, $D\tilde{w} \sim O(|x^\prime|)$. It follows that $\tilde{g}^{ij}\tilde{w}_{x^ix^j}^\alpha \sim O(|x^\prime|^2)$ for each $\alpha = 1,\ldots, m+1$. Thus, \eqref{tildeH} combined with the above estimates shows
\begin{equation}
    \tilde{H}(x) \sim O(|x^\prime|^2) \text{ on } B_\epsilon^{n-1}(0) \times [1,2]
\end{equation}
provided $\epsilon < r_0$. We can now smooth the boundary of $C_1 \cup \tilde{P}_{1,\epsilon}$ where $\tilde{P}_{1,\epsilon}$ meets $C_1$ by restricting the domain of $\tilde{\psi}$ to a sub-strip $S_{1,\epsilon} \subset B_\epsilon^{n-1}(0) \times [1,2]$ satisfying 
$$
    S_\epsilon \cap B_{\epsilon}^{n-1}(0) \times \Big[\frac{3}{2}, 2\Big] = B_{\frac{\epsilon}{2}}^{n-1}(0) \times \Big[\frac{3}{2}, 2\Big].
$$ 
Set $P_{1,\epsilon} := \tilde{\psi}(S_{1,\epsilon})$ (i.e. the patching region for $C_1$) and $\tilde{C}_1 := C_1 \cup P_{1,\epsilon}$. Let $p_2$ be the vertex of $C_2$ and fix $q_2 \in \Sigma_2$. Repeat the process above on $C_2$ with $p_2$ and $q_2$ replacing $p_1$ and $q_1$, respectively, to obtain an extended submanifold $\tilde{C}_2$ and associated patching region $P_{2,\epsilon}$. Going forward, we will write $P_\epsilon := P_{1,\epsilon} \cup P_{2,\epsilon}$ for any $\epsilon < r_0$.

Suppose $n = 3$ and choose coordinates $(x,z)$ for $\mathbb{R}^{4+m}$ such that $x \in \mathbb{R}^3$ and $z \in \mathbb{R}^{m+1}$. Suppose further that the $x$-subspace is tangent to $\tilde{C}_i$ at each point along the segment joining $p_i$ to $\tilde{q}_i$ and each $i = 1,2$. Orient $\tilde{C}_1$ and $\tilde{C}_2$ so that $p_1$ coincides with the origin, $p_2 = (0,0,5,0 \ldots,0)$, the segment joining $p_1$ to $p_2$ lies on the $x^3$-axis in $\mathbb{R}^{4+m}$, and the segment joining $\tilde{q}_1$ to $\tilde{q}_2$ is a subset of the segment joining $p_1$ to $p_2$. Set
$$
    \Gamma(\epsilon) := P_\epsilon \cup (B_{\frac{\epsilon}{2}}^2(0) \times [2, 3]),  
$$
where the second set in the union is embedded in the $x$-subspace in $\mathbb{R}^{4+m}$. Then $M^\epsilon := C_1 \cup C_2 \cup \Gamma(\epsilon)$ is an approximate solution with mean curvature $H_0$ satisfying
$$
    \Big(\int_{M^\epsilon} |H_0|^p \Big)^{\frac{1}{p}} = \Big(\int_{P_\epsilon} |H_0|^p \Big)^{\frac{1}{p}} \leq c_1 (\epsilon^{2p}\mathcal{H}^3(P_\epsilon))^{\frac{1}{p}} = c_2\epsilon^{\frac{2 + 2p}{p}} \text{ for } p \geq 1.
$$

When $n = 4,5$, we construct a strip $T_\epsilon$ that smoothly attaches $\tilde{C}_1$ to $\tilde{C}_2$ and has zero mean curvature along its center curve. 
Set $\tilde{q}_1 := \tilde{\psi}(0,2) \in \tilde{C}_1$ and define $\tilde{q}_2$ similarly for $\tilde{C_2}$. Let $\gamma: [0,\ell_0] \rightarrow \mathbb{R}^{n+ m + 1}$ be any smooth embedded unit-speed curve such that:
\begin{itemize}
    \item[(i)] $\gamma(0) = \tilde{q}_1$ and $\gamma(\ell_0) = \tilde{q}_2$;
    \item[(ii)] $\gamma$ smoothly extends the line segment joining $p_1$ and $\tilde{q}_1$ as well as the line segment joining $p_2$ and $\tilde{q}_2$.
\end{itemize}
Let $\mathcal{N}$ be any tubular neighborhood of $\gamma$ and let $\mu_1(t),\ldots, \mu_{n-1}(t)$ be smooth vector fields along $\gamma$ for $t \in [0,\ell_0]$ such that $\{\mu_i(t), \gamma^\prime(t)\}_{i = 1}^{n-1}$ is an orthonormal set for each $t$. Suppose further that the set spans $T_{\tilde{q}_1} \tilde{C}_1$ when $t = 0$ and spans $T_{\tilde{q}_2} \tilde{C}_2$ when $t = \ell_0$. For $\epsilon < r_0$, define $\psi: B_{\frac{\epsilon}{2}}^{n-1}(0) \times [0,\ell_0] \rightarrow \mathbb{R}^{n +m +1}$ by 
\begin{equation}
    \psi(x^1, \ldots, x^n) := \gamma(x^n) + \sum_{i = 1}^{n-1} x^i \mu_i(x^n) - \sum_{i = 1}^{n-1} \frac{(x^i)^2}{2(n-1)}\gamma^{\prime \prime}(x^n) \label{psiparam}
\end{equation}
and let $T_\epsilon$ be the image of $\psi$. Then $T_\epsilon$ lies entirely inside $\mathcal{N}$ for $r_0$ small enough. Notice that $\psi(0,\ldots, x^n) = \gamma(x^n)$ for each $x^n \in [0,\ell_0]$ and that $\psi_{x^i}(0,\ldots,x^n) = \mu_i(x^n)$ for each $i = 1, \ldots, n-1$, while $\psi_{x^n}(0,\ldots,x^n) = \gamma^\prime(x^n)$ for each $x^n \in [0,\ell_0]$. Since $\{\mu_i(x^n), \gamma^\prime(x^n)\}_{i=1}^{n-1}$ is an orthonormal basis for $T_{\gamma(x^n)}T_\epsilon$, the mean curvature of $T_\epsilon$ along $\gamma$ is
\begin{equation}
    \Big(\sum_{i = 1}^n \psi_{x^ix^i}(0,\ldots,x^n)\Big)^{\perp} =  (\gamma^{\prime \prime}(x^n) - \gamma^{\prime \prime}(x^n))^{\perp} = 0 \text{ for each } x^n \in [0,\ell_0]. \label{vmc2}
\end{equation}
Moreover, since $|\gamma^{\prime \prime}(0)| = |\gamma^{\prime \prime}(\ell_0)| = 0$, $T_\epsilon$ smoothly attaches to the $\tilde{C}_i$ along the ends of $P_\epsilon$.

Define $\Gamma(\epsilon) := P_\epsilon \cup T_\epsilon$. Then $M^\epsilon := C_1 \cup C_2 \cup \Gamma(\epsilon)$ is an approximate solution. Furthermore, a second-order Taylor expansion of the mean curvature $H_0$ on $P_\epsilon$ shows that $|H_0| \leq c_3 \epsilon^2$, while a first-order Taylor expansion on $T_\epsilon$ shows $|H_0| \leq c_4 \epsilon$. Therefore, 
\begin{equation*}
    \int_{M^\epsilon} |H_0|^p  = \int_{P_\epsilon}|H_0|^p + \int_{T_\epsilon} |H_0|^p \leq c_3\epsilon^{2p}\mathcal{H}^n(P_\epsilon) + c_4\epsilon^p\mathcal{H}^n(T_\epsilon)
\end{equation*}
implying
$$
    \Big(\int_{M^\epsilon}|H_0|^p\Big)^{\frac{1}{p}} \leq c_5 \epsilon^{\frac{n-1 + p}{p}} \text{ for each } p \geq 1 \text{ when } n =4,5.
$$

\begin{rmk}\label{nbdds}
Notice that the $L^p$ estimate for $H_0$ above is better than necessary when $n = 3,4,5$ (see \eqref{Hsmall1}). After examining the construction of $M^\epsilon$ for $n = 3,4,5$, we see that in each case there is a finite collection of smooth local parameterizations of $M_\delta^\epsilon$, $\mathscr{U} := \{(\Omega_l,\psi_l)\}_{l = 1}^J$, such that $\psi_1$ parameterizes $\Gamma(\epsilon)$ and \eqref{param1}-\eqref{gest} are satisfied. Furthermore, for each $l = 2, \ldots, J$ we can find an orthonormal frame for $NM_{\delta}^\epsilon \lvert_{\Omega_l}$ satisfying \eqref{nest}. On $NM_{\delta}^\epsilon \lvert_{\Omega_1}$, the bound in \eqref{nest} is independent of $\delta$.
\end{rmk}

\subsection{Schauder Estimates}
To prove Proposition \ref{sest4}, we need interior Schauder estimates in $B_1 := B_1^n(0)$ for elliptic systems in $\mathbb{R}^n$. Define the operator $L_0$ on $u := (u^1,\ldots, u^{m+1}) \in C^2(B_1; \mathbb{R}^{m+1})$ given component-wise by
\begin{equation}
    (L_0 u)^\alpha := a^{ij}u_{x^ix^j}^\alpha +  b_{\beta}^{i\alpha} u_{x^i}^\beta + c_{\beta}^\alpha u^\beta \text{ for each } \alpha = 1,2,\ldots,m+1, \label{genstab3}
\end{equation}
where we have used summation notation in the variables $i,j,\beta$. We assume the coefficients are uniformly bounded in $C^{0,\gamma}$:
\begin{align*}
    \sum_{i,j}\big(\norm{a^{ij}}_{C^0(B_1)} + |a^{ij}|_{\gamma, B_1}\big) &\leq c_0 \\
    \sum_{i,\alpha,\beta}\big(\norm{b_{\beta}^{i\alpha}}_{C^0(B_1)} + |b_{\beta}^{i\alpha}|_{\gamma, B_1}\big) &\leq c_0 \\
    \sum_{\alpha,\beta} \big(\norm{c_{\beta}^\alpha}_{C^0(B_1)} + |c_{\beta}^\alpha|_{\gamma, B_1}\big) &\leq c_0.
\end{align*}
In addition, we assume:
$$
    \lambda_0^{-1} I \leq (a^{ij}) \leq \lambda_0 I.
$$
The Schauder estimates we need for the system \eqref{genstab3} are as follows:

\begin{prop}\label{eschauder}
    If $u \in C^{2,\gamma}(B_1;\mathbb{R}^{m+1})$ satisfies $L_0 u = f$ in $B_1$ for some $f \in C^{0,\gamma}(B_1; \mathbb{R}^{m+1})$, then there is a constant $c$ depending on $n$, $m$, $\gamma$, and $c_0$ such that 
    $$
        \norm{u}_{C^{2,\gamma}(B_{\frac{1}{2}})} \leq c\big(\norm{u}_{C^0(B_1)} + \norm{f}_{C^{0,\gamma}(B_1)}\big).
    $$
\end{prop}
 A simple rescaling argument then gives the following useful corollary:
\begin{coro} \label{cor1}
     If $u \in C^{2,\gamma}(B_{2r};\mathbb{R}^{m+1})$ satisfies $L_0 u = f$ in $B_{2r}$ for some $f \in C^{0,\gamma}(B_{2r}; \mathbb{R}^{m+1})$, then there is a constant $c$ depending on $n$, $m$, $\gamma$, and $c_0$ such that 
    $$
        |D^2u|_{\gamma, B_r} \leq c\big(r^{-2-\gamma}\norm{u}_{C^0(B_{2r})} + r^{-\gamma}\norm{f}_{C^{0,\gamma}(B_{2r})}\big).
    $$
\end{coro}
Since the second-order term in \eqref{genstab3} does not have any coupling of the components of $u$, Proposition \ref{eschauder} and Corollary \ref{cor1} can be proved by the usual perturbation methods used in the proof of the Schauder estimates for uniformly elliptic scalar equations (i.e. comparing with the Laplacian). We leave the details to the reader (see also Sec. 3 in \cite{NS1}).

\end{document}